\documentclass[11pt,reqno, a4paper]{amsart} 

\newcommand\version{August 6, 2026}

\usepackage{epsfig}
\usepackage{graphicx}
\usepackage{enumerate}
\usepackage{color} 
\usepackage[toc,page]{appendix}

\usepackage{setspace}

\usepackage[dvipsnames]{xcolor}
\usepackage{bm,bbm}
\usepackage{cancel}
\usepackage{enumitem}

\usepackage{physics}

\usepackage{amsmath,euscript}
\usepackage{amsthm}
\usepackage{amssymb}
\usepackage{graphicx}
\usepackage{mathrsfs}
\usepackage{epsf}
\usepackage{xcolor}
\usepackage{psfrag}
\usepackage{enumerate}
\usepackage{amsfonts,amssymb,amsthm,amsmath} 
\usepackage[unicode]{hyperref}

\usepackage[numbers]{natbib}
\usepackage{enumitem} 
\usepackage{xspace}
\usepackage[margin=1.37in]{geometry}

\usepackage{changes}
\definechangesauthor[name={SB}, color=magenta]{sb}
\definechangesauthor[name={JF}, color=blue]{jf}
\definechangesauthor[name={ML}, color=red]{ml}
\definechangesauthor[name={MS}, color=green]{ms}

\usepackage{soul}

\newtheorem{thm}{Theorem}[section]
\newtheorem{lemma}[thm]{Lemma}
\newtheorem{prop}[thm]{Proposition}
\newtheorem{cor}[thm]{Corollary}

\theoremstyle{remark}

\theoremstyle{definition}

\numberwithin{equation}{section}

\usepackage{soul}
\setstcolor{red}

\newcommand{\nc}{\newcommand}

\nc{\la}{\label}
\nc{\ba}{\begin{array}}
\nc{\ea}{\end{array}}
\nc{\bs}{\begin{split}}
\nc{\es}{\end{split}}

\newcommand{\R}{\mathbb{R}}
\newcommand{\C}{\mathbb{C}}
\newcommand{\Z}{\mathbb{Z}}
\newcommand{\T}{\mathbb{T}}

\newcommand{\cD}{\mathcal{D}}
\newcommand{\cC}{\mathcal{C}}

\newcommand{\cF}{\mathcal{F}}
\newcommand{\cG}{\mathcal{G}}
\newcommand{\cH}{\mathcal{H}}

\newcommand{\cL}{\mathcal{L}}
\newcommand{\cN}{\mathcal{N}}       

\newcommand{\fh}{\mathfrak{h}}
\nc{\ran}{\rangle}
\nc{\lan}{\langle}

\newcommand{\one}{\mathbf{1}}

\newcommand{\Ran}{\operatorname{Ran}}

\renewcommand{\Re}{\mathrm{Re}} 
\renewcommand{\Im}{\mathrm{Im}} 

\nc{\bfone}{{\bf 1}}
\nc{\1}{{\bf 1}}

\usepackage{centernot}

\usepackage[normalem]{ulem}

\newcommand{\DETAILS}[1]{}

\nc{\den}{\text{den}}
\nc{\ex}{\text{xc}}

\usepackage{soul}

\usepackage{tikz}
\usetikzlibrary{automata,positioning,arrows}
\begin{document}
\title[From Bogoliubov--de Gennes to Ginzburg--Landau]{From Bogoliubov--de Gennes to Ginzburg--Landau: Critical Points Near $T_{\rm c}$ in the Non-Magnetic Case}

\author{Rupert L. Frank}
\address[Rupert L. Frank]{Mathe\-matisches Institut, Ludwig-Maximilians-Universit\"at M\"unchen, The\-resienstr.~39, 80333 M\"unchen, Germany, and Munich Center for Quantum Science and Technology, Schel\-ling\-str.~4, 80799 M\"unchen, Germany, and Mathematics 253-37, Caltech, Pasa\-de\-na, CA 91125, USA}
\email{r.frank@lmu.de}

\author{Christian Hainzl}
\address[Christian Hainzl]{Mathe\-matisches Institut, Ludwig-Maximilians-Universit\"at M\"unchen, The\-resienstr.~39, 80333 M\"unchen, Germany, and Munich Center for Quantum Science and Technology, Schel\-ling\-str.~4, 80799 M\"unchen, Germany}
\email{hainzl@math.lmu.de}

\author{Dong Hao Ou Yang}
\address[Dong Hao Ou Yang]{Mathe\-matisches Institut, Ludwig-Maximilians-Universit\"at M\"un\-chen, The\-resienstr.~39, 80333 M\"unchen, Germany}
\email{ouyang@math.lmu.de}

\date{\version}

\thanks{\copyright\, 2026 by the authors. This paper may be reproduced, in its entirety, for non-commercial purposes.\\
	Partial support through US National Science Foundation grant DMS-1954995 (R.L.F.), as well as through the Deutsche Forschungsgemeinschaft (DFG, German Research Foundation) through Germany’s Excellence Strategy EXC-2111-390814868 (R.L.F.) and through Project-ID 470903074 – TRR 352 (R.L.F., C.H., D.H.O.Y.) is acknowledged.}

\begin{abstract}
	We study the relation between the Bogoliubov--de Gennes (BdG) equation and the Ginzburg--Landau (GL) equation for a BCS model without external fields. While previous rigorous derivations of GL theory from BCS theory have focused on energies and minimizers, here we consider arbitrary critical points in the relevant energy regime. For temperatures close to the critical temperature, we prove that every sufficiently small solution of the BdG equation admits an asymptotic factorization into a microscopic Cooper-pair profile and a macroscopic order parameter. The latter satisfies the GL equation up to an error that vanishes in the scaling limit. We also prove the opposite direction, going from a solution of the GL equation to an approximate solution of the BdG equation. Our analysis relies on a Birman--Schwinger reformulation of the BdG equation, a Lyapunov--Schmidt type reduction, and semiclassical estimates at low regularity. 
\end{abstract}

\maketitle

\tableofcontents

\section{Introduction and main results}\label{sec:intro}

Many physical phenomena can be described both by microscopic and macroscopic models. A microscopic model takes the individual constituent particles into account, while a macroscopic model only considers their distribution. Usually, macroscopic models are much easier to deal with, both analytically and numerically, while microscopic models have the conceptual advantage of being closer to first principles. It is an important and relevant question to understand the relation between macroscopic and microscopic models and, more specifically, to identify parameter regimes in which a microscopic model is approximated by a macroscopic model and to estimate the accuracy of this approximation. This is often associated with a scaling limit.

One way of considering the micro-to-macro transition is on the level of the relevant energies and their minimizers. Another, more refined way is to show that any critical point of the microscopic model, at least in an energy range comparable to the minimal energy, is approximated by a critical point of the macroscopic model. This is precisely what we will do in the present paper in the context of superconductivity.

In 1950, Ginzburg and Landau \cite{LanLifs-81} introduced a model of superconductivity, describing its \textit{macroscopic} properties, which has been very successful and widely used in physics.  They arrived at their conclusion from a phenomenological perspective without consideration of a microscopic mechanism. Critical points of the corresponding energy function satisfy what is known as the \textit{Ginzburg--Landau (GL) equation}.

In 1957, Bardeen, Cooper and Schrieffer \cite{BCS-57} established the first \textit{microscopic} theory of superconductivity based on an effective interacting fermionic many-body Hamiltonian. In this framework, superconductivity arises from a pairing instability of the normal state: an effective attraction between fermionic quasiparticles in the vicinity of the Fermi surface makes the normal state thermodynamically unstable below a critical temperature, and the system develops a nonzero pairing amplitude. The resulting Cooper pairs are extended many-body correlations near the Fermi surface rather than two-particle bound states in the vacuum Schr\"odinger sense, but their collective behavior can be described by an effective complex wave function. For more information we refer to \cite{deGennes-66,LanLifs-81}. Critical points of the Bardeen--Cooper--Schrieffer (BCS) functional satisfy what is known as the \textit{Bogoliubov--de Gennes (BdG) equation}.

In 1959, Gor'kov \cite{Gorkov-59} established on a physical level of rigor a connection between the BCS model and the GL model close to the critical temperature. The first mathematically rigorous proof was given in \cite{FHSS-12} in the case of small external fields that vary on the macroscopic scale; see also \cite{FHSS-16}. For later extensions to magnetic fields with non-vanishing flux, see \cite{FHL-19,DeuHainzlMaier-22,DeuHainzlMaier-23}. These works show that, in the relevant scaling limit, the minimal BCS energy is approximated by the minimal GL energy and that corresponding minimizers are close, in some sense.
 
In this paper, we initiate the study of this approximation on the level of critical points. Our main result is that solutions of the BdG equation (that is, critical points of the BCS functional) are approximated by solutions of the GL equation (that is, critical points of the GL functional) and vice versa. We prove this for temperatures close to the critical temperature and under a natural assumption on the energy of solutions of the BdG equation. Our most serious assumption, which we hope to remove in future work, is that there are no external fields. While this makes the BdG equation translation invariant, we emphasize that we are looking for general, not necessarily translation invariant solutions. Thus, our work retains many of the difficulties that are encountered in the presence of external fields.

Anticipating the precise formulation in Theorems \ref{thm:deriv-GL} and \ref{thm:deriv-BdG}, our main results show that solutions $\alpha$ of the BdG equation and solutions $\psi$ of the GL equation are related by the asymptotic factorization
\begin{equation}
    \label{eq:introprod}
    \alpha(x,y) \approx h\, \psi(h(x+y)/2) \, \alpha_*(x-y)
\end{equation}
with a universal function $\alpha_*$ depending only on the microscopic interaction and the chemical potential. Here $h>0$ is the small asymptotic parameter that characterizes the closeness of the temperature to the critical temperature. Moreover, $\alpha$, the \emph{Cooper pair wave function}, is a function on $\R^d\times\R^d$ that satisfies $\alpha(x,y) = \alpha(y,x)$ and $\alpha(x+z,y+z)=\alpha(x,y)$ for any $x,y\in\R^d$ and $z\in h^{-1}\Z^d$, and $\psi$, the \emph{Ginzburg--Landau order parameter}, is a function on $\R^d$ that satisfies $\psi(X+Z)=\psi(X)$ for all $X\in\R^d$ and $Z\in\Z^d$.

The product form \eqref{eq:introprod}, where the center of mass variable $(x+y)/2$ and the relative variable $x-y$ live on different scales, is reminiscent of the Weyl quantization in the theory of pseudodifferential operators and, indeed, techniques from semiclassical analysis play an important role in what follows. In contrast to other works on semiclassical analysis, however, the smoothness of the objects we are interested in (or rather the control of this smoothness in terms of $h$) is not guaranteed and we need to work at very low regularity.

We now turn towards a rigorous formulation of our main results.


\subsection{The Bogoliubov--de Gennes equation}\label{sec:BCS}
In BCS theory, the state of a superconductor is described by a pair of operators $(\gamma,\alpha)$ on $L^{2}(\R^d)$ that satisfies
\begin{align}\label{BdG-1pdm}
	\gamma=\gamma^{*},\quad\overline{\alpha}=\alpha^{*},\quad 0\leq \Gamma_{(\gamma,\alpha)}:=\begin{pmatrix}
		\gamma	&	\alpha\\
		\overline{\alpha}	&	1-\overline{\gamma}
	\end{pmatrix}\leq \one\text{ on }L^{2}(\R^{d})\oplus L^{2}(\R^{d}),
\end{align}
where $\overline{K}:=C K C$ and $C$ denotes complex conjugation on $L^{2}(\R^{d})$.  On the level of integral kernels, the requirement $\overline{\alpha}=\alpha^*$ means $\alpha(y,x)=\alpha(x,y)$.  Furthermore, we assume $\gamma$ and $\alpha$ are \textit{$h^{-1}\Z^{d}$-periodic operators} in the sense that they commute with translations of the lattice (equivalently, in terms of integral kernel, $\alpha(x+Z/h,y+Z/h)=\alpha(x,y)$ for any $Z\in\Z^{d}$).

The Bogoliubov--de Gennes (BdG) equation that we study in this paper depends on the following `microscopic parameters':
\begin{itemize}
    \item an even function $V$ on $\R^d$,
    \item the chemical potential $\mu\in\R$,
    \item the temperature $T=\beta^{-1}> 0$.
\end{itemize}
The function $-2V$ is the local 2-body potential modeling the interaction of the fermions. In the usual physics applications we have $\mu>0$, but we do not require this.

The BdG equation reads
\begin{align}\label{BdG-eqn3a}
	\boxed{\alpha=-\frac{1}{2}\Big[\tanh\Big(\frac{\beta H(-2V\alpha)}{2}\Big)\Big]_{12} \,.}
\end{align}
Here $V\alpha$ denotes the operator on $L^2(\R^d)$ with integral kernel $V(x-y)\alpha(x,y)$. Moreover, we set
$$
\fh:=-\Delta-\mu
$$
and, for any periodic operator $\sigma$ on $L^2(\R^d)$,
\begin{align*}
	H(\sigma):=\begin{pmatrix}
		\fh	&	\sigma\\
		\overline{\sigma}	&	-\overline{\fh}
	\end{pmatrix} \,.
\end{align*}
Note that \eqref{BdG-eqn3a} is a nonlinear equation since the right side depends on $\alpha$ through the off-diagonal of the operator $H(-2V\alpha)$. We will explain in Appendix \ref{app:derbdg} that the BdG equation implies that there is a $\gamma$ such that $(\gamma,\alpha)$ is a critical point of the BCS functional.


\subsection{The Ginzburg--Landau equation}\label{sec:GL}
In Ginz\-burg--Landau (GL) theory, a superconductor is described a complex-valued function $\psi:\R^{d}\rightarrow\C$ (known as the \textit{order parameter} in Landau's theory).  This functions is $\Z^{d}$-periodic, i.e., $\psi(X+Z)=\psi(X)$ for all $X\in\R^{d}$ and $Z\in\Z^{d}$.

The GL equation reads
\begin{equation}
    \label{GL-eqn}
    \boxed{\big(-\nabla\cdot \Lambda_0\nabla-\Lambda_{2} D\big)\psi+\Lambda_{3}|\psi|^{2}\psi = 0 \,.}
\end{equation}
Here $\Lambda_0$ is a positive definite $d\times d$-matrix and $\Lambda_{2}, \Lambda_{3}$ and $D$ are some positive constants. (The notation may be surprising, but we choose it for consistancy with our previous work \cite{FHSS-16}.) As explained in the introduction, Ginzburg and Landau \cite{GinzLandau-50} viewed $\Lambda_0$, $\Lambda_2$, $\Lambda_3$ and $D$ as phenomenological parameters, while Gorkov \cite{Gorkov-59} interpreted them as arising from a microscopic model.

Clearly, a solution $\psi$ of the GL equation is a critical point of the well-known GL energy functional.


\subsection{Assumptions}
In our model, we assume
\begin{enumerate}
	\item[(A1)] $\mu\in\R$
	\item[(A2)] $V$ is reflection-symmetric (that is, $V(-x)=V(x)$ for all $x$), nonnegative, $V\in L^{1}(\R^{d}) \cap L^{\infty}(\R^{d})$, and $|x|^{2}V\in L^{\infty}(\R^{d})$
\end{enumerate}

To formulate our assumptions on the critical temperature, we recall \cite{HainHaSeriSolov-08} that there is a uniquely defined number $T_c\geq0$ such that
$$
\inf{\rm spec}_{L^2_{\rm symm}(\R^d)} \left( \frac{\mathfrak h}{\tanh \frac{\mathfrak h}{2T}} - V \right) < 0
\qquad\text{for all}\ 0\leq T<T_c
$$
and
$$
\frac{\mathfrak h}{\tanh \frac{\mathfrak h}{2T_c}} - V \geq 0
\qquad 
$$
Here $L_{\rm symm}^{2}(\R^{d})$ denotes the subspace of reflection-symmetric functions in $L^{2}(\R^{d})$.
\begin{enumerate}
	\item[(A3)] We have $T_c>0$.
\end{enumerate}
Assumption (A2) implies that the essential spectrum of the operator $\fh / \tanh \frac{\fh}{2T_c} - V$ in $L^2_{\rm symm}(\R^d)$ is $[2T_c,\infty)$ if $\mu\geq 0$ and $[|\mu| /\tanh(|\mu|/2T_c ),\infty)$ if $\mu < 0$. Thus, under assumption (A3) the operator has a zero eigenvalue and this eigenvalue is of finite multiplicity.

\begin{enumerate}
	\item[(A4)] The eigenvalue 0 of $\fh / \tanh \frac{\fh}{2T_c} - V$ in $L^2_{\rm symm}(\R^d)$ is simple.
\end{enumerate}

We choose a reflection-symmetric eigenfunction $\alpha_*$ corresponding to the zero-eigenvalue of $\fh / \tanh\frac{\fh}{2T_c} - V$. As $\fh$ and $V$ are real operators, $\alpha_*$ can be chosen real.

Finally, we suppose the inverse temperature $\beta=T^{-1}$ belongs to the following regime, where $\beta_c:=T_c^{-1}$:
\begin{enumerate}
	\item[(A5)] $\beta=\beta_{c}(1+Dh^{2})$ for some constant $D>0$.
\end{enumerate}


\subsection{Main results}\label{sec:main1}
To state our main results, we fix some notation. We assume throughout that $1\leq d\leq 3$ and write $\T^d=\R^d/\Z^d$, which is the box $[0,1]^d$ with opposite sides identified. The Lebesgue space $L^2(\T^d)$, the Sobolev space $H^1(\T^d)$ and its dual space $H^{-1}(\T^d)$ are defined as usual, and for a function $\psi$ on $\T^d$ belonging to one of these spaces $\mathcal H$, $\|\psi\|_{\mathcal H}$ denotes the corresponding norm. Moreover, for an $h^{-1}\Z^d$-periodic operator $\sigma$ on $L^2(\R^d)$ let
\begin{equation}
    \label{eq:h1normop}
    \|\sigma \|_{\cH_h^1}:= \Big( h^d \Tr \chi_{\T^d_{h}} \sigma^{*}\sigma \chi_{\T^d_{h}} + h^{d-2} \sum_{j=1}^d \Tr \chi_{\T^d_{h}} (p_j\sigma -\sigma p_j)^{*}(p_j\sigma -\sigma p_j) \chi_{\T^d_{h}} \Big)^{1/2}.
\end{equation}
Here $\chi_{\T^d_{h}}$ is the characteristic function of $\T^d_h = \R^d/h^{-1}\Z^d$ and $p_j = -i\nabla_j$. This is an $H^1$-type norm for operators where, however, the derivatives act only with respect to the center of mass variable $(x+y)/2$ of the integral kernel $\sigma(x,y)$. These norms are discussed in more detail in Section \ref{sec:preliminary}.

The following is our first main result. It shows that, given a solution of the BdG equation, we obtain an approximate solution of the GL equation.

\begin{thm}\label{thm:deriv-GL}
	Let $1\leq d\leq 3$ and suppose Assumptions (A1)--(A5) hold. Let $h>0$ be sufficiently small and let $\alpha$ be a solution of the BdG equation \eqref{BdG-eqn3a} such that
	\begin{align}\label{small-data-deriv}
		\|\alpha\|_{\cH_{h}^{1}}&\lesssim h \,.
	\end{align}
	Then one can decompose $\alpha$ as 
	\begin{align}\label{alpha-decomposition}
		\alpha(x,y)&=h \, \psi(h(x+y)/2) \, \alpha_{*}(x-y)+\xi(x,y)
	\end{align}
	for a $\Z^{d}$-periodic function $\psi$ on $\T^d$ and an $h^{-1} \Z^d$-periodic operator $\xi$, satisfying
	\begin{align}\label{GL-esti}
		\|\psi\|_{H^{1}(\T^d)}\lesssim\|\psi\|_{L^{2}(\T^d)}\lesssim 1 \,, \qquad\|\xi\|_{\cH_{h}^{1}}&\lesssim h^{7/6} \|\psi\|_{L^{2}(\T^{d})}\,.
	\end{align}
	Moreover, $\psi$ is an approximate solution to the GL equation \eqref{GL-eqn} in the sense that
	\begin{align}\label{GL-esti1}
		\big\|\big(-\nabla\cdot \Lambda_0\nabla-\Lambda_{2}D\big)\psi+\Lambda_{3}|\psi|^{2}\psi \big\|_{H^{-1}(\T^d)}\lesssim h^{1/6}\,\|\psi\|_{L^2(\T^d)} \,.
	\end{align}
    Here the coefficients $\Lambda_0$, $\Lambda_2$ and $\Lambda_3$ are given by \eqref{eq:lambda0}--\eqref{eq:lambda3} below.
\end{thm}

Our second main result addresses the reverse direction, where, given a solution of the GL equation, we obtain an approximate solution of the BdG equation.

\begin{thm}\label{thm:deriv-BdG}
	Let $1\leq d\leq 3$ and suppose Assumptions (A1)--(A5) hold.  Let $\psi\in H^{1}(\T^{d})$ be a solution to the GL equation \eqref{GL-eqn}, with $\Lambda_{0},\Lambda_{2}$ and $\Lambda_{3}$ defined by \eqref{eq:lambda0}--\eqref{eq:lambda3} below. Then, for any sufficiently small $h>0$, there is a pair $(\gamma_\psi,\alpha_\psi)$ of $h^{-1}\Z^d$-periodic operators satisfying \eqref{BdG-1pdm} such that one can decompose $\alpha_\psi$ as
    \begin{align}\label{Valpha-trial-decomp}
        \alpha_{\psi}(x,y)&=h \, \psi(h(x+y)/2) \, \alpha_{*}(x-y)+\xi_{\psi}(x,y)
    \end{align}
    with
    \begin{align}\label{Valpha-trial-rem}
        \|\xi_{\psi}\|_{\cH_{h}^{1}}&\lesssim h^{7/6}\|\psi\|_{L^{2}(\T^{d})}.
    \end{align}
    Moreover, $\alpha_{\psi}$ is an approximate solution to the BdG equation in the sense that
	\begin{align}\label{BdG-approx}
		\Big\|\alpha_{\psi}+\frac{1}{2}\Big[\tanh\Big(\frac{\beta H(-2V\alpha_{\psi})}{2}\Big)\Big]_{12}\Big\|_{\cL_{h}^{2}}\lesssim h^{19/6}\|\psi\|_{L^{2}(\T^{d})}.
	\end{align}    
\end{thm}


\subsection{Discussion of the result}

Let us make the following remarks concerning Theorems \ref{thm:deriv-GL} and \ref{thm:deriv-BdG}.

\smallskip

(a) It is justified to think of the part $h\,\psi(h(x+y)/2)\,\alpha_*(x-y)$ in the decompositions \eqref{alpha-decomposition} and \eqref{Valpha-trial-decomp} as the main part of $\alpha$ and $\alpha_\psi$. Indeed, the $\cH^1_h$-norm of this term is $h \|\psi\|_{H^1(\T^d)}$ and this is larger than $O(h^{7/6})\|\psi\|_{L^2(\T^d)}$, which is the bound on the remainder $\xi$ in \eqref{GL-esti} and \eqref{Valpha-trial-rem}. Note that this also implies that the operator $\alpha=\alpha_\psi$ in Theorem \ref{thm:deriv-BdG} satisfies \eqref{small-data-deriv}.

\smallskip

(b) The crucial thing about the quantity $h^{19/6}= h^{3+1/6}$ in \eqref{BdG-approx} is that it is $o(h^3)$. It captures a cancellation at order $h^3$ due to the GL equation. This is analogous to the quantity $h^{1/6}$ in \eqref{GL-esti1}.

\smallskip

(c) Note that the same quantity $h^{1/6}$ expresses a smallness in \eqref{alpha-decomposition}, \eqref{GL-esti1}, \eqref{Valpha-trial-rem} and \eqref{BdG-approx}.
The exponent $1/6$ is technical and can be improved upon, but our choice is convenient, since it works in all dimensions and for all these quantities. 

\smallskip

(d) Theorem \ref{thm:deriv-GL} does not provide a lower bound on $\|\psi\|_{L^2(\T^d)}$. However, if $\psi\to 0$ in $L^2(\T^d)$ along a sequence of $h$'s tending to zero, then \eqref{GL-esti1} implies that $\Lambda_2 D$ is an eigenvalue of the operator $-\nabla\cdot\Lambda_0\nabla$ in $L^2(\T^d)$. Thus, for all but a discrete set of values of $D$, $\psi$ does not tend to zero.

\smallskip

(e) The assumption \eqref{small-data-deriv} in Theorem \ref{thm:deriv-GL}
implies both that $\alpha$ is small and that it varies on the scale $h^{-1}$ in the center of mass variable $(x+y)/2$.

\smallskip

(f) When $(\gamma,\alpha)$ is a \emph{minimizer} of the BCS functional, then with some work one can show that $\alpha$ satisfies \eqref{small-data-deriv}. Therefore Theorem \ref{thm:deriv-GL} gives the decomposition \eqref{alpha-decomposition} with the bounds \eqref{GL-esti}, which is one of the main results of \cite{FHSS-12}. However, the results of \cite{FHSS-12} are valid in the presence of external fields, which we have not yet included in the critical point setting. Also, the assumptions on $V$ are slightly weaker in \cite{FHSS-12}.

\smallskip

(g) The solution of the GL equation \eqref{GL-eqn} with minimal energy is the constant solution. Depending on $\Lambda_0$, $\Lambda_2D$ and $\Lambda_3$, the equation can also have non-constant solutions. This is easy to see in the one-dimensional case by a phase-plane analysis, and such one-dimensional solutions also provide examples of non-constant solutions in higher dimensions. In general, there are also genuinely higher-dimensional solutions. Meanwhile, while we believe that there are non-minimal solutions of the BdG equation, we are not aware of any unpublished results on this and we will provide a proof in a forthcoming paper.


\subsection{Conventions}

In Theorems \ref{thm:deriv-GL} and \ref{thm:deriv-BdG} and throughout the rest of this paper, to simplify notations, we write $A\lesssim B$ if there is some constant $c>0$, depending only on $\mu$, $D$ and norms of $V$, such that $A\leq cB$. More precisely, the statement of Theorem \ref{thm:deriv-GL} is that for $V$, $\mu$ and $D$ such that Assumptions (A1)--(A5) hold and for a given parameter $M$, there are constants $C$ and $h_0$ such that for all $0<h\leq h_0$ and for any $h^{-1}\Z^d$-periodic solution $\alpha$ satisfying $\|\alpha\|_{\cH_{h}^{1}} \leq M h$, the inequalities \eqref{GL-esti} and \eqref{GL-esti1} hold with $\lesssim \ldots$ replaced by $\leq C \times \ldots$.

\smallskip

Finally, let us state the formulas for the coefficients in the GL equation. With $\alpha_*$ as specified after Assumption (A4) we let
\begin{equation}
\label{eq:deft}
t_*(p) := 2(2\pi)^{-d/2} \int_{\R^d} V(x)\alpha_*(x) e^{-ip\cdot x}\,dx \,,\qquad p\in\R^d \,.
\end{equation}
We define a matrix $\Lambda_0\in\R^{d\times d}$ and constants $\Lambda_2,\Lambda_3\in\R$ in terms of $t_*$. We need the functions
\begin{equation}
\label{eq:defg12}
g_1(z) := \frac{e^{2z}-2ze^z-1}{z^2(1+e^z)^2}
\qquad\text{and}\qquad
g_2(z) := \frac{2e^z (e^z-1)}{z(e^z+1)^3} \,.
\end{equation}
Then
\begin{align}
\label{eq:lambda0}
\left(\Lambda_0\right)_{ij} & := \frac{\beta_c^2}{16} \int_{\R^d} t_*(p)^2 \left( \delta_{ij} g_1(\beta_c(p^2-\mu)) + 2\beta_c p_ip_j g_2(\beta_c(p^2-\mu))\right)\frac{dp}{(2\pi)^d} \,, \\
\label{eq:lambda2}
\Lambda_2 & := \frac{\beta_c}{8} \int_{\R^d} t_*(p)^2\, \cosh^{-2}(\beta_c(p^2-\mu)/2) \,\frac{dp}{(2\pi)^d} \,, \\
\label{eq:lambda3}
\Lambda_3 & := \frac{\beta_c^2}{16} \int_{\R^d} t_*(p)^4\, \frac{g_1(\beta_c(p^2-\mu))}{p^2-\mu}\,\frac{dp}{(2\pi)^d} \,.
\end{align}
Note that $\Lambda_2>0$ and $\Lambda_3>0$, since the integrands are pointwise positive. Moreover, one can show that the matrix $\Lambda_0$ is positive definite, see \cite[Sec.~1.4]{FHSS-12}.


\subsection{Outline of the proof}

Let us briefly outline the strategy that we use to prove our first main result, Theorem \ref{thm:deriv-GL}.

The first step in our proof is to reformulate the BdG equation in Birman--Schwinger form, which has the advantage that we work with bounded, rather than unbounded operators. This is done in Section \ref{sec:bs}. This reformulation is due to \cite{FHL-19} and later also used in \cite{FrankHainzl-18,DeuHainzlMaier-22,DeuHainzlMaier-23}. A novelty of our paper here is that the final result in Theorem 1.1 is stated in terms of the original Cooper pair wave function and not in its Birman--Schwinger form.

In the BdG equation in Birman--Schwinger form we separate the linear and the nonlinear contributions. Sections \ref{sec:linear} and \ref{sec:nonlinear} are somewhat technical, but important, and concern mapping properties of these linear and nonlinear terms.

The main part of the proof of Theorem \ref{thm:deriv-GL} appears in Section \ref{sec:pf-deriv-GL}. Our strategy there is in the spirit of the Lyapunov--Schmidt reduction technique in the sense that we study separately the projection of the equation on the range of a `relevant' subspace and its complement. In contrast to many other applications of this method, the `relevant' subspace has a huge dimension. In fact, the dimension is only finite because we introduced a momentum cut-off and, in particular, it diverges as $h\to 0$. 

It is the projection of the BdG equation on this `relevant' subspace that eventually gives the approximate GL equation. However, to extract the GL equation, we need a rather precise control on the `irrelevant' contribution of $\alpha$, which is obtained from the projection onto the complement of the `relevant' subspace. This is the content of Proposition \ref{prop:small-NP-diff}. The actual extraction of the GL equation appears in Propositions \ref{prop:GL-matrix} and \ref{prop:GL-nonlinear} and uses ideas related to semiclassical analysis.

The strategy to prove Theorem \ref{thm:deriv-BdG} is similar to the proof of Theorem \ref{thm:deriv-GL}, but in reverse order. The main analytic ingredients are the same. Remarkably, in Theorem \ref{thm:deriv-BdG} a first-order ansatz is not sufficient and a second-order correction needs to be taken into account. This is the term $\varphi_{\psi,1}$ in \eqref{eq:defphipsi}. It does not have an analogue in the corresponding upper bound construction in the minimization setting in \cite{FHSS-16}.


\subsection*{Acknowledgements}

The suggestion to study the relation between solutions of the BDG and GL equations is due to I.~M.~Sigal, to whom we are very grateful. The origins of this work go back to the fall of 2014 and throughout these years we had the opportunity to discuss the results with many people including I.~M.~Sigal, L.~Chen, T.~Tzanateas and A.~Geisinger. We would like to thank them for their input, which has greatly advanced our understanding of the subject.

\smallskip

\section{The Birman--Schwinger formulation}\label{sec:bs}

Our goal in this section is to rewrite the BdG equation \eqref{BdG-eqn3a} in an equivalent form that is reminiscent of the Birman--Schwinger reformulation of the Schr\"odinger equation.

We set
$$
\rho(z) := \tanh(z/2)
$$
and let
$$
\omega_{n}:=\pi(2n+1)/\beta
$$
denote the poles of $z\mapsto\rho(\beta z)$, known in the physics literature as \textit{Matsubara frequencies}. Then we have the representations
$$
\tanh(\frac{\beta E}{2}) = \int_{\mathcal C} \frac{dz}{2\pi i} \,\rho(\beta z) (z-E)^{-1}
= - \frac{2}{\beta} \sum_{n\in\Z} \frac{1}{i\omega_n - E}
$$
for $E\in\R$. Here $\cC=\{r\pm i\pi/(2\beta)\mid r\in\R\}$ is positively oriented. Both the integral and the sum representation are understood in a principal value sense.

By the spectral mapping theorem, we obtain, for any periodic operator $\sigma$ on $L^2(\R^d)$,
$$
\tanh(\frac{\beta H(\sigma)}{2}) = \int_{\mathcal C} \frac{dz}{2\pi i} \,\rho(\beta z) (z-H(\sigma))^{-1}
= - \frac{2}{\beta} \sum_{n\in\Z} \frac{1}{i\omega_n - H(\sigma)} \,.
$$
Writing
\begin{align*}
    (z-H(\sigma))^{-1} & = (z-H(0))^{-1} + (z-H(0))^{-1} \begin{pmatrix}
    0& \sigma \\ \overline\sigma & 0
\end{pmatrix} (z-H(0))^{-1} \\
& \quad + (z-H(0))^{-1} \begin{pmatrix}
    0& \sigma \\ \overline\sigma & 0
\end{pmatrix} (z-H(0))^{-1} \begin{pmatrix}
    0& \sigma \\ \overline\sigma & 0
\end{pmatrix} (z-H(\sigma))^{-1}
\end{align*}
and taking the upper off-diagonal argument, we find
\begin{align*}
    \left[(z-H(\sigma))^{-1}\right]_{12} & = (z-\fh)^{-1}\sigma(z+\overline{\fh})^{-1} \\
& \quad + (z-\fh)^{-1}\sigma(z+\overline{\fh})^{-1} \overline\sigma (z-\fh)^{-1} \left[(z-H(\sigma))^{-1}\right]_{12} \,.
\end{align*}
Consequently, if we define the linear operator $K_T$ by
\begin{align}\label{linear}
	K_{T}(\sigma)\DETAILS{&:=\frac{\tanh[\beta(\fh)_{x}/2]+\tanh[\beta(\fh)_{y}/2]}{(\fh)_{x}/2+(\fh)_{y}/2}\nonumber\\}
	&: = \int_{\cC}\frac{dz}{2\pi i} \, \rho(\beta z)(z-\fh)^{-1}\sigma(z+\overline{\fh})^{-1} = - \frac{2}{\beta}\sum_{n\in\Z}(i\omega_{n}-\fh)^{-1}\sigma(i\omega_{n}+\overline{\fh})^{-1}
\end{align}
and the nonlinear mapping $N_{T}$ by
\begin{align}\label{nonlinear}
	N_{T}(\sigma)&:=\int_{\cC}\frac{dz}{2\pi i} \, \rho(\beta z)(z-\fh)^{-1}\sigma(z+\overline{\fh})^{-1}\overline{\sigma}[(z-H(\sigma)^{-1})]_{12}\nonumber\\
	&= - \frac{2}{\beta}\sum_{n\in\Z}(i\omega_{n}-\fh)^{-1}\sigma(i\omega_{n}+\overline{\fh})^{-1}\overline{\sigma}[(z-H(\sigma))^{-1}]_{12},
\end{align}
then we have shown that
\begin{equation}
    \label{eq:tanhkn}
    \left[ \tanh(\frac{\beta H(\sigma)}{2}) \right]_{12} = K_{T}(\sigma) + N_{T}(\sigma) \,.
\end{equation}
This allows us to rewrite the BdG equation \eqref{BdG-eqn3a} as
\begin{align}\label{BdG-eqn5}
	(\one-K_{T}V)\alpha+\tfrac{1}{2}N_{T}(-2V\alpha) = 0.
\end{align}

Next, we introduce a convenient and symmetric version of \eqref{BdG-eqn5}.  We denote $\varphi :=V^{1/2}\alpha$ and, multiplying both sides of \eqref{BdG-eqn5} by $V^{1/2}$, we find
\begin{align}\label{BdG-BS-eqn}
	\boxed{ (\one-V^{1/2}K_{T}V^{1/2})\varphi+\tfrac{1}{2}V^{1/2}N_{T}(-2V^{1/2}\varphi) = 0 \,.}
\end{align}
We call \eqref{BdG-BS-eqn} the \textit{Birman--Schwinger representation} of the BdG equation. What we have shown is that, if $\alpha$ solves \eqref{BdG-eqn5}, then $\varphi=V^{1/2}\alpha$ solves \eqref{BdG-BS-eqn}.  Conversely, if $\varphi$ solves \eqref{BdG-BS-eqn}, then $\alpha = K_TV^{1/2}\varphi  - \tfrac12 N_T(-2V^{1/2}\varphi)$ solves \eqref{BdG-eqn5}, as is easily verified. Therefore \eqref{BdG-BS-eqn} and \eqref{BdG-eqn5} are equivalent.


\section{Preliminary material}\label{sec:setup}

\subsection{Coordinates}
Originally, the solutions $\alpha$ and $\varphi$ of equations \eqref{BdG-eqn3a} and \eqref{BdG-BS-eqn} are functions of the variables $(x,y)\in\R^d\times\R^d$. In what follows it is often easier to work in the relative and center of mass coordinates $r=x-y$ and $X=(x+y)/2$, so that it is convenient to introduce the notation
\begin{align}
	\zeta_{X}^{r} :=X+\frac{r}{2},
\end{align}
and thus $x=\zeta_{X}^{r}$ and $y=\zeta_{X}^{-r}$. 


\subsection{Lebesgue and Sobolev norms of functions}

We will use two different $L^p$-type spaces for $1\leq p<\infty$. The first one is the usual space $L^p(\R^d)$, while the second one, denoted by $L^p_h(\T_h^d)$, consists of $L^p$-integrable functions $\Psi$ on $\T^d_h = \R^d/h^{-1}\Z^d$ with the volume-normalized norm
$$
\| \Psi \|_{L^p_h(\T^d_h)} = \left( h^d \int_{\T^d_h} |\Psi(X)|^p\,dX \right)^{1/p}.
$$
For $p=\infty$ this distinction disappears.

For $\Psi\in L^2_h(\T^d_h)$ we have the Fourier transform
\begin{equation}
    \label{hZ-FT}
    \hat\Psi(q) = h^{d} \int_{\T^d_{h}} e^{-ihq\cdot X} \Psi(X) \,dX \,,
    \quad q\in 2\pi\Z^d \,,
\end{equation}
and the Plancherel identity
$$
\|\Psi \|_{L^2_h(\T^d_h)} = \Big( \sum_{q\in 2\pi\Z^d} |\hat\Psi(q)|^2 \Big)^{1/2}.
$$
For $s\geq 0$ the Sobolev space $H^s_h(\T^d_h)$ consists of functions $\Psi\in L^2_h(\T^d_h)$ such that
$$
\| \Psi \|_{H^s_h(\T^d_h)} := \left( \sum_{q\in 2\pi\Z^d} (1+|q|^2)^s |\hat\Psi(q)|^2 \right)^{1/2}<\infty \,.
$$
In terms of the Laplacian $-\Delta_X$ with respect to the coordinate $X$ in $\T^d_h$ we have
$$
\| \Psi \|_{H^s_h(\T^d_h)} = \Big( \int_{\T^d_h} \big|((\one-h^{-2}\Delta_{X})^{s/2}\Psi)(X)\big|^{2} \, dX\Big)^{1/2}.
$$
For $s<0$, $H^s_h(\T^d_h)$ is the dual space of $H^{-s}_h(\T^d_{h})$.

Finally, when $h=1$, we drop the subscript in $\T^d_h$, $L^p_h(\T^d_h)$ and $H^s_h(\T^d_h)$.


\subsection{Local \texorpdfstring{$p$}{p}-th Schatten space}\label{sec:preliminary}

We let $\cL_{h}^\infty$ denote the space of bounded operators $\sigma$ on $L^2(\R^d)$ that are $h^{-1}\Z^{d}$-periodic and satisfy $\overline\sigma =\sigma^*$. For the operator norm we use interchangeably the notations
$$
\|\sigma\|\equiv \|\sigma\|_{\cL_{h}^{\infty}} \,.
$$

We define, for $h^{-1}\Z^d$ periodic operators $\sigma$ on $L^2(\R^d)$, the trace per unit volume $\Tr_{\T^d_{h}}$ by
\begin{align}\label{tr-unit}
	\Tr_{\T^d_{h}}(\sigma)&=\Tr(\chi_{\T^d_{h}}\sigma\chi_{\T^d_{h}}),
\end{align}
where $\Tr$ is the standard trace on $L^{2}(\R^{d})$ and $\chi_{\T^d_{h}}$ is the characteristic function of $\T^d_{h}$.

For each $1\leq p<\infty$, we let
$$
\cL_{h}^{p}
$$
be the space of operators $\sigma\in\cL_h^\infty$ such that
\begin{align}
	\|\sigma\|_{\cL_{h}^{p}}& := \big[h^{d}\Tr_{\T^d_{h}}\big(\sigma^{*}\sigma\big)^{p/2}\big]^{1/p}<\infty \,.
\end{align}
We emphasize the $h$-dependence of this norm and we note that the $\cL_{h}^{2}$-norm can be expressed as 
\begin{align*}
	\|\sigma\|_{\cL_{h}^{2}}^{2}&=h^{d}\int_{\R^{d}}\int_{\T^d_{h}}|\sigma(x,y)|^{2} \,dy \, dx =h^{d}\int_{\R^{d}}\int_{\T^d_{h}}|\sigma(\zeta_{X}^{r},\zeta_{X}^{-r})|^{2} \,dX\,dr.
\end{align*}
Moreover, $\cL_{h}^{2}$ is a Hilbert space with inner product
\begin{align}
	\langle \sigma,\sigma'\rangle_{\cL_{h}^{2}}&=h^{d}\int_{\R^{d}}\int_{\T^d_{h}}\overline{\sigma(\zeta_{X}^{r},\zeta_{X}^{-r})}\sigma'(\zeta_{X}^{r},\zeta_{X}^{-r})\,dX\,dr.
\end{align}
It is clear that the Cauchy--Schwarz and triangle inequalities hold
\begin{align}\label{CS-triangle}
	\big|\langle \sigma,\sigma'\rangle_{\cL_{h}^{2}}\big|\leq \|\sigma\|_{\cL_{h}^{2}}\|\sigma'\|_{\cL_{h}^{2}},\quad\|\sigma+\sigma'\|_{\cL_{h}^{p}}&\leq \|\sigma\|_{\cL_{h}^{p}}+\|\sigma'\|_{\cL_{h}^{p}}.
\end{align}
Furthermore, for $1\leq p,q,r\leq \infty$ with $1/p+1/q=1/r$, we have the general H\"older inequality
\begin{align}\label{Holder-periodic}
	\|\sigma\sigma'\|_{\cL_{h}^{r}}&\leq \|\sigma\|_{\cL_{h}^{p}}\|\sigma'\|_{\cL_{h}^{q}}.
\end{align}
For proof of these facts, we refer to \cite{DeuHainzlMaier-23,FHSS-12} and references therein.  

Next, we introduce Sobolev type norms for operators in $\cL_{h}^{2}$.  This is more conveniently formulated on the level of integral kernels.  Let $\sigma$ be an $h^{-1}\Z^{d}$-periodic operator.  For each $r\in\R^{d}$, the function $X\mapsto \sigma(\zeta_{X}^{r},\zeta_{X}^{-r})$ is $h^{-1}\Z^{d}$-periodic, and we define the Fourier transform on $(2\pi\Z^{d})\times\R^{d}$ in the center of mass and relative position coordinate $(X,r)$ as
\begin{align}\label{eq:fourier}
	\hat{\sigma}(q,p)& :=\frac{h^{d}}{(2\pi)^{d/2}}\int_{\T^d_{h}}\int_{\R^{d}}e^{-ihq\cdot X}e^{-ip\cdot r}\sigma(\zeta_{X}^{r},\zeta_{X}^{-r})\,dr\,dX \,,
    \quad q\in 2\pi\Z^d \,,\ p\in\R^d \,.
\end{align}
Note that 
\begin{align}
	\label{invFT}\sigma(\zeta_{X}^{r},\zeta_{X}^{-r})&=\frac{1}{(2\pi)^{d/2}}\sum_{q\in 2\pi\Z^{d}}\int_{\R^{d}}e^{ihq\cdot X}e^{ip\cdot r}\hat{\sigma}(q,p)dp,\\
	\label{FT-unitary}\|\sigma\|_{\cL_{h}^{2}}^{2}&=\sum_{q\in 2\pi\Z^{d}}\int_{\R^{d}}|\hat{\sigma}(q,p)|^{2}dp \,.
\end{align}
Then, for each $s\geq 0$, we define $\cH_{h}^{s}$ as the space of operators $\sigma\in\cL_h^2$ such that the following Sobolev type norm is finite:
%
\begin{align}\label{Hs-norm-K}
	\|\sigma\|_{\cH_{h}^{s}}
	&=\Big(\sum_{q\in 2\pi\Z^{d}}\int_{\R^{d}}(1+|q|^{2})^{s}|\hat{\sigma}(q,p)|^{2}dp\Big)^{1/2}.
\end{align}
Clearly, $\cH_{h}^{0}=\cL_{h}^{2}$ and $\cH_{h}^{s_{1}}\subseteq \cH_{h}^{s_{2}}$ if $s_{1}\geq s_{2}$.

There is a different way to interpret the norms $\|\sigma\|_{\cH_{h}^{s}}$, which will be useful in Section \ref{sec:nonlinear}. We note that $\mathcal L^2_h$ is a complex Hilbert space with respect to its natural inner product. In this Hilbert space we define an unbounded, selfadjoint operator $-\Delta_X$ via its quadratic form
$$
h^d \sum_{j=1}^d \Tr_{\T^d_{h}} (p_j \sigma - \sigma p_j)^* (p_j\sigma - \sigma p_j)
$$
with form domain given by all $\sigma\in\mathcal L^2_h$ for which $p_j \sigma -\sigma p_j\in \mathcal L^2_h$ for all $j$. Here, as before, $p_j = -i\nabla_j$. Using cyclicity of the trace per unit volume, one easily verifies that
$$
-\Delta_X \sigma = (-\Delta)\sigma + \sigma(-\Delta) - 2 \sum_{j=1}^d p_j \sigma p_j \,.
$$
Note, in particular, that $\overline\sigma=\sigma^*$ implies the same for $-\Delta_X\sigma$. 

The motivation for the notation $-\Delta_X\sigma$ is that its integral kernel satisfies
$$
(-\Delta_X \sigma)(\zeta^r_X,\zeta_X^{-r}) = -\sum_{j=1}^d \frac{\partial^2}{\partial X_j^2} \left( \sigma(\zeta^r_X,\zeta_X^{-r})\right).
$$
In terms of Fourier coefficients, one finds that
$$
\widehat{(-\Delta_X\sigma)}(q,p) = h^2 |q|^2 \hat\sigma(q,p)
\qquad \text{for all}\ q\in 2\pi\Z^d \,,\ p\in\R^d \,.
$$
This equality shows, in particular, that the expressions \eqref{eq:h1normop} and \eqref{Hs-norm-K} for $\|\sigma\|_{\mathcal H^1_h}$ coincide.

Now for every $s\geq 0$, the operator $(\one-h^{-2}\Delta_X)^{s/2}$ is defined by the functional calculus and we find the following alternative expression for the $\mathcal H^s_h$-norm:
$$
\| \sigma \|_{\mathcal H^s_h} = \| (\one-h^{-2}\Delta_X)^{s/2} \sigma \|_{\mathcal L^2_h} \,.
$$


\subsection{Useful estimates}\label{sec:useful}

In this subsection, we prove some useful estimates, which will be crucial when we study the nonlinear operator $N_{T}$ in \eqref{BdG-BS-eqn}.

\begin{lemma}\label{lem:op-norm-Delta}
	Let $\varepsilon>0$ and $s:=d/2+\varepsilon$.
    Then there is a constant $C_{d,\varepsilon}$ such that for $V\in L^{1}(\R^{d})$ and $\sigma\in\cH^s_h$ one has
	\begin{align}\label{op-norm-Delta}
		\|V^{1/2}\sigma\|\leq C_{d,\varepsilon}\|V\|_{L^{1}}^{1/2}\|\sigma\|_{\cH_{h}^{s}} \,.
	\end{align}
\end{lemma}

\begin{proof}
	Denoting $\tilde{\sigma}(r):=\sup_{X\in \T^d_{h}}\big|\sigma(\zeta_{X}^{r},\zeta_{X}^{-r})\big|$ we obtain, by Schur's test,
	\begin{align}
		\|V^{1/2}\sigma\|^{2}&\leq \Big(\sup_{x}\int_{\R^{d}}\big|V^{1/2}(x-y)||\sigma(x,y)|\big|dy\Big)\Big(\sup_{y}\int_{\R^{d}}\big|V^{1/2}(x-y)||\sigma(x,y)|\big|dx\Big)\nonumber\\
		&\leq \Big( \sup_{x}\int_{\R^{d}} \! |V^{1/2}(x-y)||\tilde{\sigma}(x-y)|dy\Big)\!\Big(\sup_{y}\int_{\R^{d}} \! |V^{1/2}(x-y)||\tilde{\sigma}(x-y)|dx\Big)\nonumber\\
		&\leq \Big(\int_{\R^{d}}|V^{1/2}(r)||\tilde{\sigma}(r)|dr\Big)^{2} \nonumber \\
        & \leq \|V\|_{L^{1}} \|\tilde{\sigma}\|_{L^{2}}^2 \,.
	\end{align}
	By the Sobolev embedding on the torus $\T^d$ (see, e.g., \cite{Hebey}), there is a constant $C_{d,\varepsilon}$ such that
	\begin{align}
		\sup_{X\in \T^d}|f(X)|&\leq C_{d,\varepsilon}\|f\|_{H^{s}(\T^d)},
	\end{align}
	where $s=d/2+\varepsilon$.  It then follows from rescaling for each $r\in\R^{d}$ that
	\begin{align*}
		|\tilde{\sigma}(r)|&\leq C_{d,\varepsilon}\big\|\sigma(\cdot+r/2,\cdot-r/2)\big\|_{H_{h}^{s}(\T^d_{h})}.
	\end{align*}
	This implies that
	\begin{align}
		\|V^{1/2}\sigma\|&\leq \|V\|_{L^{1}}^{1/2}\|\tilde{\sigma}\|_{L^{2}}\nonumber\\
		&\leq C_{d,\varepsilon}\|V\|_{L^{1}}^{1/2}\Big(\int_{\R^{d}}\big\|\sigma(\cdot+r/2,\cdot-r/2)\big\|_{H_{h}^{s}(\T^d_{h})}^{2}dr\Big)^{1/2}\nonumber\\
		&= C_{d,\varepsilon}\|V\|_{L^{1}}^{1/2}\|\sigma\|_{\cH_{h}^{s}}.
	\end{align}
	This completes the proof.
\end{proof}

\begin{lemma}\label{lem:Lpnorm-Delta}
	Let $2\leq p<\infty$ be a positive integer and let $s=1$ for all $2\leq p<\infty$ when $d=1,2$ and
	$$
    s=\begin{cases}
		1	\quad	&	\text{for }2\leq p\leq 3 \,,\\
		3/2&\text{for }p>3 \,.
	\end{cases}
    $$
    when $d=3$. Then there is a constant $C_{p,d}$ such that if $V\in L^{p/(p-1)}(\R^{d})$ and $\sigma \in \cH_{h}^s$, then 
	\begin{align}\label{Lpnorm-Delta}
		\|V^{1/2}\sigma\|_{\cL_{h}^{2p}}&\leq C_{p,d}\|V\|_{L^{p/(p-1)}}^{1/2}\|\sigma\|_{\cH_{h}^{s}} \,.
	\end{align}
\end{lemma}

\begin{proof}
	First, we compute
	\begin{align*}
		\|V^{1/2}\sigma\|_{\cL_{h}^{2p}}^{2p}\DETAILS{&=h^{d}\Tr_{\T^d_{h}}(\Delta_{\alpha}^{*}\Delta_{\alpha})^{p}\nonumber\\}
		&=h^{d}\int_{\R^{2pd}}dx_{1}...dx_{2p} \, \chi_{\T^d_{h}}(x_{1})\prod_{k=1}^{2p}V^{1/2}(x_{k}-x_{k+1})\sigma_{(k)}(x_{k},x_{k+1}) \,,
	\end{align*}
	where $x_{2p+1}=x_{1}$, $\sigma_{(k)}(x,y)=\sigma(x,y)$ if $k$ is even and its complex conjugate if $k$ is odd.  To compute the above integral, we introduce the following notations:
	\begin{align*}
		X&=x_{1},\quad r_{1}=x_{1}-x_{2},\quad r_{2}=x_{2}-x_{3},\quad...,\\
		r_{2p-1}&=x_{2p-1}-x_{2p},\quad r_{2p}=-(r_{1}+r_{2}+...+r_{2p-1}).
	\end{align*}
	Furthermore, we denote $S_{m}$ to be a linear combination of $r_{k}$'s such that $(x_{m}+x_{m+1})/2=X-S_{m}$.  By a change of variable, periodicity of $\sigma$, H\"older and Sobolev inequalities, the above integral satisfies the following estimate
	\begin{align}\label{Lp-Delta-comp2}
		\|V^{1/2}\sigma\|_{\cL_{h}^{2p}}^{2p} \DETAILS{=h^{d}\int_{\R^{(2p+1)d}}dXdr_{1}...dr_{2p} \delta(r_{1}+...+r_{2p})\chi_{\T^d_{h}}(X)\nonumber\\
			&\quad\quad\quad\quad\quad\times\prod_{k=1}^{2p}V^{1/2}(r_{k})\sigma_{(k)}(X-S_{k}+r_{k}/2,X-S_{k}-r_{k}/2)\nonumber\\
			}& \leq \int_{\R^{2pd}}dr_{1}...dr_{2p} \, \delta(r_{1}+...+r_{p})V^{1/2}(r_{1})...V^{1/2}(r_{2p})\nonumber\\
		&\quad\quad\quad\quad\quad\times\Big(h^{d}\int_{\T^d_{h}}dX\prod_{k=1}^{2p}\big|\sigma(X-S_{k}+r_{k}/2,X-S_{k}-r_{k}/2)\big|\Big)\nonumber\\
		&\leq\int_{\R^{2pd}}dr_{1}...dr_{2p} \, \delta(r_{1}+...+r_{2p}) \prod_{k=1}^{2p}V^{1/2}(r_{k})f(r_{k})\nonumber \\
        & \leq \|V^{1/2}f\|_{L^{2p/(2p-1)}}^{2p}\leq \|V^{1/2}\|_{L^{2p/(p-1)}}^{2p}\|f\|_{L^{2}}^{2p} \nonumber \\
        &=\|V\|_{L^{p/(p-1)}}^{p}\|f\|_{L^{2}}^{2p},
	\end{align}
	where
	\begin{align*}
		f(r)&:=\big\|\sigma(\cdot+r/2,\cdot -r/2)\big\|_{L_{h}^{2p}(\T^d_{h})}.
	\end{align*}
	
	Now, by the Sobolev embedding on the torus $\T^d$ and scaling, we obtain for each $r\in\R^{d}$ that
	\begin{align}
		f(r)&\leq C_{p,d}\big\|\sigma(\cdot+r/2,\cdot-r/2)\big\|_{H_{h}^{s}(\T^d_{h})} \,,
	\end{align}
	where $s$ is as in the theorem. This gives
	\begin{align*}
		\|f\|_{L^{2}}&\leq C_{p,d}\Big(\int_{\R^{d}}\big\|\sigma(\cdot+r/2,\cdot-r/2)\big\|_{H_{h}^{s}(\T^d_{h})}^{2}dr\Big)^{1/2}=C_{p,d}\|\sigma\|_{\cH_{h}^{s}}.
	\end{align*}
	Putting everything together, we arrive at the claimed bound.
\end{proof}

\begin{lemma}\label{lem:Hsnorm-Delta}
	If $V\in L^{\infty}(\R^{d})$, then it holds for any $s\geq 0$ that 
	\begin{align}\label{H12norm-Delta}
		\|V^{1/2}\sigma\|_{\cH_{h}^{s}}&\leq \|V\|_{L^{\infty}}^{1/2}\|\sigma\|_{\cH_{h}^{s}}.
	\end{align}
\end{lemma}

\begin{proof}
	This estimate essentially follows from the fact that $V$ is independent of $X$.  Indeed, by H\"older inequality \eqref{Holder-periodic} and the fact that $[V,-\Delta_{X}]=0$, we have
	\begin{align*}
		\|V^{1/2}\sigma\|_{\cH_{h}^{s}}&=\|(\one-h^{-2}\Delta_{X})^{s/2}V^{1/2}\sigma\|_{\cL_{h}^{2}}=\|V^{1/2}(\one-h^{-2}\Delta_{X})^{s/2}\sigma\|_{\cL_{h}^{2}}\nonumber\\
		&\leq \|V^{1/2}\|_{L^{\infty}}\|(\one-h^{-2}\Delta_{X})^{s/2}\sigma\|_{\cL_{h}^{2}}=\|V\|_{L^{\infty}}^{1/2}\|\sigma\|_{\cH_{h}^{s}}.
	\end{align*}
	This completes the proof.
\end{proof}
%
%
%

\bigskip


\section{Linear analysis}\label{sec:linear}

In this section, we study the linear operator $\one-V^{1/2}K_{T}V^{1/2}$
that appears in the Birman--Schwinger version \eqref{BdG-BS-eqn} of the BdG equation.  


\subsection{The operators \texorpdfstring{$K_T$}{KT} and \texorpdfstring{$k_T$}{kT}}

As in \cite{FrankHainzl-18}, for $0<\beta<\infty$ we introduce the functions
\begin{align}\label{Xi-beta}
	\Xi_{\beta}(E,E')& :=
    \begin{cases}   
    \frac{\tanh(\beta E/2)+\tanh(\beta E'/2)}{E+E'} & \text{if}\ E\neq -E' \,, \\
    (\beta/2)/\cosh^{2}(\beta E/2) & \text{if}\ E= -E' \,,
    \end{cases}
\end{align}
and
$$
\chi_{\beta}(E) := \Xi_\beta(E,E) = \frac{\tanh(\beta E/2)}{E} \,.
$$

We recall that $K_T$, defined in \eqref{linear}, is an operator, acting on $h^{-1}\Z^{d}$-periodic operators in $\cL_{h}^{2}$. We introduce the operator
$$
k_{T}:=\chi_{\beta}(-\Delta_{r}-\mu)
\qquad \text{in}\ L^2_{\rm symm}(\R^d)
$$ 
with $T=\beta^{-1}$, where $-\Delta_{r}$ denotes the Laplacian in the $r$-variable. We will identify $k_T$ with an operator acting on $h^{-1}\Z^d$-periodic operators in $\cL_{h}^{2}$ that only acts with respect to the relative variable $r=x-y$.

Our goal in this subsection is to derive a Fourier representation of the operators $K_T$ and $k_T$. This will involve the function
\begin{align}\label{eq:fbeta}
	f_{\beta}(p,q)&:=\Xi_{\beta}\big(|p|^{2}-\mu,|q|^{2}-\mu\big).\DETAILS{=\frac{\tanh(\beta(|p|^{2}-\mu)/2)+\tanh(\beta(|q|^{2}-\mu)/2)}{(|p|^{2}-\mu)+(|q|^{2}-\mu)}.}
\end{align}
For the Fourier transform of operators in $\cL_{h}^2$, see \eqref{eq:fourier}.

\begin{lemma}
    Let $\sigma\in \cL_{h}^{2}$ be an even, $h^{-1}\Z^{d}$-periodic operator in $\cL_{h}^{2}$. Then
    $$
    [K_{T}(\sigma)](\zeta_{X}^{r},\zeta_{X}^{-r})
	=\frac{1}{(2\pi)^{d/2}}\sum_{q\in 2\pi\Z^{d}}\int_{\R^{d}} f_{\beta}(p+hq/2,p-hq/2)e^{ihq\cdot X}e^{ip\cdot r}\hat{\sigma}(q,p)dp
    $$
    and
    $$
    [k_{T}(\sigma)](\zeta_{X}^{r},\zeta_{X}^{-r})=\frac{1}{(2\pi)^{d/2}}\sum_{q\in 2\pi\Z^{d}}\int_{\R^{d}}f_{\beta}(p,p)e^{ihq\cdot X}e^{ip\cdot r}\hat{\sigma}(q,p)dp.
    $$
\end{lemma}

\begin{proof} 
    Recall that $\fh=\overline{\fh}=-\Delta-\mu$.  Let $G_{z}:\R^{d}\times\R^{d}\rightarrow\R$ be the Green's function, that is, the integral kernel of the resolvent $(z-\fh)^{-1}$.  Since
\begin{align*}
	(z+\overline{\fh})^{-1}&=-(-z-\fh)^{-1},
\end{align*}
it follows that $-G_{-z}$ is the Green function of $(z+\overline{\fh})^{-1}$. Recalling the definition \eqref{linear} of $K_T(\sigma)$ for a $h^{-1}\Z^{d}$-periodic operator $\sigma\in\cL_{h}^{2}$, we can write 
\begin{align}
	[K_{T}(\sigma)](\zeta_{X}^{r},\zeta_{X}^{-r})&=-\int_{\cC}\frac{dz}{2\pi i}\,\rho(\beta z)\iint_{\R^d\times\R^d}G_{z}(\zeta_{X}^{r},u)G_{-z}(v,\zeta_{X}^{-r})\sigma(u,v) \,dudv\nonumber\\
	&=-\int_{\cC}\frac{dz}{2\pi i}\,\rho(\beta z)\iint_{\R^d\times\R^d}G_{z}(\zeta_{X}^{r},\zeta_{Y}^{s})G_{-z}(\zeta_{Y}^{-s},\zeta_{X}^{-r})\sigma(\zeta_{Y}^{s},\zeta_{Y}^{-s}) \, dY \, ds\nonumber\\
	&=-\int_{\cC}\frac{dz}{2\pi i}\,\rho(\beta z)\iint_{\R^d\times\R^d}G_{z}(\zeta_{X}^{r},\zeta_{X+Y}^{s})G_{-z}(\zeta_{X+Y}^{-s},\zeta_{X}^{-r})\nonumber\\
	&\quad\quad\quad\quad\quad\quad\quad\quad\quad\quad\quad\quad\quad\times e^{iY\cdot(-i\nabla_{X})}\sigma(\zeta_{X}^{s},\zeta_{X}^{-s}) \,dY\,ds,
\end{align}
where we have applied a change of variable $u=Y+s/2=\zeta_{Y}^{s}$ and $v=Y-s/2=\zeta_{Y}^{-s}$.  We define 
\begin{align}\label{F-kernel}
	F_{T}(X,Y,r,s)& :=-\int_{\cC}\frac{dz}{2\pi i} \, \rho(\beta z)G_{z}(\zeta_{X}^{r},\zeta_{X+Y}^{s})G_{-z}(\zeta_{X+Y}^{-s},\zeta_{X}^{-r}),
\end{align}
so that 
\begin{align}
	&[K_{T}(\sigma)](\zeta_{X}^{r},\zeta_{X}^{-r})=\iint_{\R^d\times\R^d} dY \, ds \, F_{T}(X,Y,r,s)e^{iY\cdot(-i\nabla_{X})}\sigma(\zeta_{X}^{s},\zeta_{X}^{-s}).
\end{align}
Since the operator $\fh$ is translation invariant, we have $G_{z}(x,y)=G_{z}(x-y)=G_{z}(y-x)$, so that $F_{T}(X,Y,r,s)$ depends only on $Y$ and $s-r$ since
\begin{align*}
	G_{z}(\zeta_{X}^{r},\zeta_{X+Y}^{s})&=G_{z}(\zeta_{X}^{r}-\zeta_{X+Y}^{s})=G_{z}(\zeta_{Y}^{s-r}).
\end{align*}
Hence, we shall write $F_{T}(Y,s-r)$ for $F_{T}(X,Y,r,s)$ for simplicity.  Moreover, by Fourier transform and integrating over $z$, we can write $F_{T}$ as follows:
\begin{align}
	F_{T}(Y,r-s)&=-\int_{\cC}\frac{dz}{2\pi i}\,\rho(\beta z)G_{z}(\zeta_{Y}^{s-r})G_{-z}(\zeta_{Y}^{-(s-r)})\nonumber\\
	&=\frac{1}{(2\pi)^{2d}} \int_{\cC}\frac{dz}{2\pi i}\,\rho(\beta z) \iint_{\R^{d}\times\R^d} \frac{e^{ip\cdot \zeta_{Y}^{s-r}}e^{iq\cdot \zeta_{Y}^{-(s-r)}}}{(z-(|p|^{2}-\mu))(z+(|q|^{2}-\mu))}dpdq\nonumber\\
	\DETAILS{&=\frac{1}{(2\pi)^{2d}}\iint_{\R^d\times\R^d}\int_{\cC}\frac{dz}{2\pi i}\,\rho(\beta z)\frac{e^{ip\cdot (s-r)}e^{iq\cdot Y}}{(z-(|p+q/2|^{2}-\mu))(z+(|p-q/2|^{2}-\mu))} \, dpdq\nonumber\\}
	&=\frac{1}{(2\pi)^{2d}} \iint_{\R^d\times\R^d} f_{\beta}(p+q/2,p-q/2)e^{ip\cdot(s-r)}e^{iq\cdot Y} \, dpdq,
\end{align}
with $f_\beta$ from \eqref{eq:fbeta}. Consequently,
\begin{align}\label{KT-kernel-rep}
	[K_{T}(\sigma)](\zeta_{X}^{r},&\zeta_{X}^{-r})=\frac{1}{(2\pi)^{2d}}\iint_{\R^d\times\R^d}\iint_{\R^d\times\R^d}f_{\beta}(p+q/2,p-q/2)e^{ip\cdot(s-r)}\nonumber\\
	&\quad\quad\quad\times\Big[\frac{1}{(2\pi)^{d/2}}\int_{\R^{d}}\sum_{k\in 2\pi\Z^{d}}e^{i(hk+q)\cdot Y}e^{ihk\cdot X}e^{i\ell\cdot s}\hat{\sigma}(k,\ell)d\ell\Big]dYdsdpdq\nonumber\\
	&=\frac{1}{(2\pi)^{d/2}}\sum_{q\in 2\pi\Z^{d}}\int_{\R^{d}} f_{\beta}(p+hq/2,p-hq/2)e^{ihq\cdot X}e^{ip\cdot r}\hat{\sigma}(q,p)dp.
\end{align}
This proves the first assertion in the lemma.

The proof of the second one is similar, noting that $f_{\beta}(p,p)=\chi_{\beta}(|p|^{2}-\mu)$. We omit the details.
\end{proof}

The formulas in the previous lemma leads to the following two lemmas, relating the operators $K_{T}$ and $k_{T}$.

\begin{lemma}\label{lem:KTkT-diff}
	It holds for each $\sigma\in\cL_{h}^{2}$ that 
	\begin{align}\label{KTkT-diff}
		\|(K_{T}-k_{T})\sigma\|_{\cL_{h}^{2}}&\lesssim\Big\|\frac{\Delta_{X}}{\one-\Delta_{X}}\sigma\Big\|_{\cL_{h}^{2}}.
	\end{align}
\end{lemma}

\begin{proof}
	Using \eqref{FT-unitary} and the inequality
	\begin{align*}
		|f_{\beta}(p+hq/2,p-hq/2)-f_{\beta}(p,p)|\lesssim\min\{1,h^{2}|q|^{2}\}\lesssim\frac{h^{2}|q|^{2}}{1+h^{2}|q|^{2}},
	\end{align*}
	we obtain
	\begin{align*}
		\|(K_{T}-k_{T})\sigma\|_{\cL_{h}^{2}}^{2}\DETAILS{&=\|(f_{\beta}(p_{X}/2+p_{r},p_{X}/2-p_{r})-f_{\beta}(p_{r},p_{r}))\hat{\varphi}\|_{\widetilde{\cL}_{h}^{2}}^{2}\nonumber\\}
		&=\sum_{q\in2\pi\Z^{d}}\int_{\R^{d}}\big|f_{\beta}(p+hq/2,p-hq/2)-f_{\beta}(p,p)\big|^{2}|\hat{\sigma}(q,p)|^{2}dp\nonumber\\
		&\lesssim \sum_{q\in 2\pi\Z^{d}}\int_{\R^{d}}\Big(\frac{h^{2}|q|^{2}}{1+h^{2}|q|^{2}}\Big)^{2}|\hat{\sigma}(q,p)|^{2}dp=\Big\|\frac{\Delta_{X}}{\one-\Delta_{X}}\sigma\Big\|_{\cL_{h}^{2}}^{2},
	\end{align*}
	which completes the proof.
\end{proof}

\begin{lemma}\label{lem:kT-diff}
	Let $0<\beta\leq\beta'<\infty$ with $T=\beta^{-1}$ and $T'=\beta'^{-1}$.  Then it holds that
	\begin{align}\label{kT-diff}
		\|k_{T}-k_{T'}\|_{\cL_{h}^{2}\rightarrow\cL_{h}^{2}}&\lesssim\beta^{-1}|\beta-\beta'|,
	\end{align}
	and
	\begin{align}\label{KT-diff}
		\|K_{T}-K_{T'}\|_{\cL_{h}^{2}\rightarrow\cL_{h}^{2}}&\lesssim \beta^{-1}|\beta-\beta'|.
	\end{align}
\end{lemma}
\begin{proof}
	We denote $\delta\beta=\beta'-\beta\geq 0$.  Define the function $G_{\beta}:\R^{d}\times\R^{d}\rightarrow\R$ by
	\begin{align}
		G(x,y)&=\frac{\tanh(x)+\tanh(y)}{x+y}.
	\end{align}
	Then, by the fundamental theorem of calculus, we have for each fixed $x,y\in\R^{d}$ that 
	\begin{align*}
		0&\leq G(\beta' x,\beta' y)-G(\beta x,\beta y)=G((\beta+t\delta\beta)x,(\beta+t\delta\beta)y)\big|_{t=0}^{t=1}\nonumber\\
		&=\int_{0}^{1}\Big(\frac{d}{dt}G((\beta+t\delta\beta)x,(\beta+t\delta\beta)y)\Big) dt\nonumber\\
		&=\delta\beta\int_{0}^{1}\nabla G((\beta+t\delta\beta)x,(\beta+t\delta\beta)y)\cdot (x,y)^{\intercal}dt\nonumber\\
		&=\delta\beta\int_{0}^{1}\frac{1}{(\beta+t\delta\beta)(x+y)}\Big[\frac{x}{\cosh^{2}((\beta+t\delta\beta)x)}+\frac{y}{\cosh^{2}((\beta+t\delta\beta)y)}\nonumber\\
		&\quad\quad\quad\quad\quad-\frac{1}{\beta+t\delta\beta}G((\beta+t\delta\beta)x,(\beta'+\delta\beta)y)\Big]dt\leq \beta^{-1}\delta\beta.
	\end{align*}
	Since
	\begin{align*}
		f_{\beta}(p+hq/2,p-hq/2)=G(\beta((p+hq/2)^{2}-\mu),\beta((p-hq/2)^{2}-\mu)),
	\end{align*}
	by the same computations as in the proof of Lemma \ref{lem:KTkT-diff}, we obtain \eqref{KT-diff}.  The same method gives \eqref{kT-diff}.
\end{proof}


\subsection{The operators \texorpdfstring{$L_T$}{LT} and \texorpdfstring{$\ell_T$}{lT}}

In this subsection we study the operators
\begin{align}
	L_{T}&=\one-V^{1/2}K_{T}V^{1/2}
\end{align}
and
\begin{align}
	\ell_{T}&=\one-V^{1/2}k_{T}V^{1/2} \,.
\end{align}

Similarly as $k_T$, we consider $\ell_T$ both as an operator in $L^2_{\rm symm}(\R^d)$ and as an operator acting on even, $h^{-1}\Z^d$-periodic operators in $\cL_h^2$.

We begin by noting that by a Birman--Schwinger argument one can reformulate Assumptions (A1)--(A4) equivalently in terms of the family of operators $\ell_T$. In particular, the operator $\ell_{T_{c}}$ has a simple eigenvalue $0$. We let $\varphi_*$ denote a normalized, reflection symmetric and real eigenfunction. By simplicity of the eigenvalue, it follows that
\begin{align}\label{alpha*}
	\varphi_{*}&=\pm \frac{1}{\|V^{1/2}\alpha_{*}\|_{L^{2}}}V^{1/2}\alpha_{*}.
\end{align}
We shall choose a normalization for $\alpha_{*}$ such that $\|V^{1/2}\alpha_{*}\|_{L^{2}}=1$ and define $\varphi_{*}$ such that $\langle\varphi_*,V^{1/2}\alpha_{*}\rangle=1$.  Let 
$$
P:=\ket{\varphi_{*}}\bra{\varphi_{*}}
\quad\text{on}\ L^{2}_{\rm symm}(\R^{d}) \,.
$$
Note that Assumptions (A2) and (A4) imply that there is some $\theta>0$ such that 
\begin{align}\label{gap}
	\ell_{T_{c}}&\geq \theta P^{\perp},\quad\text{ where }P^{\perp}:=\one-P.
\end{align}
We define the projection (a cutoff in momentum space of center of mass coordinate) for a some fixed parameter $0<\kappa\leq 1$ that will be chosen later:
\begin{align*}
	\lambda_{\kappa}&:=\one(-\Delta_{X}\leq \kappa).
\end{align*}
Furthermore, we define
\begin{align}\label{P}
	P_{\kappa}&:=\lambda_{\kappa}P,\quad P_{\kappa}^{\perp}:=\one-P_{\kappa}.
\end{align}
Clearly, we can further decompose $P_{\kappa}^{\perp}=P^{\perp}+\lambda_{\kappa}^{\perp}P$, where $\lambda_{\kappa}^{\perp}=\one(-\Delta_{X}>\kappa)$.


%

\begin{prop}\label{prop:inv-bdd}
Under Assumptions (A1)--(A5) with sufficiently small $h$ and $0<h\ll\kappa\leq 1$, one has the following bounds
\begin{align}
	\label{inv-bdd1}P^{\perp}L_{T}P^{\perp}&\gtrsim P^{\perp},\\
	\label{inv-bdd2}P_{\kappa}^{\perp}L_{T}P_{\kappa}^{\perp}&\gtrsim \kappa P_{\kappa}^{\perp}.
\end{align}
\end{prop}
To prove Proposition \ref{prop:inv-bdd}, we need the following elementary lemma:

\begin{lemma}\label{lem:fT-identity}
For each $p,q\in\R^{d}$ and $T=\beta^{-1}>0$, it holds that 
\begin{align}\label{fT-identity}
	f_{\beta}(p,q)&\leq \frac{1}{2}\big(f_{\beta}(p,p)+f_{\beta}(q,q)\big).
\end{align}
\end{lemma}
\begin{proof}
It suffices to show for each $a,b\in\R$ that
\begin{align*}
	\frac{\tanh(a)+\tanh(b)}{a+b}\leq \frac{1}{2}\Big(\frac{\tanh(a)}{a}+\frac{\tanh(b)}{b}\Big).
\end{align*}
Without loss of generality, we assume $a\geq |b|$.  In this case, we have $\tanh(b)/b\geq\tanh(a)/a$, which implies that
\begin{align*}
	\frac{\tanh(a)+\tanh(b)}{a+b}&=\frac{1}{2}\Big(\frac{a\tanh(a)/a+b\tanh(b)/b}{a+b}+\frac{\tanh(a)+\tanh(b)}{a+b}\Big)\nonumber\\
	&\leq \frac{1}{2}\Big(\frac{b\tanh(a)/a+a\tanh(b)/b}{a+b}+\frac{\tanh(a)+\tanh(b)}{a+b}\Big)\nonumber\\
	&\leq \frac{1}{2}\Big(\frac{\tanh(a)}{a}+\frac{\tanh(b)}{b}\Big),
\end{align*}
which completes the proof.
\end{proof}

\begin{proof}[Proof of Proposition \ref{prop:inv-bdd}]
We shall first prove the estimates \eqref{inv-bdd1} and \eqref{inv-bdd2} at critical temperature $T=T_{c}$ and then generalize to those temperatures satisfying Assumption (A5).  

We begin with \eqref{inv-bdd1} at $T=T_{c}$.  By Lemma \ref{lem:fT-identity}, we have
\begin{align*}
	K_{T_{c}}&\leq\frac{1}{2}\chi_{\beta_{c}}(-\Delta_{x}-\mu)+\frac{1}{2}\chi_{\beta_{c}}(-\Delta_{y}-\mu).
\end{align*}
By introducing the unitary operator $U=e^{-ir\cdot(-i\nabla_{X})/2}$ and using the fact that $U$ commutes with multiplication of $V^{1/2}$, we can rewrite the r.h.s. of the above inequality as
\begin{align}
	V^{1/2}K_{T_{c}}V^{1/2}&\leq \frac{1}{2}V^{1/2}\big(U^{-1}k_{T_{c}}U+Uk_{T_{c}}U^{-1}\big)V^{1/2}\nonumber\\
	&=\frac{1}{2}U^{-1}V^{1/2}k_{T_{c}}V^{1/2}U+\frac{1}{2}UV^{1/2}k_{T_{c}}V^{1/2}U^{-1}.
\end{align}
Together with \eqref{gap}, this implies that
\begin{align}\label{LTc-lbdd}
	L_{T_{c}}&\geq \one-\frac{1}{2}\Big(U^{-1}V^{1/2}k_{T_{c}}V^{1/2}U+UV^{1/2}k_{T_{c}}V^{1/2}U^{-1}\Big)\nonumber\\
	&\geq \theta\Big(\one-\frac{1}{2}\big(U^{-1}P U+UP U^{-1}\big)\Big).
\end{align}
Consequently, we obtain
\begin{align}\label{Q0LTQ0-lbbd}
	P^{\perp}L_{T_{c}}P^{\perp}&\geq \theta\Big(P^{\perp}-\frac{1}{2}P^{\perp}\big(U^{-1}P U+UP U^{-1}\big)P^{\perp}\Big).
\end{align}

Next, we prove a lower bound for $P^{\perp}-\tfrac{1}{2}P^{\perp}\big(U^{-1}P U+UP U^{-1}\big)P^{\perp}$.  According to \cite{FrankHainzl-18}, we can write
\begin{align}\label{Apr-proj}
	\frac{1}{2}\big(U^{-1}P U+UP U^{-1}\big)&=\ket{A_{p_{X}}}\bra{A_{p_{X}}},
\end{align}
where $p_{X}=-i\nabla_{X}$ and 
\begin{align}
	A_{p}(r)&:=\varphi_{*}(r)\cos\big(p\cdot r/2\big).
\end{align}
Note that, because $p_{X}$ and $r$ commutes, we can treat $p_{X}$ as vector in $\R^{d}$ and $A_{p_{X}}$ as a $L^{2}$-function for each fixed $p_{X}$.  Since $\ket{A_{p_{X}}}\bra{A_{p_{X}}}$ is a rank-1 projection, we simply estimate using its eigenvalue and obtain
\begin{align*}
	P^{\perp}&\ket{A_{p_{X}}}\bra{A_{p_{X}}}P^{\perp}\leq \langle A_{p_{X}},P^{\perp}A_{p_{X}}\rangle\cdot P^{\perp}\nonumber\\
	&=\Big[\int_{\R^{d}}\cos^{2}\big(p_{X}\cdot r/2\big)|\varphi_{*}(r)|^{2}dr-\Big(\int_{\R^{d}}\cos\big(p_{X}\cdot r/2\big)|\varphi_{*}(r)|^{2}dr\Big)^{2}\Big]P^{\perp}.
\end{align*}
Hence, by \eqref{Q0LTQ0-lbbd}, we have
\begin{align}
	P^{\perp}L_{T_{c}}P^{\perp}&\geq cP^{\perp},
\end{align}
where
\begin{align*}
	c&=1-\sup_{q\in\R^{d}}\Big[\int_{\R^{d}}\cos^{2}\big(q\cdot r/2\big)|\varphi_{*}(r)|^{2}dr-\Big(\int_{\R^{d}}\cos\big(q\cdot r/2\big)|\varphi_{*}(r)|^{2}dr\Big)^{2}\Big].
\end{align*}
Since the term in the above bracket vanishes at $p=0$ and, by Riemann-Lebesgue lemma, it is $1/2$ for $p\rightarrow\infty$, we see that $c>0$, which gives \eqref{inv-bdd1} at $T=T_{c}$.

Now, we show \eqref{inv-bdd2} at $T=T_{c}$.  Since $[L_{T_{c}},\lambda_{\kappa}]=[P_{\kappa}^{\perp},\lambda_{\kappa}]=0$, we can decompose
\begin{align}\label{inv-bdd2-1}
	P_{\kappa}^{\perp}L_{T_{c}}P_{\kappa}^{\perp}&=\lambda_{\kappa}P_{\kappa}^{\perp}L_{T_{c}}P_{\kappa}^{\perp}\lambda_{\kappa}+\lambda_{\kappa}^{\perp}P_{\kappa}^{\perp}L_{T_{c}}P_{\kappa}^{\perp}\lambda_{\kappa}^{\perp}\nonumber\\
	&=\lambda_{\kappa}P^{\perp}L_{T_{c}}P^{\perp}\lambda_{\kappa}+P_{\kappa}^{\perp}\lambda_{\kappa}^{\perp}L_{T_{c}}\lambda_{\kappa}^{\perp}P_{\kappa}^{\perp}.
\end{align}
By \eqref{inv-bdd1} and $\kappa\leq 1$, the first component gives
\begin{align*}
	\lambda_{\kappa}P^{\perp}L_{T_{c}}P^{\perp}\lambda_{\kappa}&\gtrsim \lambda_{\kappa}P^{\perp}\lambda_{\kappa}\gtrsim \kappa \lambda_{\kappa}P^{\perp}.
\end{align*}
It remains to estimate the component $P_{\kappa}^{\perp}\lambda_{\kappa}^{\perp}L_{T_{c}}\lambda_{\kappa}^{\perp}P_{\kappa}^{\perp}$.  We observe that $\lambda_{\kappa}^{\perp}P_{\kappa}^{\perp}=\lambda_{\kappa}^{\perp}$ and, by \eqref{LTc-lbdd} and \eqref{Apr-proj}, 
\begin{align}\label{Pper-chikappa-lbdd}
	P_{\kappa}^{\perp}\lambda_{\kappa}^{\perp}L_{T_{c}}\lambda_{\kappa}^{\perp}P_{\kappa}^{\perp}&\gtrsim \lambda_{\kappa}^{\perp}(\one-\ket{A_{p_{X}}}\bra{A_{p_{X}}})\lambda_{\kappa}^{\perp}\gtrsim \big(\one-\langle A_{p_{X}},\lambda_{\kappa}^{\perp}A_{p_{X}}\rangle\big)\cdot \lambda_{\kappa}^{\perp}\nonumber\\
	&=\one(p_{X}^{2}>\kappa)\int_{\R^{d}}\sin^{2}\big(p_{X}\cdot r/2\big)|\varphi_{*}(r)|^{2}dr.
\end{align}
By the same argument in \cite[Lemma 20]{FrankHainzl-18}, we have
\begin{align}
	\int_{\R^{d}}\sin^{2}\big(p_{X}\cdot r/2\big)|\varphi_{*}(r)|^{2}dr&\gtrsim\min\{1,p_{X}^{2}\}.
\end{align}
Substituting this into \eqref{Pper-chikappa-lbdd} and using the fact that $\kappa\leq 1$ yields
\begin{align}
	P_{\kappa}^{\perp}\lambda_{\kappa}^{\perp}L_{T_{c}}\lambda_{\kappa}^{\perp}P_{\kappa}^{\perp}&\gtrsim\lambda_{\kappa}^{\perp}\min\{1,p_{X}^{2}\}\gtrsim\kappa \lambda_{\kappa}^{\perp}.
\end{align}
Consequently, we obtain
\begin{align}
	P_{\kappa}^{\perp}L_{T_{c}}P_{\kappa}^{\perp}&\gtrsim\kappa(\lambda_{\kappa}P^{\perp}+\lambda_{\kappa}^{\perp})=\kappa P_{\kappa}^{\perp},
\end{align}
which proves \eqref{inv-bdd2} at $T=T_{c}$.

Finally, we generalize these estimates to those $T$'s satisfying Assumption (A5).  By Lemma \ref{lem:kT-diff} and Assumption (A2), we have
\begin{align*}
	\big\|L_{T}-L_{T_{c}}\big\|_{\cL_{h}^{2}\rightarrow\cL_{h}^{2}}&=\big\|V^{1/2}(K_{T}-K_{T_{c}})V^{1/2}\big\|_{\cL_{h}^{2}\rightarrow\cL_{h}^{2}}\lesssim \beta-\beta_{c}\lesssim h^{2},
\end{align*}
which implies that there is some constant $C>0$ such that 
\begin{align*}
	L_{T}&\geq L_{T_{c}}-Ch^{2}.
\end{align*}
It then follows from estimates \eqref{inv-bdd1} and \eqref{inv-bdd2} at $T=T_{c}$ and the condition $0<h\ll\kappa\leq 1$ that 
\begin{align*}
	P^{\perp}L_{T}P^{\perp}&\geq P^{\perp}L_{T_{c}}P^{\perp}-Ch^{2}P^{\perp}\gtrsim (1-h^{2})P^{\perp}\gtrsim P^{\perp},
\end{align*}
and
\begin{align*}
	P_{\kappa}^{\perp}L_{T}P_{\kappa}^{\perp}&\geq P_{\kappa}^{\perp}L_{T_{c}}P_{\kappa}^{\perp}-Ch^{2}P_{\kappa}^{\perp}\gtrsim(\kappa-h^{2})P_{\kappa}^{\perp}\gtrsim\kappa P_{\kappa}^{\perp},
\end{align*}
which give \eqref{inv-bdd1} and \eqref{inv-bdd2} for those $T$ satisfying Assumption (A5) and thus completes the proof.
\end{proof}

\begin{prop}\label{prop:PLP-inv-bdd}
Under Assumptions (A1)--(A5) with sufficiently small $h$ and $0<h\ll\kappa\leq 1$, it holds that 
\begin{align}\label{PLP-inv-bdd}
	P^{\perp}(P_{\kappa}^{\perp}L_{T}P_{\kappa}^{\perp})^{-1}P^{\perp}&\lesssim \Big(\one+\kappa^{-1}\frac{-\Delta_{X}}{\one-\Delta_{X}}\Big)P^{\perp}.
\end{align}
Consequently, we obtain
\begin{align}\label{PLP-inv-bdd1}
	P^{\perp}(P_{\kappa}^{\perp}L_{T}P_{\kappa}^{\perp})^{-2}P^{\perp}\lesssim \kappa^{-1}\Big(\one+\kappa^{-1}\frac{-\Delta_{X}}{\one-\Delta_{X}}\Big)P^{\perp}.
\end{align}
\end{prop}
\begin{proof}
Again, as in the proof of Proposition \ref{prop:inv-bdd}, we first consider the case at $T=T_{c}$ and later generalize to $T$ satisfying Assumption (A5).

First, since $L_{T_{c}}\geq 0$ due to \eqref{LTc-lbdd}, by the Cauchy--Schwarz inequality, we have
\begin{align}\label{Q0LQ0-bdd}
	P^{\perp}L_{T_{c}}P^{\perp}&\leq 2\big(P_{\kappa}^{\perp}L_{T_{c}}P_{\kappa}^{\perp}+\lambda_{\kappa}^{\perp}PL_{T_{c}}P\lambda_{\kappa}^{\perp}\big).
\end{align}
Then, by Lemma \ref{lem:KTkT-diff} and the fact that $P\ell_{T_{c}}P=0$, we obtain for each $\sigma\in\cL_{h}^{2}$ that 
\begin{align}\label{Q1LQ1-bdd}
	\|\lambda_{\kappa}^{\perp}PL_{T_{c}}P\lambda_{\kappa}^{\perp}\sigma\|_{\cL_{h}^{2}}&=\|\lambda_{\kappa}^{\perp}P(L_{T_{c}}-\ell_{T_{c}})P\lambda_{\kappa}^{\perp}\sigma\|_{\cL_{h}^{2}}\leq \|V^{1/2}(K_{T_{c}}-k_{T_{c}})V^{1/2}P\lambda_{\kappa}^{\perp}\sigma\|_{\cL_{h}^{2}}\nonumber\\
	&\lesssim\Big\|\frac{\Delta_{X}}{\one-\Delta_{X}}V^{1/2}P\lambda_{\kappa}^{\perp}\sigma\Big\|_{\cL_{h}^{2}}\lesssim\Big\|\lambda_{\kappa}^{\perp}P\frac{-\Delta_{X}}{\one-\Delta_{X}}P\lambda_{\kappa}^{\perp}\sigma\Big\|_{\cL_{h}^{2}},
\end{align}
where, in the last equality, we used the fact that $[-i\nabla_{X},\lambda_{\kappa}^{\perp}P]=[-i\nabla_{X},V]=0$.  Since square root function is operator-monotone and $[-i\nabla_{X},P^{\perp}]=0$, this implies that 
\begin{align}\label{Q1LTcQ1-ubdd}
	\lambda_{\kappa}^{\perp}PL_{T_{c}}P\lambda_{\kappa}^{\perp}&\lesssim \lambda_{\kappa}^{\perp}P\Big(\frac{-\Delta_{X}}{\one-\Delta_{X}}\Big)P\lambda_{\kappa}^{\perp}\lesssim P_{\kappa}^{\perp}\Big(\frac{-\Delta_{X}}{\one-\Delta_{X}}\Big)P_{\kappa}^{\perp}.
\end{align}
Substituting \eqref{Q1LTcQ1-ubdd} into \eqref{Q0LQ0-bdd} yields
\begin{align}\label{Q0LQ0-bdd2}
	P^{\perp}L_{T_{c}}P^{\perp}&\lesssim P_{\kappa}^{\perp}\Big(L_{T_{c}}+\frac{-\Delta_{X}}{\one-\Delta_{X}}\Big)P_{\kappa}^{\perp}.
\end{align}

Next, by denoting $\cD_{\kappa}=(P_{\kappa}^{\perp}L_{T_{c}}P_{\kappa}^{\perp})^{-1/2}$ and using \eqref{inv-bdd1} and \eqref{Q0LQ0-bdd2}, we obtain
\begin{align}\label{inv-bdd1-1}
	\cD_{\kappa}P^{\perp}\cD_{\kappa}&\lesssim \cD_{\kappa}P^{\perp}L_{T_{c}}P^{\perp}\cD_{\kappa}\lesssim \cD_{\kappa}P_{\kappa}^{\perp}\Big(L_{T_{c}}+\frac{-\Delta_{X}}{\one-\Delta_{X}}\Big)P_{\kappa}^{\perp}\cD_{\kappa}\nonumber\\
	&\lesssim P_{\kappa}^{\perp}+\cD_{\kappa}\Big(\frac{-\Delta_{X}}{\one-\Delta_{X}}\Big)\cD_{\kappa}\lesssim P_{\kappa}^{\perp}+\Big(\frac{-\Delta_{X}}{\one-\Delta_{X}}\Big)^{1/2}\cD_{\kappa}^{2}\Big(\frac{-\Delta_{X}}{\one-\Delta_{X}}\Big)^{1/2}\nonumber\\
	&\lesssim \Big(\one+\kappa^{-1}\frac{-\Delta_{X}}{\one-\Delta_{X}}\Big)P_{\kappa}^{\perp}.
\end{align}
By using the fact that $\Delta_{X}$ commutes with $P_{\kappa}^{\perp}$ and $L_{T_{c}}$ and sandwiching both sides of \eqref{inv-bdd1-1} with $\Big(\one+\kappa^{-1}\frac{-\Delta_{X}}{\one-\Delta_{X}}\Big)^{-1/2}$, we obtain
\begin{align*}
	\cD_{\kappa}\Big(\one+\kappa^{-1}\frac{-\Delta_{X}}{\one-\Delta_{X}}\Big)^{-1/2}P^{\perp}\Big(\one+\kappa^{-1}\frac{-\Delta_{X}}{\one-\Delta_{X}}\Big)^{-1/2}\cD_{\kappa}\lesssim P_{\kappa}^{\perp}.
\end{align*}
Since $A^{*}A$ and $AA^{*}$ has the same non-zero spectrum for any operator $A$, by regarding $A=\cD_{\kappa}\Big(\one+\kappa^{-1}\frac{-\Delta_{X}}{\one-\Delta_{X}}\Big)^{-1/2}P^{\perp}$ as an operator from $\Ran(P^{\perp})$ to $\Ran(P_{\kappa}^{\perp})$, we obtain
\begin{align*}
	P^{\perp}&\gtrsim P^{\perp}\Big(\one+\kappa^{-1}\frac{-\Delta_{X}}{\one-\Delta_{X}}\Big)^{-1/2}\cD_{\kappa}^{2}\Big(\one+\kappa^{-1}\frac{-\Delta_{X}}{\one-\Delta_{X}}\Big)^{-1/2}P^{\perp}\nonumber\\
	&=\Big(\one+\kappa^{-1}\frac{-\Delta_{X}}{\one-\Delta_{X}}\Big)^{-1/2}P^{\perp}\cD_{\kappa}^{2}P^{\perp}\Big(\one+\kappa^{-1}\frac{-\Delta_{X}}{\one-\Delta_{X}}\Big)^{-1/2},
\end{align*}
where we have used commutivity of above operators.  Again, by sandwiching both sides with $\Big(\one+\kappa^{-1}\frac{-\Delta_{X}}{\one-\Delta_{X}}\Big)^{1/2}$, we obtain \eqref{PLP-inv-bdd} at $T=T_{c}$.  The estimate \eqref{PLP-inv-bdd1} at $T=T_{c}$ then follows from \eqref{PLP-inv-bdd} and \eqref{inv-bdd2} at $T=T_{c}$.

Now, we extend these estimates to those $T$ satisfying Assumption (A5).   Since $P_{\kappa}^{\perp}L_{T}P_{\kappa}^{\perp}$ is invertible, thanks to Proposition \ref{prop:inv-bdd}, by the second resolvent identity, we obtain
\begin{align*}
	(P_{\kappa}^{\perp}L_{T}P_{\kappa}^{\perp})^{-1}&=(P_{\kappa}^{\perp}L_{T_{c}}P_{\kappa}^{\perp})^{-1}+(P_{\kappa}^{\perp}L_{T}P_{\kappa}^{\perp})^{-1}P_{\kappa}^{\perp}(L_{T_{c}}-L_{T})P_{\kappa}^{\perp}(P_{\kappa}^{\perp}L_{T_{c}}P_{\kappa}^{\perp})^{-1}.
\end{align*}
By Assumption (A2), Lemma \ref{lem:kT-diff} and Proposition \ref{prop:inv-bdd}, the second term satisfies the following estimate:
\begin{align*}
	&\big\|(P_{\kappa}^{\perp}L_{T}P_{\kappa}^{\perp})^{-1}P_{\kappa}^{\perp}(L_{T_{c}}-L_{T})P_{\kappa}^{\perp}(P_{\kappa}^{\perp}L_{T_{c}}P_{\kappa}^{\perp})^{-1}\big\|_{\cL_{h}^{2}\rightarrow\cL_{h}^{2}}\nonumber\\
	&\leq\big\|(P_{\kappa}^{\perp}L_{T}P_{\kappa}^{\perp})^{-1}\big\|_{\cL_{h}^{2}\rightarrow\cL_{h}^{2}}\big\|(P_{\kappa}^{\perp}L_{T_{c}}P_{\kappa}^{\perp})^{-1}\big\|_{\cL_{h}^{2}\rightarrow\cL_{h}^{2}}\big\|V^{1/2}(K_{T}-K_{T_{c}})V^{1/2}\big\|_{\cL_{h}^{2}\rightarrow\cL_{h}^{2}}\nonumber\\
	&\lesssim\kappa^{-2}(\beta-\beta_{c})\lesssim \kappa^{-2}h^{2}.
\end{align*}
Together with \eqref{PLP-inv-bdd} at $T=T_{c}$, there are constants $C,C'>0$ such that
\begin{align*}
	P^{\perp}(P_{\kappa}^{\perp}L_{T}P_{\kappa}^{\perp})^{-1}P^{\perp}&\leq C\Big(\one+\kappa^{-1}\frac{-\Delta_{X}}{\one-\Delta_{X}}\Big)P^{\perp}\pm C'\kappa^{-2}h^{2}
\end{align*}
Since $0<h\ll\kappa\leq 1$, we can absorb the term $\pm C'\kappa^{-2}h^{2}$ into the constant $C$ for sufficiently small $h$, which gives \eqref{PLP-inv-bdd} for $T$ satisfying Assumption (A5).  Then \eqref{PLP-inv-bdd1} for such $T$ follows readily from \eqref{PLP-inv-bdd} and Proposition \ref{prop:inv-bdd}.
\end{proof}

\begin{lemma}\label{lem:Q0LP-bdd}
Under Assumptions (A1)--(A5) with sufficiently small $h$ and $0<h\ll\kappa\leq 1$, it holds that
\begin{align}\label{Q0LP-bdd}
	\big\|L_{T}P_{\kappa}\|_{\cL_{h}^{2}\rightarrow\cL_{h}^{2}}\lesssim\kappa,\quad\|L_{T}P_{\kappa}\|_{\cH_{h}^{1}\rightarrow\cL_{h}^{2}}\lesssim \kappa^{1/2}h.
\end{align}
\end{lemma}
\begin{proof}
Since $P_{\kappa}$ projects onto the kernel of $\ell_{T_{c}}$, we obtain
\begin{align*}
	L_{T}P_{\kappa}&=(L_{T}-\ell_{T_{c}})P_{\kappa}=(L_{T}-L_{T_{c}})P_{\kappa}+(L_{T_{c}}-\ell_{T_{c}})P_{\kappa}\nonumber\\
	&=V^{1/2}(K_{T_{c}}-K_{T})V^{1/2}P_{\kappa}+V^{1/2}(k_{T_{c}}-K_{T_{c}})V^{1/2}P_{\kappa}.
\end{align*}
Then, by Lemmas \ref{lem:KTkT-diff}, \ref{lem:kT-diff} and Assumption (A5), since $V$ commutes with $-\Delta_{X}$, we obtain for each $\sigma\in\cL_{h}^{2}$ that
\begin{align*}
	\big\|L_{T_{c}}P_{\kappa}\sigma\|_{\cL_{h}^{2}}&\leq \big\|V^{1/2}(K_{T_{c}}-K_{T})V^{1/2}P_{\kappa}\sigma\big\|_{\cL_{h}^{2}}+\big\|V^{1/2}(k_{T_{c}}-K_{T_{c}})V^{1/2}P_{\kappa}\sigma\big\|_{\cL_{h}^{2}}\nonumber\\
	&\lesssim(\beta-\beta_{c})\big\|V^{1/2}P_{\kappa}\sigma\big\|_{\cL_{h}^{2}}+\Big\|\frac{\Delta_{X}}{\one-\Delta_{X}}V^{1/2}P_{\kappa}\sigma\Big\|_{\cL_{h}^{2}}\nonumber\\
	&\lesssim h^{2}\|\sigma\|_{\cL_{h}^{2}}+\Big\|\frac{\Delta_{X}}{\one-\Delta_{X}}\lambda_{\kappa}P_{\kappa}\sigma\Big\|_{\cL_{h}^{2}}\lesssim(\kappa+h^{2})\|\sigma\|_{\cL_{h}^{2}}.
\end{align*}
Since $0<h\ll\kappa\leq 1$, we can absorb $h^{2}$ into $\kappa$ and thus gives the first estimate in \eqref{Q0LP-bdd}.

Similarly, together with the fact that $V$ and $P_{\kappa}$ commute with $-\Delta_{X}$, the same argument gives us for each $\sigma\in\cH_{h}^{1}$ that
\begin{align}
	\|L_{T_{c}}P_{\kappa}\sigma\|_{\cL_{h}^{2}}&\lesssim h^{2}\|\sigma\|_{\cL_{h}^{2}}+\Big\|\frac{\Delta_{X}}{\one-\Delta_{X}}\lambda_{\kappa}P_{\kappa}\sigma\Big\|_{\cL_{h}^{2}}\nonumber\\
	&\lesssim h^{2}\|\sigma\|_{\cL_{h}^{2}}+\kappa^{1/2}\Big\|\sqrt{\frac{-\Delta_{X}}{\one-\Delta_{X}}}\sigma\Big\|_{\cL_{h}^{2}}\lesssim(h^{2}+\kappa^{1/2}h)\|\sigma\|_{\cH_{h}^{1}}.
\end{align}
Again, we can absorb $h^{2}$ into $\kappa^{1/2}h$, which gives the second relation in \eqref{Q0LP-bdd}.
\end{proof}

\begin{prop}\label{prop:inv-bdd-LP}
Under Assumptions (A1)--(A5) with sufficiently small $h$ and $0<h\ll\kappa\leq 1$, it holds that
\begin{align}
	\label{inv-bdd-LP1}\|(P_{\kappa}^{\perp}L_{T}P_{\kappa}^{\perp})^{-1}L_{T}P_{\kappa}\|_{\cH_{h}^{1}\rightarrow\cL_{h}^{2}}&\lesssim h,\\
	\label{inv-bdd-LP2}\|P^{\perp}(P_{\kappa}^{\perp}L_{T}P_{\kappa}^{\perp})^{-1}L_{T}P_{\kappa}\|_{\cH_{h}^{1}\rightarrow\cL_{h}^{2}}&\lesssim\kappa^{1/2}h,\\
	\label{inv-bdd-LP3}\|P_{\kappa}L_{T}(P_{\kappa}^{\perp}L_{T}P_{\kappa}^{\perp})^{-1}L_{T}P_{\kappa}\|_{\cH_{h}^{1}\rightarrow\cL_{h}^{2}}&\lesssim\kappa^{3/2}h.
\end{align}
\end{prop}
\begin{proof}
First, since $L_{T}$ commutes with $\lambda_{\kappa}$, we obtain
\begin{align*}
	A:=(P_{\kappa}^{\perp}L_{T}P_{\kappa}^{\perp})^{-1}L_{T}P_{\kappa}&=(P_{\kappa}^{\perp}L_{T}P_{\kappa}^{\perp})^{-1}P_{\kappa}^{\perp}L_{T}P_{\kappa}=(P_{\kappa}^{\perp}L_{T}P_{\kappa}^{\perp})^{-1}P^{\perp}L_{T}P_{\kappa}.
\end{align*}
Then the bound \eqref{inv-bdd-LP1} follows from Proposition \ref{prop:PLP-inv-bdd} and Lemma \ref{lem:Q0LP-bdd}.  Indeed, since
\begin{align*}
	\lambda_{\kappa}\Big(\one+\kappa^{-1}\frac{-\Delta_{X}}{\one-\Delta_{X}}\Big)\lambda_{\kappa}&\leq \one,
\end{align*}
we obtain for each $\sigma\in\cH_{h}^{1}$ that
\begin{align*}
	\big\|A\sigma\|_{\cL_{h}^{2}}^{2}&=\langle\sigma,P_{\kappa}L_{T}P^{\perp}(P_{\kappa}^{\perp}L_{T}P_{\kappa}^{\perp})^{-2}P^{\perp}L_{T}P_{\kappa}\sigma\rangle_{\cL_{h}^{2}}\nonumber\\
	&\lesssim \kappa^{-1} \big\langle\sigma,P_{\kappa}L_{T}P^{\perp}\lambda_{\kappa}\Big(\one+\kappa^{-1}\frac{-\Delta_{X}}{\one-\Delta_{X}}\Big)\lambda_{\kappa}P^{\perp}L_{T}P_{\kappa}\sigma\big\rangle_{\cL_{h}^{2}}\nonumber\\
	&\lesssim\kappa^{-1}\|P^{\perp}L_{T_{c}}P_{\kappa}\sigma\|_{\cL_{h}^{2}}^{2}\lesssim \kappa^{-1}\|P^{\perp}L_{T}P_{\kappa}\|_{\cH_{h}^{1}\rightarrow\cL_{h}^{2}}^{2}\|\sigma\|_{\cH_{h}^{1}}^{2}\lesssim h^{2}\|\sigma\|_{\cH_{h}^{1}}^{2}.
\end{align*}
Similarly, using the same argument, we have
\begin{align*}
	P^{\perp}A&=P^{\perp}(P_{\kappa}^{\perp}L_{T}P_{\kappa}^{\perp})^{-1}L_{T}P_{\kappa}=P^{\perp}(P_{\kappa}^{\perp}L_{T}P_{\kappa}^{\perp})^{-1}P^{\perp}L_{T}P_{\kappa},
\end{align*}
which gives
\begin{align*}
	\big\|P^{\perp}A\sigma\big\|_{\cL_{h}^{2}}^{2}&=\langle\sigma,P_{\kappa}L_{T}P^{\perp}(P^{\perp}(P_{\kappa}^{\perp}L_{T}P_{\kappa}^{\perp})^{-1}P^{\perp})^{2}P^{\perp}L_{T}P_{\kappa}\sigma\rangle_{\cL_{h}^{2}}\nonumber\\
	&\lesssim\big\langle\sigma, P_{\kappa}L_{T}P^{\perp}\lambda_{\kappa}\Big(\one+\kappa^{-1}\frac{-\Delta_{X}}{\one-\Delta_{X}}\Big)\lambda_{\kappa}P^{\perp}L_{T_{c}}P_{\kappa}\sigma\big\rangle_{\cL_{h}^{2}}\nonumber\\
	&\lesssim\|P^{\perp}L_{T_{c}}P_{\kappa}\sigma\|_{\cL_{h}^{2}}^{2}\lesssim\|P^{\perp}L_{T_{c}}P_{\kappa}\|_{\cH_{h}^{1}\rightarrow\cL_{h}^{2}}^{2}\|\sigma\|_{\cH_{h}^{1}}^{2}\lesssim\kappa h^{2}\|\sigma\|_{\cH_{h}^{1}}^{2}.
\end{align*}
and 
\begin{align*}
	\big\|P_{\kappa}L_{T}A\sigma\big\|_{\cL_{h}^{2}}^{2}&=\langle P^{\perp}A\sigma,(P^{\perp}L_{T}P_{\kappa})^{*}(P^{\perp}L_{T}P_{\kappa})P^{\perp}A\sigma\rangle_{\cL_{h}^{2}}\nonumber\\
	&\leq \|P^{\perp}L_{T}P_{\kappa}\|_{\cL_{h}^{2}\rightarrow\cL_{h}^{2}}^{2}\|P^{\perp}A\sigma\|_{\cL_{h}^{2}}^{2}\lesssim\kappa^{3}h^{2}\|\sigma\|_{\cH_{h}^{1}}^{2}.
\end{align*}
These estimates are precisely \eqref{inv-bdd-LP2} and \eqref{inv-bdd-LP3}, which completes the proof.
\end{proof}

\section{Nonlinear analysis}\label{sec:nonlinear}

In this section, we study the nonlinear operator $N_{T}$. Recall from \eqref{nonlinear} that, for each $\sigma\in\cL_{h}^{2}$, 
\begin{align*}
	N_{T}(\sigma)&=\int_{\cC}\frac{dz}{2\pi i} \, \rho(\beta z)(z-\fh)^{-1}\sigma(z+\overline{\fh})^{-1}\overline{\sigma}[(z-H(\sigma)^{-1})]_{12}\nonumber\\
	&=\frac{2}{\beta}\sum_{n\in\Z}(i\omega_{n}-\fh)^{-1}\sigma(i\omega_{n}+\overline{\fh})^{-1}\overline{\sigma}[(z-H(\sigma))^{-1}]_{12}
\end{align*}
where $\rho(z)=\tanh(z/2)$, $\cC=\{r\pm i\pi/(2\beta_{c})\mid r\in\R\}$ is oriented positively and $ \omega_{n}:=\pi(2n+1)/\beta$ are poles of $\rho(\beta z)$.


\subsection{Mapping properties of \texorpdfstring{$N_T$}{NT}}

The main result of this subsection is Proposition \ref{prop:BCS-NL-norm}, which concerns the norm of $N_T$ as a map from $\cL^6_h$ to $\cH^s_h$. The proof will require several preliminary results.

By a straightforward computation, we obtain the following lemma.

\begin{lemma}\label{lem:inv-blockij}
Let $z\in\C$ with $\Im(z)\neq 0$.  It holds for any $\sigma\in\cL_{h}^{2}$ that 
\begin{align}\label{inv-block12}
	[(z-H(\sigma))^{-1}]_{12}&=(z-\fh)^{-1}\sigma [(z-H(\sigma))^{-1}]_{22}\nonumber\\
	&=[(z-H(\sigma))^{-1}]_{11}\sigma(z+\overline{\fh})^{-1},\\
	\label{inv-block22}[(z-H(\sigma))^{-1}]_{22}&=(z+\overline{\fh})^{-1}+(z+\overline{\fh})^{-1}\overline{\sigma}[(z-H(\sigma))^{-1}]_{12},\\
	\label{inv-block11}[(z-H(\sigma))^{-1}]_{11}&=(z-\fh)^{-1}+[(z-H(\sigma))^{-1}]_{12}\sigma(z-\fh)^{-1}.
\end{align}
Moreover, we have for each $i,j=1,2$ that
\begin{align}\label{inv-blockij-norm}
	\|[(z-H(\sigma))^{-1}]_{ij}\|&\leq |\Im(z)|^{-1}.
\end{align} 
\end{lemma}
\begin{proof}
First, let $A_{ij}:=[(z-H(\sigma))^{-1}]_{ij}$ be the $(i,j)$-th entry of $(z-H(\sigma))^{-1}$.  Then we compute
\begin{align*}
	\begin{pmatrix}
		\one	&	0\\
		0	&	\one
	\end{pmatrix}&=(z-H(\sigma))(z-H(\sigma))^{-1}=\begin{pmatrix}
		z-\fh	&	-\sigma\\
		-\overline{\sigma}	&	z+\overline{\fh}
	\end{pmatrix}\begin{pmatrix}
		A_{11}	&	A_{12}\\
		A_{21}	&	A_{22}
	\end{pmatrix}\nonumber\\
	&=\begin{pmatrix}
		(z-\fh)A_{11}-\sigma A_{12}	&	(z-\fh)A_{12}-\sigma A_{22}\\
		-\overline{\sigma}A_{11}+(z+\overline{\fh})A_{21}	&	-\overline{\sigma}A_{12}+(z+\overline{\fh})A_{22}
	\end{pmatrix}.
\end{align*}
Hence, we have
\begin{align*}
	A_{12}=(z-\fh)^{-1}\sigma A_{22},\quad A_{22}=(z+\overline{\fh})^{-1}+(z+\overline{\fh})^{-1}\overline{\sigma}A_{12},
\end{align*}
This proves \eqref{inv-block22} and the first equality in \eqref{inv-block12}.  Similarly, we compute
\begin{align*}
	\begin{pmatrix}
		\one	&	0\\
		0	&	\one
	\end{pmatrix}&=(z-H(\sigma))^{-1}(z-H(\sigma))=\begin{pmatrix}
		A_{11}	&	A_{12}\\
		A_{21}	&	A_{22}
	\end{pmatrix}\begin{pmatrix}
	z-\fh	&	-\sigma\\
	-\overline{\sigma}	&	z+\overline{\fh}
	\end{pmatrix}\nonumber\\
	&=\begin{pmatrix}
		A_{11}(z-\fh)-A_{12}\sigma	&	-A_{11}\sigma+A_{12}(z+\overline{\fh})\\
		A_{21}(z-\fh)-A_{22}\overline{\sigma}	&	-A_{21}\sigma+A_{22}(z+\overline{\fh})
	\end{pmatrix}.
\end{align*}
Thus, we obtain
\begin{align*}
	A_{12}&=A_{11}\sigma(z+\overline{\fh})^{-1},\quad A_{11}=(z-\fh)^{-1}+A_{12}\sigma(z-\fh)^{-1}.
\end{align*}
which proves \eqref{inv-block11} and the second equality in \eqref{inv-block12}.

Next, since $H(\sigma)$ is self-adjoint for each $\sigma\in\cL_{h}^{2}$, the $2\times 2$ matrix $(z-H(\sigma))^{-1}$ has operator norm  bounded by $|\Im(z)|^{-1}$ on $L^{2}(\R^{d})\oplus L^{2}(\R^{d})$.  Moreover, we observe that, for each $\psi\in L^{2}(\R^{d})$,
\begin{align*}
	\|A_{12}\psi\|_{L^{2}}&=\|P_{12}(z-H(\sigma))^{-1}\Psi\|_{L^{2}(\R^{d})\oplus L^{2}(\R^{d})},
\end{align*}
where $\Psi=(\psi,0)\in L^{2}(\R^{d})\oplus L^{2}(\R^{d})$ and $P_{12}=\begin{pmatrix}
	0	&	1\\
	0	&	0
\end{pmatrix}$.  Since $\|P\|=1$, this implies that 
\begin{align*}
	\|A_{12}\psi\|_{L^{2}}&\leq \|P_{12}(z-H(\sigma))^{-1}\Psi\|_{L^{2}(\R^{d})\oplus L^{2}(\R^{d})}\nonumber\\
	&\leq \|P_{12}(z-H(\sigma))^{-1}\|\|\Psi\|_{L^{2}(\R^{d})\oplus L^{2}(\R^{d})}\nonumber\\
	&\leq \|(z-H(\sigma))^{-1}\|\|\psi\|_{L^{2}}\leq |\Im(z)|^{-1}\|\psi\|_{L^{2}}.
\end{align*}
This proves $\|[(z-H(\sigma))^{-1}]_{12}\|\leq |\Im(z)|^{-1}$.  For other pairs of $(i,j)$'s, we can prove them using the same argument with appropriate chosen $P$ and $\Psi$; hence, we omit for simplicity.  This completes the proof.
\end{proof}

For the the following lemma, we recall the definition of the operator $-\Delta_X$ in the Hilbert space $\mathcal L^2_h$ from Subsection \ref{sec:preliminary}. Also, for $z\in\C$ and $s\in\R$, we define
\begin{align*}
	F_{s}(z):=|\Im(z)|^{-2}(1+|z|)^{s/2} \,.
\end{align*}

\begin{lemma}\label{lem:Hs-resol-norm}
Let $z\in\C\setminus\R$ and $1\leq p\leq \infty$. Then it holds for each $\sigma\in\cL_{h}^{p}$, $\delta>0$ and $s=0,2$ that
\begin{align}\label{resolvent-Laplacian}
	\big\|(\one-\delta^{2}\Delta_{X})^{s/2}\big((z-\fh)^{-1}\sigma(z+\overline{\fh})^{-1}\big)\big\|_{\cL_{h}^{p}}&\lesssim (1+\delta)^{s}F_{s}(z)\|\sigma\|_{\cL_{h}^{p}}.
\end{align}
Moreover, when $p=2$, we have
\begin{align}\label{resolvent-Laplacian2}
	\big\|(\one-\delta^{2}\Delta_{X})^{1/2}\big((z-\fh)^{-1}\sigma(z+\overline{\fh})^{-1}\big)\big\|_{\cL_{h}^{2}}&\lesssim (1+\delta)F_{1}(z)\|\sigma\|_{\cL_{h}^{2}}.
\end{align}
\end{lemma}
\begin{proof}
	First, we recall that $\fh=\overline{\fh}=-\Delta-\mu$ so that, by a direct computation,
	\begin{align*}
		-\Delta(z-\fh)^{-1}&=(z-\fh)^{-1}(-\Delta)=-\one-(z+\mu)(z-\fh)^{-1},\\
		-\Delta(z+\overline{\fh})^{-1}&=(z+\overline{\fh})^{-1}(-\Delta)=\one+(-z+\mu)(z+\overline{\fh})^{-1}.
	\end{align*}
	which gives immediately that
	\begin{align}\label{resol-Lap-esti}
		\big\|\Delta(z-\fh)^{-1}\big\|+\big\|\Delta(z+\overline{\fh})^{-1}\big\|&\leq 2\big(1+(|z|+|\mu|)|\Im(z)|^{-1}\big).
	\end{align}
	
	Next, for each $\sigma\in\cL_{h}^{p}$ with $1\leq p\leq \infty$, we introduce
	\begin{align*}
		\widetilde{R}_{z}(\sigma)&:=(z-\fh)^{-1}\sigma(z+\overline{\fh})^{-1},
	\end{align*}
	\begin{align*}
		\widetilde{R}_{z}^{(0)}(\sigma)&:=\big(-\Delta(z-\fh)^{-1}\big)\sigma(z+\overline{\fh})^{-1}+(z-\fh)^{-1}\sigma\big(-\Delta(z+\overline{\fh})^{-1}\big)
	\end{align*}
	and
	\begin{align*}
		\widetilde{R}_{z}^{(1)}(\sigma)&:=\sum_{j=1}^{d}\big(p_{j}(z-\fh)^{-1}\big)\sigma\big(p_{j}(z+\overline{\fh})^{-1}\big),
	\end{align*}
	where $p_{j}=-i\partial_{x_{j}}$ is the differential operator acting on $j$-th variable.  By the estimate \eqref{resol-Lap-esti} and H\"older's inequality, we have for each $\sigma\in\cL_{h}^{p}$ that
	\begin{align}\label{Rz-esti}
		\big\|\widetilde{R}_{z}(\sigma)\big\|_{\cL_{h}^{p}}&\leq \|\sigma\|_{\cL_{h}^{p}}\big\|(z-\fh)^{-1}\big\|\big\|(z+\overline{\fh})^{-1}\big\|\nonumber\\
		&\leq \|\sigma\|_{\cL_{h}^{p}}|\Im(z)|^{-2}=\|\sigma\|_{\cL_{h}^{p}}F_{0}(z)
	\end{align}
	and
	\begin{align}\label{Rz0-esti}
		\big\|\widetilde{R}_{z}^{(0)}(\sigma)\big\|_{\cL_{h}^{p}}&\leq \|\sigma\|_{\cL_{h}^{p}}\Big(\big\|\Delta(z-\fh)^{-1}\big\|\big\|(z+\overline{\fh})^{-1}\big\|+\big\|(z-\fh)^{-1}\big\|\big\|\Delta(z+\overline{\fh})^{-1}\big\|\Big)\nonumber\\
		&\leq \|\sigma\|_{\cL_{h}^{p}}|\Im(z)|^{-1}\Big(\big\|\Delta(z-\fh)^{-1}\big\|+\big\|\Delta(z+\overline{\fh})^{-1}\big\|\Big)\nonumber\\
		&\leq 2\|\sigma\|_{\cL_{h}^{p}} |\Im(z)|^{-2}\big(|\Im(z)|+|z|+|\mu|\big) \nonumber\\
		&\lesssim \|\sigma\|_{\cL_{h}^{p}}|\Im(z)|^{-2}(1+|z|)=\|\sigma\|_{\cL_{h}^{p}}F_{2}(z).
	\end{align}
	Similarly, together with the Cauchy--Schwarz inequality, we obtain
	\begin{align}\label{Rz1-esti}
		\big\|\widetilde{R}_{z}^{(1)}(\sigma)\big\|_{\cL_{h}^{p}}&\leq \|\sigma\|_{\cL_{h}^{p}}\sum_{j=1}^{d}\big\|p_{j}(z-\fh)^{-1}\big\|\big\|p_{j}(z+\overline{\fh})^{-1}\big\|\nonumber\\
		&\leq d\|\sigma\|_{\cL_{h}^{p}}\big\|\sqrt{-\Delta}(z-\fh)^{-1}\big\|\big\|\sqrt{-\Delta}(z+\overline{\fh})^{-1}\big\|\nonumber\\
		&\leq d\|\sigma\|_{\cL_{h}^{p}}\big\|(\bar{z}-\fh)^{-1}(-\Delta)(z-\fh)^{-1}\big\|^{1/2}\big\|(\bar{z}+\overline{\fh})^{-1}(-\Delta)(z+\overline{\fh})^{-1}\big\|^{1/2}\nonumber\\
		&\leq d\|\sigma\|_{\cL_{h}^{p}}\Big(\big\|(z-\fh)^{-1}\big\|\big\|(z+\overline{\fh})^{-1}\big\|\big\|\Delta(z-\fh)^{-1}\big\|\big\|\Delta(z+\overline{\fh})^{-1}\big\|\Big)^{1/2}\nonumber\\
		&\leq 2d\|\sigma\|_{\cL_{h}^{p}}|\Im(z)|^{-1}\big(1+(|z|+|\mu|)|\Im(z)|^{-1}\big)\nonumber\\
		&\lesssim \|\sigma\|_{\cL_{h}^{p}}|\Im(z)|^{-2}(1+|z|)=\|\sigma\|_{\cL_{h}^{p}}F_{2}(z).
	\end{align}
	Here, in the second line, we have used the following operator inequality for each $j=1,...,d$:
	\begin{align*}
		\big((z-\fh)^{-1}\big)^{*}p_{j}^{2}(z-\fh)^{-1}&\leq \big((z-\fh)^{-1}\big)^{*}(-\Delta)(z-\fh)^{-1}\nonumber\\
		&=\big(\sqrt{-\Delta}(z-\fh)^{-1}\big)^{*}\big(\sqrt{-\Delta}(z-\fh)^{-1}\big),
	\end{align*}
	which implies that 
	\begin{align*}
		\big\|p_{j}(z-\fh)^{-1}\big\|&\leq \big\|\sqrt{-\Delta}(z-\fh)^{-1}\big\|.
	\end{align*}
	The same estimate holds if we replace $\fh$ with $-\overline{\fh}$.
	
	Now, we prove \eqref{resolvent-Laplacian}.  Note that, by definition of the operator $-\Delta_X$,
	\begin{align}\label{Lap-conj-resolv}
		-\Delta_{X}\big(\widetilde{R}_{z}(\sigma)\big)\DETAILS{&=\big(-\Delta(z-\fh)^{-1}\big)\sigma(z+\overline{\fh})^{-1}+(z-\fh)^{-1}\sigma\big(-\Delta(z+\overline{\fh})^{-1}\big)\nonumber\\
		&\quad\quad+2\sum_{j=1}^{d}\big(p_{j}(z-\fh)^{-1}\big)\sigma\big(p_{j}(z+\overline{\fh})^{-1}\big)}=\widetilde{R}_{z}^{(0)}(\sigma)+2\widetilde{R}_{z}^{(1)}(\sigma) \,.
	\end{align}
	Then, for each $\delta>0$, we obtain from H\"older inequality, \eqref{Rz-esti}--\eqref{Rz1-esti} that
	\begin{align*}
		\big\|(\one-\delta^{2}\Delta_{X})\widetilde{R}_{z}(\sigma)\big\|_{\cL_{h}^{p}}&\leq \big\|\widetilde{R}_{z}(\sigma)\big\|_{\cL_{h}^{p}}+\delta^{2}\big\|\Delta_{X}(\widetilde{R}_{z}(\sigma))\big\|_{\cL_{h}^{p}}\nonumber\\
		&\leq \big\|\widetilde{R}_{z}(\sigma)\big\|_{\cL_{h}^{p}}+\delta^{2}\big(\big\|\widetilde{R}_{z}^{(0)}(\sigma)\big\|_{\cL_{h}^{p}}+2\big\|\widetilde{R}_{z}^{(1)}(\sigma)\big\|_{\cL_{h}^{p}}\big)\nonumber\\
		&\lesssim\big(F_{0}(z)+\delta^{2}F_{2}(z)\big)\|\sigma\|_{\cL_{h}^{p}}\lesssim (1+\delta)^{2}F_{2}(z)\|\sigma\|_{\cL_{h}^{p}}.
	\end{align*}
	This proves \eqref{resolvent-Laplacian} for $s=2$.  The case $s=0$ follows directly from \eqref{Rz-esti}. 
	
	Finally, for \eqref{resolvent-Laplacian2}, by the Cauchy--Schwarz inequality and \eqref{resolvent-Laplacian}, we obtain for $p=2$ that
	\begin{align*}
		\big\|(\one-\delta^{2}\Delta_{X})^{1/2}(\widetilde{R}_{z}(\sigma))\big\|_{\cL_{h}^{2}}&=\sqrt{\big\langle \widetilde{R}_{z}(\sigma),(\one-\delta^{2}\Delta_{X})(\widetilde{R}_{z}(\sigma))\big\rangle_{\cL_{h}^{2}}}\nonumber\\
		&\leq \big\|\widetilde{R}_{z}(\sigma)\big\|_{\cL_{h}^{2}}^{1/2}\big\|(\one-\delta^{2}\Delta_{X})(\widetilde{R}_{z}(\sigma))\big\|_{\cL_{h}^{2}}^{1/2}\nonumber\\
		&\lesssim (1+\delta)\big(F_{0}(z)F_{2}(z)\big)^{1/2}\|\sigma\|_{\cL_{h}^{2}}=(1+\delta)F_{1}(z)\|\sigma\|_{\cL_{h}^{2}}.
	\end{align*}
	This completes the proof.
\end{proof}

Together with boundedness of $L_{T}$, these lemmas lead to the following proposition:
\begin{prop}\label{prop:BCS-NL-norm}
If $\sigma\in\cL_{h}^{2}\cap\cL_{h}^{6}$, then it holds for $s=0,1,2$ that
\begin{align}\label{nonlinear-norm}
	\|N_{T}(\sigma)\|_{\cH_{h}^{s}}\lesssim h^{-s}\|\sigma\|_{\cL_{h}^{6}}^{3}.
\end{align}
\end{prop}
\begin{proof}
First, by Lemma \ref{lem:inv-blockij} and H\"older inequality \eqref{Holder-periodic}, we obtain from the first equality in \eqref{inv-block12} that
\begin{align*}
	\|N_{T}(\sigma)\|_{\cL_{h}^{2}} &\leq \frac{2}{\beta}\sum_{n\in\Z}\|(i\omega_{n}-\fh)^{-1}\|\|\sigma(i\omega_{n}+\overline{\fh})^{-1}\overline{\sigma}(i\omega_{n}-\fh)^{-1}\sigma[(i\omega_{n}-H(\sigma))^{-1}]_{22}\|_{\cL_{h}^{2}}\nonumber\\
	&\leq \frac{2}{\beta}\|\sigma\|_{\cL_{h}^{6}}^{3}\sum_{n\in\Z}\|(i\omega_{n}-\fh)^{-1}\|\|(i\omega_{n}+\overline{\fh})^{-1}\|\|[(i\omega_{n}-H(\sigma))^{-1}]_{22}\|\nonumber\\
	&\lesssim\|\sigma\|_{\cL_{h}^{6}}^{3}\sum_{n\in\Z}|\omega_{n}|^{-3}\lesssim\|\sigma\|_{\cL_{h}^{6}}^{3},
\end{align*}
where we recall $\omega_{n}=\pi(2n+1)/\beta$.  This proves estimate \eqref{nonlinear-norm} for $s=0$. 

Next, together with Lemma \ref{lem:Hs-resol-norm} and the fact that $F_{1}(i\omega_{n})\lesssim |\omega_{n}|^{-1}$ for each $n\in\Z$, the same argument above and the second equality in \eqref{inv-block12} gives
\begin{align*}
	\|N_{T}(\sigma)\|_{\cH_{h}^{1}}&=\big\|(\one-h^{-2}\Delta_{X})^{1/2}N_{T}(\sigma)\big\|_{\cL_{h}^{2}}\nonumber\\
	&\lesssim \frac{2}{\beta}\sum_{n\in\Z}\big\|(\one-h^{-2}\Delta_{X})^{1/2}\big((i\omega_{n}-\fh)^{-1}\sigma(i\omega_{n}+\overline{\fh})^{-1}\overline{\sigma}\nonumber\\
	&\quad\quad\quad\quad\quad\quad\quad\quad\quad\quad\quad\times[(i\omega_{n}-H(\sigma))^{-1}]_{11}\sigma(i\omega_{n}+\overline{\fh})^{-1}\big)\big\|_{\cL_{h}^{2}}\nonumber\\
	&\lesssim h^{-1}\sum_{n\in\Z}F_{1}(i\omega_{n})\big\|\sigma(i\omega_{n}+\overline{\fh})^{-1}\overline{\sigma}[(i\omega_{n}-H(\sigma))^{-1}]_{11}\sigma\big\|_{\cL_{h}^{2}}\nonumber\\
	&\lesssim h^{-1}\|\sigma\|_{\cL_{h}^{6}}^{3}\sum_{n\in\Z}|\omega_{n}|^{-1}\|(i\omega_{n}+\overline{\fh})^{-1}\|\|[(i\omega_{n}-H(\sigma))^{-1}]_{11}\|\lesssim h^{-1}\|\sigma\|_{\cL_{h}^{6}}^{3}.
\end{align*}
This proves estimate \eqref{nonlinear-norm} for $s=1$.  The case $s=2$ is done in the same way; hence, we omit details for simplicities.  
%
\end{proof}


\subsection{Decomposition of the nonlinearity}

Lemma \ref{lem:inv-blockij} motivates the following decomposition of the nonlinear operator $N_{T}$.  We define, for each $\sigma\in\cL_{h}^{2}$ and $z\in\C$ with $\Im(z)\neq 0$,
\begin{align}
	R_{z}(\sigma)& :=(z-\fh)^{-1}\sigma(z+\overline{\fh})^{-1}\overline{\sigma}(z-\fh)^{-1}\sigma(z+\overline{\fh})^{-1}.
\end{align}
Then we can write $N_{T}(\sigma)=N_{T}'(\sigma)+\widetilde{N}_{T}(\sigma)$ with
\begin{align}\label{NL-leading}
	N_{T}'(\sigma)& := \frac{2}{\beta} \sum_{n\in\Z} R_{i\omega_n}(\sigma)
\end{align}
and
\begin{align}\label{NL-higher}
	\widetilde{N}_{T}(\sigma)& := \frac{2}{\beta} \sum_{n\in\Z} R_{i\omega_n}(\sigma)\overline{\sigma}[(i\omega_n-H(\sigma))^{-1}]_{12} \,.
\end{align}

The rest of this section is devoted to estimates on the norms of $N_{T}'(\sigma),\widetilde{N}_{T}(\sigma)$ and the difference $N_{T}'(\sigma_{1})-N_{T}'(\sigma_{2})$.  In the next proposition, we estimate norms of $N_{T}'(\sigma)$ and $\widetilde{N}_{T}(\sigma)$.

\begin{prop}\label{prop:NL-norm}
For each $\sigma\in\cL_{h}^{2}\cap\cL_{h}^{6}$ such that $(\one-\Delta_{X})^{-1}\sigma\in \cL_{h}^{\infty}$, it holds for $s=0,1$ that
\begin{align}
	\label{Hs-NL}
    &\|N_{T}'(\sigma)\|_{\cH_{h}^{s}}\lesssim h^{-s}\|\sigma\|_{\cL_{h}^{6}}^{3},\quad\|\widetilde{N}_{T}(\sigma)\|_{\cL_{h}^{2}}\lesssim\|(\one-\Delta_{X})^{-1}\sigma\|\|\sigma\|_{\cL_{h}^{6}}^{3}. 
\end{align}
\end{prop}

To prove this, we first prove the following lemma:

\begin{lemma}\label{lem:resolvent-norms}
Let $z\in\C\setminus\R$ and $\sigma\in\cL_{h}^{2}$.  Then it holds for $s=0,1,2$ that
\begin{align}
	\label{Hs-resol}&\|R_{z}(\sigma)\|_{\cH_{h}^{s}}\lesssim h^{-s}|\Im(z)|^{-2}F_{s}(z)\|\sigma\|_{\cL_{h}^{6}}^{3},
\end{align}
where we recall $F_{s}(z)$ from Lemma \ref{lem:Hs-resol-norm}.
\end{lemma}

\begin{proof}
By H\"older inequality \eqref{Holder-periodic} and Lemma \ref{lem:Hs-resol-norm} (with $\delta=h^{-1}$), we obtain
\begin{align}
	\|R_{z}(\sigma)\|_{\cH_{h}^{s}}&=\|(\one-h^{-2}\Delta_{X})^{s/2} \left( (z-\fh)^{-1}\sigma(z+\overline{\fh})^{-1}\overline{\sigma}(z-\fh)^{-1}\sigma(z+\overline{\fh})^{-1} \right) \|_{\cL_{h}^{2}}\nonumber\\
	&\lesssim h^{-s}F_{s}(z)\|\sigma(z+\overline{\fh})^{-1}\overline{\sigma}(z-\fh)^{-1}\sigma\|_{\cL_{h}^{2}}\nonumber\\
	&\lesssim h^{-s}F_{s}(z)\|(z-\fh)^{-1}\|\|(z+\overline{\fh})^{-1}\|\|\sigma\|_{\cL_{h}^{6}}^{3}\nonumber\\
	&\lesssim h^{-s}|\Im(z)|^{-2}F_{s}(z)\|\sigma\|_{\cL_{h}^{6}}^{3},
\end{align}
which proves \eqref{Hs-resol}.
\end{proof}

\begin{proof}[Proof of Proposition \ref{prop:NL-norm}]
First, for the $\cH_{h}^{s}$-norm of $N_{T}'(\sigma)$ with $s=0,1$, by Lemma \ref{lem:resolvent-norms} and the fact that $F_{s}(i\omega_{n})\lesssim|\omega_{n}|^{-2+s/2}$ for each $n\in\Z$, we have
\begin{align*}
	\|N_{T}'(\sigma)\|_{\cH_{h}^{s}}&\leq\frac{2}{\beta}\sum_{n\in\Z}\|R_{i\omega_{n}}(\sigma)\|_{\cH_{h}^{s}}\lesssim h^{-s}\sum_{n\in\Z}|\omega_{n}|^{-2}F_{s}(i\omega_{n})\|\sigma\|_{\cL_{h}^{6}}^{3}\nonumber\\
    &\lesssim h^{-s}\|\sigma\|_{\cL_{h}^{6}}^{3}\sum_{n\in\Z}|\omega_{n}|^{-4+s/2}\lesssim h^{-s}\|\sigma\|_{\cL_{h}^{6}}^{3}.
\end{align*}
These prove the first estimate in \eqref{Hs-NL}.


Next, for the second estimate in \eqref{Hs-NL}, 
we recall from the proof of Lemma \ref{lem:Hs-resol-norm} that, for $z\in\C\setminus\R$,
\begin{align*}
    \widetilde{R}_{z}(\sigma)&=(z-\fh)^{-1}\sigma(z+\overline{\fh})^{-1},
\end{align*}
so that
\begin{align*}
	R_{z}(\sigma)=\widetilde{R}_{z}(\sigma)\overline{\sigma}(z-\fh)^{-1}\sigma(z+\overline{\fh})^{-1}.
\end{align*}
It the follows from the triangle and H\"older inequalities that
\begin{align}\label{NL-subleading}
	\big\|\widetilde{N}_{T}(\sigma)\big\|_{\cL_{h}^{2}}&\leq \frac{2}{\beta}\sum_{n\in\Z}\big\|R_{i\omega_{n}}(\sigma)\overline{\sigma}[(i\omega_n -H(\sigma))^{-1}]_{12}\big\|_{\cL_{h}^{2}}\nonumber\\
	&\lesssim\sum_{n\in\Z}\big\|\widetilde{R}_{i\omega_{n}}(\sigma)\overline{\sigma}(i\omega_{n}-\fh)^{-1}\sigma(i\omega_{n}+\overline{\fh})^{-1}\overline{\sigma}[(i\omega_{n}-H(\sigma))^{-1}]_{12}\big\|_{\cL_{h}^{2}}\nonumber\\
	&\lesssim\sum_{n\in\Z}\big\|\widetilde{R}_{i\omega_{n}}(\sigma)\big\|\big\|\overline{\sigma}(i\omega_{n}-\fh)^{-1}\sigma(i\omega_{n}+\overline{\fh})^{-1}\overline{\sigma}[(i\omega_{n}-H(\sigma))^{-1}]_{12}\big\|_{\cL_{h}^{2}}\nonumber\\
	&\lesssim \|\sigma\|_{\cL_{h}^{6}}^{3}\sum_{n\in\Z}\|(i\omega_{n}-\fh)^{-1}\|\|(i\omega_{n}+\overline{\fh})^{-1}\|\big\|[(i\omega_{n}-H(\sigma))^{-1}]_{12}\big\|\big\|\widetilde{R}_{i\omega_{n}}(\sigma)\big\|\nonumber\\
	&\lesssim\|\sigma\|_{\cL_{h}^{6}}^{3}\sum_{n\in\Z}|\omega_{n}|^{-3}\|\big\|\widetilde{R}_{i\omega_{n}}(\sigma)\big\|.
\end{align}
It remains to estimate the last factor in each summand.  Since $\fh$ commutes with $p_{j}=-i\partial_{j}$ for each $j=1,...,d$, we see that $-\Delta_{X}$ commutes with $\widetilde{R}_{i\omega_{n}}$.  By denoting $\widetilde{\sigma}=(\one-\Delta_{X})^{-1}\sigma$, we have
\begin{align*}
	\widetilde{R}_{i\omega_{n}}(\sigma)&=\widetilde{R}_{i\omega_{n}}\big((\one-\Delta_{X})\widetilde{\sigma}\big)=(\one-\Delta_{X})\big(\widetilde{R}_{i\omega_{n}}(\sigma)\big).
\end{align*}
Hence, by Lemma \ref{lem:Hs-resol-norm} with $\delta=1$, we obtain 
\begin{align*}
	\big\|\widetilde{R}_{i\omega_{n}}(\sigma)\big\|&=\big\|(\one -\Delta_{X})\widetilde{R}_{i\omega_{n}}(\widetilde{\sigma})\big\|\lesssim F_{2}(i\omega_{n})\|\widetilde{\sigma}\|=|\omega_{n}|^{-2}(1+|\omega_{n}|)\|\widetilde{\sigma}\|.
\end{align*}
Substituting this into \eqref{NL-subleading} yields
\begin{align}\label{NL-subleading-new}
	\big\|\widetilde{N}_{T}(\sigma)\big\|_{\cL_{h}^{2}}&\lesssim\|(\one-\Delta_{X})^{-1}\sigma\|\|\sigma\|_{\cL_{h}^{6}}^{3}\sum_{n\in\Z}|\omega_{n}|^{-4}(1+|\omega_{n}|)\nonumber\\
	&\lesssim\|(\one-\Delta_{X})^{-1}\sigma\|\|\sigma\|_{\cL_{h}^{6}}^{3}.
\end{align}
This implies the second estimate in \eqref{Hs-NL} and thus completes the proof.
\end{proof}

\smallskip

Now, we conclude this section by estimating the norm of $N_{T}'(\sigma_{1})-N_{T}'(\sigma_{2})$.

\begin{prop}\label{prop:NL-diff-norm}
Let $\sigma_{1},\sigma_{2}\in\cL_{h}^{2}\cap\cL_{h}^{6}$.  Then it holds that
\begin{align}\label{Hs-NL-diff-leading}
	\|N_{T}'(\sigma_{1})-N_{T}'(\sigma_{2})\|_{\cL_{h}^{2}}&\lesssim \Big(\max_{j=1,2}\|\sigma_{j}\|_{\cL_{h}^{6}}^{2}\Big)\|\sigma_{1}-\sigma_{2}\|_{\cL_{h}^{6}}.
\end{align}
\end{prop}

To prove this, we first prove the following lemma.

\begin{lemma}\label{lem:L2-resolv-Delta-diff}
Let $z\in \C\setminus\R$ and $\sigma_{1},\sigma_{2}\in\cL_{h}^{2}\cap\cL_{h}^{6}$.  Then it holds that 
\begin{align}
	\label{L2-resolv-Delta-diff1}\big\|R_{z}(\sigma_{1})-R_{z}(\sigma_{2})\big\|_{\cL_{h}^{2}}&\lesssim |\Im(z)|^{-4}\Big(\max_{j=1,2}\|\sigma_{j}\|_{\cL_{h}^{6}}^{2}\Big)\|\sigma_{1}-\sigma_{2}\|_{\cL_{h}^{6}}.
\end{align}
\end{lemma}
\begin{proof}
First, we compute
\begin{align*}
	R_{z}(\sigma_{1})-R_{z}(\sigma_{2})&=(z-\fh)^{-1}(\sigma_{1}-\sigma_{2})(z+\overline{\fh})^{-1}\overline{\sigma_{1}}(z-\fh)^{-1}\alpha_{1}(z+\overline{\fh})^{-1}\nonumber\\
	&\quad+(z-\fh)^{-1}\sigma_{2}(z+\overline{\fh})^{-1}(\overline{\sigma_{1}}-\overline{\sigma_{2}})(z-\fh)^{-1}\sigma_{1}(z+\overline{\fh})^{-1}\nonumber\\
	&\quad+(z-\fh)^{-1}\sigma_{2}(z+\overline{\fh})^{-1}\overline{\sigma_{2}}(z-\fh)^{-1}(\sigma_{1}-\sigma_{2})(z+\overline{\fh})^{-1}.
\end{align*}
Then, by \eqref{CS-triangle}, \eqref{Holder-periodic} and a similar argument in the proof of Lemma \ref{lem:resolvent-norms}, we obtain
\begin{align*}
	\big\|R_{z}(\sigma_{1})-R_{z}(\sigma_{2})\big\|_{\cL_{h}^{2}}\DETAILS{&\leq \|(z-\fh)^{-1}\|^{2}\|(z+\overline{\fh})^{-1}\|^{2}\|\sigma_{1}-\sigma_{2}\|_{\cL_{h}^{6}}\nonumber\\
		&\quad\quad\times\Big(\|\sigma_{1}\|_{\cL_{h}^{6}}^{2}+\|\sigma_{1}\|_{\cL_{h}^{6}}\|\sigma_{2}\|_{\cL_{h}^{6}}+\|\sigma_{2}\|_{\cL_{h}^{6}}^{2}\Big)\nonumber\\}
	&\lesssim |\Im(z)|^{-4}\|\sigma_{1}-\sigma_{2}\|_{\cL_{h}^{6}}\Big(\max_{j=1,2}\|\sigma_{j}\|_{\cL_{h}^{6}}^{2}\Big).
\end{align*}
This completes the proof.
%
\end{proof}

\begin{proof}[Proof of Proposition \ref{prop:NL-diff-norm}]
Everything follows from the same argument as in Proposition \ref{prop:NL-norm}, which we omit for simplicity.
\end{proof}
%


\begin{cor}\label{cor:NL-H1-bdd}
Under Assumptions (A1)--(A5), for any $\varphi,\varphi_{1},\varphi_{2}\in\cH_{h}^{1}$ and $s=0,1,2$, it holds that
\begin{align}\label{NL-H1-bdd}
	&\big\|N_{T}'(-2V^{1/2}\varphi)\big\|_{\cH_{h}^{s}}\lesssim h^{-s}\|\varphi\|_{\cH_{h}^{1}}^{3},\\
	\label{NL-H1-bdd3}&\big\|N_{T}'(-2V^{1/2}\varphi_{1})-N_{T}'(-2V^{1/2}\varphi_{2})\big\|_{\cL_{h}^{2}}\lesssim \Big(\max_{j=1,2} \|\varphi_{j}\|_{\cH_{h}^{1}}^{2}\Big) \big\|\varphi_{1}-\varphi_{2}\big\|_{\cH_{h}^{1}},
\end{align}
and, for each $\varepsilon>0$, there is some constant $C_{d,\varepsilon}>0$ such that
\begin{align}
	\label{NL-H1-bdd2}&\big\|\widetilde{N}_{T}(-2V^{1/2}\varphi)\big\|_{\cL_{h}^{2}}\leq C_{d,\varepsilon}h^{-(d/2-1+\varepsilon)}\|\varphi\|_{\cH_{h}^{1}}^{4}.
\end{align}
\end{cor}

\begin{proof}
The estimates \eqref{NL-H1-bdd} and \eqref{NL-H1-bdd3} are direct applications of Lemma \ref{lem:Lpnorm-Delta} and Proposition \ref{prop:NL-norm}, so we omit their proofs.

To prove \eqref{NL-H1-bdd2}, we apply Proposition \ref{prop:NL-norm} and Lemma \ref{lem:Lpnorm-Delta} to obtain
\begin{align*}
    \big\|\widetilde{N}_{T}(-2V^{1/2}\varphi)\big\|_{\cL_{h}^{2}}
    &\lesssim \| (\one-\Delta_X)^{-1}(V^{1/2}\varphi) \| \|V^{1/2}\varphi\|_{\mathcal L^6_h}^3 \\
    & \lesssim \| (\one-\Delta_X)^{-1}(V^{1/2}\varphi) \| \|\varphi\|_{\mathcal H^1_h}^3 \,.
\end{align*}
It remains to bound the first factor on the r.h.s. Using $(\one-\Delta_X)^{-1}(V^{1/2}\varphi)= V^{1/2} (\one-\Delta_X)^{-1}(\varphi)$ together with Lemma \ref{lem:op-norm-Delta}, we obtain for any $0<\varepsilon\leq 3-d/2$ a constant $C_{d,\varepsilon}>0$ such that
\begin{align*}
	\big\|(\one-\Delta_{X})^{-1}V^{1/2}\varphi\big\|&\leq C_{d,\varepsilon}\big\|(\one-h^{-2}\Delta_{X})^{(d/2+\varepsilon)/2}(\one-\Delta_{X})^{-1}\varphi\big\|_{\cL_{h}^{2}}\nonumber\\
	&=C_{d,\varepsilon}\big\|(\one-h^{-2}\Delta_{X})^{(d/2-1+\varepsilon)/2}(\one-\Delta_{X})^{-1}\varphi\big\|_{\cH_{h}^{1}}\nonumber\\
	&\leq C_{d,\varepsilon}h^{-(d/2-1+\varepsilon)}\|\varphi\|_{\cH_{h}^{1}}.
\end{align*}
The last inequality comes from
$$
\sup_{q\in 2\pi\Z^d} \frac{(1+|q|^2)^{\frac d2 - 1+ \epsilon}}{(1+h^2|q|^2)^2} \lesssim h^{-2(\frac d2 - 1+\epsilon)} \,.
$$
This proves \eqref{NL-H1-bdd2} for $\varepsilon\leq 3-d/2$, which trivially implies the inequality also for larger $\varepsilon$.
\end{proof}

\bigskip


\section{Proof of Theorem \ref{thm:deriv-GL}}\label{sec:pf-deriv-GL}

Our goal in this section is to prove our first main result, Theorem \ref{thm:deriv-GL}. In the first part of the proof we will work with the Birman--Schwinger reformulation of the BdG equation and prove bounds on the corresponding solution $\varphi=V^{1/2}\alpha$. Only at the end will we convert these bounds into bounds on $\alpha$.


\subsection{Projecting along \texorpdfstring{$P_\kappa$}{Pκ}}

Let $\varphi\in \cH_{h}^{1}$ be a solution of the BdG equation \eqref{BdG-BS-eqn} in Birman--Schwinger form and assume that $\|\varphi\|_{\cH_{h}^{1}}\lesssim h$.

Recalling the definition of the projection $P_{\kappa}$, $0<\kappa\leq 1$, in \eqref{P}, we project \eqref{BdG-BS-eqn} onto the ranges of $P_{\kappa}$ and $P_{\kappa}^{\perp}=1-P_{\kappa}$ to obtain
\begin{align}
	\label{BdG-BS-eqn-1}0&=P_{\kappa}L_{T}P_{\kappa}\varphi+P_{\kappa}L_{T}P_{\kappa}^{\perp}\varphi+\tfrac{1}{2}P_{\kappa}V^{1/2}N_{T}(-2V^{1/2}\varphi),\\
	\label{BdG-BS-eqn-2}0&=P_{\kappa}^{\perp}L_{T}P_{\kappa}\varphi+P_{\kappa}^{\perp}L_{T}P_{\kappa}^{\perp}\varphi+\tfrac{1}{2}P_{\kappa}^{\perp}V^{1/2}N_{T}(-2V^{1/2}\varphi).
\end{align}
We can rewrite \eqref{BdG-BS-eqn-1} as
\begin{align}\label{BdG-eqn8}
	& P_{\kappa}L_{T}P_{\kappa}\varphi+\tfrac{1}{2}PV^{1/2}N_{T}'(-2V^{1/2}P_{\kappa}\varphi)\nonumber\\
	&=\tfrac{1}{2}\lambda_{\kappa}^{\perp}PV^{1/2}N_{T}'(-2V^{1/2}P_{\kappa}\varphi)-P_{\kappa}L_{T}P_{\kappa}^{\perp}\varphi-\tfrac{1}{2}P_{\kappa}V^{1/2}\widetilde{N}_{T}(-2V^{1/2}\varphi)\nonumber\\
	&\quad\,-\tfrac{1}{2}(P_{\kappa}V^{1/2}N_{T}'(-2V^{1/2}\varphi)-P_{\kappa}V^{1/2}N_{T}'(-2V^{1/2}P_{\kappa}\varphi)).
\end{align}
A key step in the proof of Theorem \ref{thm:deriv-GL} will be to show that the r.h.s.~of \eqref{BdG-eqn8} is small. This is the assertion of the following proposition, which treats each one of the terms appearing on the r.h.s:

\begin{prop}\label{prop:small-NP-diff}
Let $\varphi$ be a solution to \eqref{BdG-BS-eqn}.  Under Assumptions (A1)--(A5) with sufficiently small $h>0$, if $\|\varphi\|_{\cH_{h}^{1}}\lesssim h$, then it holds that 
\begin{align}
	\label{small-offPal2}
	&\|P_{\kappa}L_{T}P_{\kappa}^{\perp}\varphi\|_{\cL_{h}^{2}}\lesssim h^{9/4}\|P_{\kappa}\varphi\|_{\cH_{h}^{1}},\\
	\label{small-NP-diff1}&\|P_{\kappa}V^{1/2}\widetilde{N}_{T}(-2V^{1/2}\varphi)\|_{\cL_{h}^{2}}\lesssim h^{9/4}\|P_{\kappa}\varphi\|_{\cH_{h}^{1}},\\
	\label{small-NP-diff2}&\|P_{\kappa}V^{1/2}N_{T}'(-2V^{1/2}\varphi)-P_{\kappa}V^{1/2}N_{T}'(-2V^{1/2}P_{\kappa}\varphi)\|_{\cL_{h}^{2}}\lesssim h^{13/6}\|P_{\kappa}\varphi\|_{\cH_{h}^{1}},\\
    \DETAILS{\nonumber\\
	&\quad\quad\quad\quad\quad\quad\quad\quad\quad\quad\quad\lesssim \big(h^{5/6}+h^{-11/6}\|P_{\kappa}\varphi\|_{\cH_{h}^{1}}^{2}\big)\|P_{\kappa}\varphi\|_{\cH_{h}^{1}}^{3}}
    \label{small-NP-diff3}&\|\lambda_{\kappa}^{\perp}PV^{1/2}N_{T}'(-2V^{1/2}P_{\kappa}\varphi)\|_{\cH_{h}^{-1}}\lesssim h^{31/12}\|P_{\kappa}\varphi\|_{\cH_{h}^{1}}.
\end{align}
\end{prop}

To prove this proposition, we make use of \eqref{BdG-BS-eqn-2}. By Proposition \ref{prop:inv-bdd}, for $\beta$ that is sufficiently close to $\beta_{c}$ (i.e., that satisfies Assumption (A5) with sufficiently small $h$), the operator $P_{\kappa}^{\perp}L_{T}P_{\kappa}^{\perp}$ is invertible on $\Ran(P_{\kappa}^{\perp})$ so that we can rewrite \eqref{BdG-BS-eqn-2} as
\begin{align}\label{BdG-eqn7}
	P_{\kappa}^{\perp}\varphi&=-(P_{\kappa}^{\perp}L_{T}P_{\kappa}^{\perp})^{-1}\Big(P_{\kappa}^{\perp}L_{T}P_{\kappa}\varphi+ \tfrac12 P_{\kappa}^{\perp}V^{1/2}N_{T}(-2V^{1/2}\varphi)\Big).
\end{align}
In what follows, we choose
\begin{align}\label{kappa-fix}
	\kappa=h^{5/6}
\end{align}
and assume $h\leq 1$, so that $\kappa\leq 1$. 

As a preparation for the proof of Proposition \ref{prop:small-NP-diff}, we first show the following lemmas:

\begin{lemma}\label{lem:small-offPal}
Let $\varphi\in\cH_{h}^{1}$ be a solution to \eqref{BdG-BS-eqn}.  Under Assumptions (A1)--(A5) with sufficiently small $h>0$, it holds for $s=0,1$ that 
\begin{align}
	\label{small-offPal}\|P_{\kappa}^{\perp}\varphi\|_{\cH_{h}^{s}}&\lesssim h^{5/12}\|P_{\kappa}\varphi\|_{\cH_{h}^{s}}+h^{-(5/6+s)}\|\varphi\|_{\cH_{h}^{1}}^{3}.
\end{align}
\end{lemma}
\begin{proof}
	First, 
    by Proposition \ref{prop:inv-bdd} and using the fact that $P_{\kappa}^{\perp} L_{T}P_{\kappa} = P^{\perp}\lambda_{\kappa}L_{T}P_{\kappa}$, we obtain
	\begin{align}\label{BdG-eqn7-1}
		P_{\kappa}^{\perp}\varphi=-(P_{\kappa}^{\perp}L_{T}P_{\kappa}^{\perp})^{-1}\big(P^{\perp}\lambda_{\kappa}L_{T}P_{\kappa}\varphi+ \tfrac12 P_{\kappa}^{\perp}V^{1/2}N_{T}(-2V^{1/2}\varphi)\big).
	\end{align}
	For the linear part, since $\|A\|_{\cH_{h}^{s}\rightarrow \cH_{h}^{s}}\leq \|A\|_{\cL_{h}^{2}\rightarrow\cL_{h}^{2}}$ for each $s\in\R$ whenever $A$ commutes with $-\Delta_{X}$, using the fact that $L_{T}$ and $P_{\kappa}$ commute with $-\Delta_{X}$ and Proposition \ref{prop:inv-bdd}, we obtain
	\begin{align*}
		\|(P_{\kappa}^{\perp}L_{T}P_{\kappa}^{\perp})^{-1}P^{\perp}\lambda_{\kappa}\|_{\cL_{h}^{2}\rightarrow\cL_{h}^{2}}\lesssim\kappa^{-1/2} 
        \quad\text{and}\quad \|P^{\perp}L_{T}P_{\kappa}\|_{\cL_{h}^{2}\rightarrow\cL_{h}^{2}}\lesssim\kappa. 
	\end{align*}
	This implies for $s=0,1$ that 
	\begin{align}\label{offP-al-L2}
		&\|(P_{\kappa}^{\perp}L_{T}P_{\kappa}^{\perp})^{-1}P^{\perp}\lambda_{\kappa}L_{T}P_{\kappa}\varphi\|_{\cH_{h}^{s}}\nonumber\\
		&\leq \|(P_{\kappa}^{\perp}L_{T}P_{\kappa}^{\perp})^{-1}P^{\perp}\lambda_{\kappa}\|_{\cL_{h}^{2}\rightarrow\cL_{h}^{2}}\|P^{\perp}L_{T}P_{\kappa}\|_{\cL_{h}^{2}\rightarrow\cL_{h}^{2}}\|P_{\kappa}\varphi\|_{\cH_{h}^{s}}\nonumber\\
		&\lesssim \kappa^{1/2}\|P_{\kappa}\varphi\|_{\cH_{h}^{s}}=h^{5/12}\|P_{\kappa}\varphi\|_{\cH_{h}^{s}}.
	\end{align}
	For the nonlinear component, by Lemma \ref{lem:Hsnorm-Delta}, Proposition \ref{prop:inv-bdd} and \ref{prop:BCS-NL-norm}, we obtain
	\begin{align*}
		&\big\|(P_{\kappa}^{\perp}L_{T}P_{\kappa}^{\perp})^{-1} P_{\kappa}^{\perp} V^{1/2}N_{T}(-2V^{1/2}\varphi)\big\|_{\cH_{h}^{s}}\nonumber\\
		&\lesssim \|(P_{\kappa}^{\perp}L_{T}P_{\kappa}^{\perp})^{-1}P_{\kappa}^{\perp}\|_{\cL_{h}^{2}\rightarrow\cL_{h}^{2}}\|V^{1/2}N_{T}(-2V^{1/2}\varphi)\|_{\cH_{h}^{s}}\nonumber\\
		&\lesssim\kappa^{-1}h^{-s}\|V^{1/2}\varphi\|_{\cL_{h}^{6}}^{3}\lesssim h^{-(5/6+s)}\|\varphi\|_{\cH_{h}^{1}}^{3}.
	\end{align*}
	Hence, we obtain \eqref{small-offPal}.
\end{proof}
\begin{lemma}\label{lem:Pper-Pbdd}
Let $\varphi$ be a solution to \eqref{BdG-BS-eqn}.  Under Assumptions (A1)--(A5) with sufficiently small $h>0$, if $\|\varphi\|_{\cH_{h}^{1}}\lesssim h$, then it holds that 
\begin{align}
	\label{Pper-H1norm}\|P_{\kappa}^{\perp}\varphi\|_{\cH_{h}^{1}}&\lesssim h^{1/6}\|P_{\kappa}\varphi\|_{\cH_{h}^{1}}.
\end{align}
\end{lemma}
\begin{proof}
Since, clearly,  $\|\varphi\|_{\cH_{h}^{1}}\leq\|P_{\kappa}\varphi\|_{\cH_{h}^{1}}+\|P_{\kappa}^{\perp}\varphi\|_{\cH_{h}^{1}}$, by Lemma \ref{lem:small-offPal} and our assumption $\|\varphi\|_{\cH_{h}^{1}}\lesssim h$, we obtain
\begin{align}
	\|P_{\kappa}^{\perp}\varphi\|_{\cH_{h}^{1}}&\lesssim h^{5/12}\|P_{\kappa}\varphi\|_{\cH_{h}^{1}}+h^{-11/6}\|\varphi\|_{\cH_{h}^{1}}^{2}\big(\|P_{\kappa}\varphi\|_{\cH_{h}}+\|P_{\kappa}^{\perp}\varphi\|_{\cH_{h}^{1}}\big)\nonumber\\
	&\lesssim (h^{5/12}+h^{1/6})\|P_{\kappa}\varphi\|_{\cH_{h}^{1}}+h^{1/6}\|P_{\kappa}^{\perp}\varphi\|_{\cH_{h}^{1}}.
\end{align}
If we assume that $h$ is sufficiently small, then the second term on the r.h.s.~can be absorbed in the l.h.s.~and we arrive at the claimed bound.
\end{proof}

\begin{proof}[Proof of Proposition \ref{prop:small-NP-diff}]
    Let us begin with the proof of \eqref{small-offPal2}. We bound
    \begin{align*}
	&\|P_{\kappa}L_{T}P_{\kappa}^{\perp}\varphi\|_{\cL_{h}^{2}}\leq\|P_{\kappa}L_{T}(P_{\kappa}^{\perp}L_{T}P_{\kappa}^{\perp})^{-1}L_{T}P_{\kappa}\varphi\|_{\cL_{h}^{2}}\nonumber\\
	&\quad\quad\quad\quad\quad\quad\quad\quad\quad+\|P_{\kappa}L_{T}P_{\kappa}^{\perp}(P_{\kappa}^{\perp}L_{T}P_{\kappa}^{\perp})^{-1}P_{\kappa}^{\perp} \tfrac12 V^{1/2}N_{T}(-2V^{1/2}\varphi)\|_{\cL_{h}^{2}}\nonumber\\
	&\leq \|P_{\kappa}L_{T}(P_{\kappa}^{\perp}L_{T}P_{\kappa}^{\perp})^{-1}L_{T}P_{\kappa}\|_{\cH_{h}^{1}\rightarrow\cL_{h}^{2}} \|P_{\kappa}\varphi\|_{\cH_{h}^{1}}\nonumber\\
	&\quad+ \|P_{\kappa}L_{T}P^{\perp}\|_{\cL_{h}^{2}\rightarrow\cL_{h}^{2}} \| \lambda_{\kappa} P^{\perp}(P_{\kappa}^{\perp}L_{T}P_{\kappa}^{\perp})^{-1} P^{\perp}\lambda_{\kappa}\|_{\cL_{h}^{2}\rightarrow\cL_{h}^{2}} \| \tfrac12 V^{1/2}N_{T}(-2V^{1/2}\varphi)\|_{\cL_{h}^{2}}.
    \end{align*}
    Here we used the fact that $\lambda_\kappa$ commutes with $L_T$ and $P$ and therefore
    $$
    P_{\kappa}L_{T}P_{\kappa}^{\perp}(P_{\kappa}^{\perp}L_{T}P_{\kappa}^{\perp})^{-1}P_{\kappa}^{\perp}
    = \big( P_\kappa L_{T} P^\bot \big) \big( \lambda_\kappa P^{\perp}(P_{\kappa}^{\perp}L_{T}P_{\kappa}^{\perp})^{-1} P^{\perp}\lambda_{\kappa} \big).
    $$
    
    According to Proposition \ref{prop:inv-bdd-LP}, we have
    $$
    \|P_{\kappa}L_{T}(P_{\kappa}^{\perp}L_{T}P_{\kappa}^{\perp})^{-1}L_{T}P_{\kappa}\|_{\cH_{h}^{1}\rightarrow\cL_{h}^{2}} \lesssim \kappa^{3/2} h = h^{9/4} \,.
    $$
    Moreover, according to Proposition \ref{prop:PLP-inv-bdd}, we have
    $$
    \| \lambda_{\kappa} P^{\perp}(P_{\kappa}^{\perp}L_{T}P_{\kappa}^{\perp})^{-1} P^{\perp}\lambda_{\kappa}\|_{\cL_{h}^{2}\rightarrow\cL_{h}^{2}} \lesssim 1
    $$
    and according to Lemma \ref{lem:Q0LP-bdd}, we have
    $$ \|P_{\kappa}L_{T}P^{\perp}\|_{\cL_{h}^{2}\rightarrow\cL_{h}^{2}} \lesssim \kappa = h^{5/6}.
    $$
    Finally, we bound, using Lemma \ref{lem:Hsnorm-Delta}, Proposition \ref{prop:PLP-inv-bdd} and Lemmas \ref{lem:Lpnorm-Delta} and \ref{lem:Pper-Pbdd},
    \begin{align*}
        \| \tfrac12 V^{1/2}N_{T}(-2V^{1/2}\varphi)\|_{\cL_{h}^{2}}
        & \lesssim \| N_{T}(-2V^{1/2}\varphi)\|_{\cL_{h}^{2}} \\
        & \lesssim \| V^{1/2}\varphi\|_{\cL_{h}^{6}}^3 \lesssim \|\varphi\|_{\cH_{h}^{1}}^3 \lesssim h^2 \|P_\kappa\varphi\|_{\cH_{h}^{1}} \,.    
    \end{align*}
    Collecting all these bounds, we arrive at \eqref{small-offPal2}.

Next, for \eqref{small-NP-diff1}, we\DETAILS{recall the definition of $\widetilde{N}_{T}$ from \eqref{NL-higher} and then} apply Corollary \ref{cor:NL-H1-bdd} with $\varepsilon=1/4$ to obtain
\begin{align*}
    \big\|P_{\kappa}V^{1/2}\widetilde{N}_{T}(-2V^{1/2}\varphi)\big\|_{\cL_{h}^{2}}&\lesssim\big\|\widetilde{N}_{T}(-2V^{1/2}\varphi)\big\|_{\cL_{h}^{2}}\lesssim h^{-3/4}\|\varphi\|_{\cH_{h}^{1}}^{4}.
\end{align*}
Since $\|\varphi\|_{\cH_{h}^{1}}\lesssim h$, we obtain \eqref{small-NP-diff1}.

Now, by Proposition \ref{prop:NL-diff-norm} and Lemmas \ref{lem:Lpnorm-Delta}, \ref{lem:Hsnorm-Delta} and \ref{lem:Pper-Pbdd}, we obtain \eqref{small-NP-diff2} as follows:
\begin{align*}
	&\|P_{\kappa}V^{1/2}N_{T}'(-2V^{1/2}\varphi)-P_{\kappa}V^{1/2}N_{T}'(-2V^{1/2}P_{\kappa}\varphi)\|_{\cL_{h}^{2}}\nonumber\\
	&\lesssim\|N_{T}'(-2V^{1/2}\varphi)-N_{T}'(-2V^{1/2}P_{\kappa}\varphi)\|_{\cL_{h}^{2}} \nonumber\\
    & \lesssim \big( \|V^{1/2}\varphi\|_{\cL_{h}^{6}}^{2} + \|V^{1/2}P_\kappa \varphi\|_{\cL_{h}^{6}}^{2} \big) \|V^{1/2}P_{\kappa}^{\perp}\varphi\|_{\cL_{h}^{6}}\nonumber\\
	&\lesssim\|\varphi\|_{\cH_{h}^{1}}^{2}\|P_{\kappa}^{\perp}\varphi\|_{\cH_{h}^{1}}\lesssim h^{1/6}\|P_{\kappa}\varphi\|_{\cH_{h}^{1}}^{3}\lesssim h^{13/6}\|P_{\kappa}\varphi\|_{\cH_{h}^{1}}.
\end{align*}

Finally, to prove \eqref{small-NP-diff3}, we first note that, by Fourier transform, for each $\sigma\in\cH_{h}^{1}$,
\begin{align*}
	\big\|\lambda_{\kappa}^{\perp}\sigma\big\|_{\cH_{h}^{-1}}^{2}&=\sum_{q\in 2\pi\Z^{d}}\int_{\R^{d}}\frac{\one(h^{2}|q|^{2}>\kappa)}{1+|q|^{2}}|\hat{\sigma}(q,p)|^{2}dp\nonumber\\
	&\leq h^{2}\kappa^{-1}\sum_{q\in 2\pi\Z^{d}}\int_{\R^{d}}\one(h^{2}|q|^{2}>\kappa)\frac{|q|^{2}}{1+|q|^{2}}|\hat{\sigma}(q,p)|^{2}dp\nonumber\\
	&\leq h^{2}\kappa^{-1}\|\sigma\|_{\cL_{h}^{2}}^{2}=h^{7/6}\|\sigma\|_{\cL_{h}^{2}}^{2}.
\end{align*}
It then follows from Lemma \ref{lem:Hsnorm-Delta}, Proposition \ref{prop:NL-norm} and Lemma \ref{lem:Lpnorm-Delta} that
\begin{align}
    \big\|\lambda_{\kappa}^{\perp}PV^{1/2}&N_{T}'(-2V^{1/2}P_{\kappa}\varphi)\big\|_{\cH_{h}^{-1}}\lesssim h^{7/12}\big\|N_{T}'(-2V^{1/2}P_{\kappa}\varphi)\big\|_{\cL_{h}^{2}}\nonumber\\
    &\lesssim h^{7/12}\|V^{1/2}P_{\kappa}\varphi\|_{\cL_{h}^{6}}^{3}\lesssim h^{7/12}\|P_{\kappa}\varphi\|_{\cH_{h}^{1}}^{3}\lesssim h^{31/12}\|P_{\kappa}\varphi\|_{\cH_{h}^{1}}.
\end{align}
This proves \eqref{small-NP-diff3} and thus completes the proof of the proposition.
\end{proof}


\subsection{Extracting the leading term}

In the previous subsection, we have shown that the r.h.s.~of \eqref{BdG-eqn8} is small. Now we will extract the leading order terms from the l.h.s. We define a function $\Psi$ on $\T^d_{h}$ by
\begin{align}\label{GL-fun1}
	\Psi&=\lambda_{\kappa}\int_{\R}\varphi_{*}(s)\varphi(\zeta_{\ \cdot \,}^{s},\zeta_{\ \cdot \,}^{-s}) \,ds \,.
\end{align}
(Here the projection $\lambda_\kappa$ acts with respect to the variable $X$ of $\varphi(\zeta_{X}^{s},\zeta_{X}^{-s})$. This definition implies that
\begin{align}\label{GL-fun}
	\Psi(X)\varphi_{*}(r)&=(P_{\kappa}\varphi)(\zeta_{X}^{r},\zeta_{X}^{-r}),
\end{align}
Since $\varphi_*$ is normalized, we see that 
\begin{equation}
    \label{eq:normpsiphi}
    \|\Psi\|_{H_{h}^{s}(\T^d_{h})}=\|P_{\kappa}\varphi\|_{\cH_{h}^{s}}
    \quad \text{for any}\ s\in\R \,.
\end{equation}
We recall that the norm in the Sobolev space $H_{h}^{s}(\T^d_{h})$ on the l.h.s.~was defined in Section \ref{sec:setup}.

The following proposition extracts the leading term from the first term on the l.h.s.~of \eqref{BdG-eqn8}.

\begin{prop}\label{prop:GL-matrix}
Let $\varphi$ be a solution to \eqref{BdG-BS-eqn} and define $\Psi$ by \eqref{GL-fun1}.  Under the same assumptions as in Lemma \ref{lem:small-offPal}, it holds that 
\begin{align}\label{GL-matrix}
	\big\|P_{\kappa}L_{T}P_{\kappa}\varphi-\big((-\nabla_{X}\cdot \Lambda_{0}\nabla_{X}-\Lambda_{2} D h^{2})\Psi\big)\varphi_{*}\|_{\cL_{h}^{2}}\lesssim h^{9/4}\|\Psi\|_{H^{1}_h(\T^d_{h})} \,,
\end{align}
where $\Lambda_{0}$ and $\Lambda_{2}$ are given in \eqref{eq:lambda0} and \eqref{eq:lambda2}, respectively.
\end{prop}

\begin{proof}
First, we note that, since $\varphi_{*}$ is the eigenvector of $\ell_{T_{c}}$ with eigenvalue zero, we obtain
\begin{align}\label{PLTP-decomp}
	P_{\kappa}L_{T}P_{\kappa}&=P_{\kappa}(L_{T}-\ell_{T})P_{\kappa}+P_{\kappa}(\ell_{T}-\ell_{T_{c}})P_{\kappa}\nonumber\\
	&=P_{\kappa}V^{1/2}(k_{T}-K_{T})V^{1/2}P_{\kappa}+P_{\kappa}V^{1/2}(k_{T_{c}}-k_{T})V^{1/2}P_{\kappa}.
\end{align}
We begin by studying $P_{\kappa}V^{1/2}(k_{T}-K_{T})V^{1/2}P_{\kappa}\varphi$. In view of \eqref{GL-fun} we have 
$$
(V^{1/2}P_{\kappa}\varphi)(\zeta_{X}^{r},\zeta_{X}^{-r})=V^{1/2}(r)\varphi_{*}(r)\Psi(X) \,,
$$
whose Fourier transform takes the form
\begin{align*}
	\widehat{(V^{1/2}P_{\kappa}\varphi)}(q,p)&=\frac{1}{2}t_{*}(p)\hat{\Psi}(q),
\end{align*}
where, recall, $t_{*}(p)$ is defined in \eqref{eq:deft} (the Fourier transform of $2V\alpha_{*}$), and $\hat{\Psi}$ denotes the Fourier transform of $\Psi$, given by \eqref{hZ-FT}. It then follows from \eqref{KT-kernel-rep} that
\begin{align}\label{GL-kinetic}
	&[P_{\kappa}V^{1/2}(k_{T}-K_{T})V^{1/2}P_{\kappa}\varphi](\zeta_{X}^{r},\zeta_{X}^{-r})\nonumber\\
	&=\varphi_{*}(r)\int_{\R^{d}}V^{1/2}(s)\varphi_{*}(s)\Big[\frac{1}{(2\pi)^{d/2}}\sum_{q\in2\pi\Z^{d}}\int_{\R^{d}}\big(f_{\beta}(p+hq/2,p-hq/2)-f_{\beta}(p,p)\big)\nonumber\\
	&\quad\quad\quad\quad\quad\quad\quad\quad\quad\quad\quad\quad\quad\quad\quad\quad\times e^{ihq\cdot X}e^{ip\cdot s}\eta(p)\hat{\Psi}(q)dp\Big]ds\nonumber\\
	&=\varphi_{*}(r)\sum_{q\in 2\pi\Z^{d}}\frac{1}{4(2\pi)^{d}}\int_{\R^{d}}\big(f_{\beta}(p+hq/2,p-hq/2)-f_{\beta}(p,p)\big)t_{*}(p)^{2}e^{ihq\cdot X}\hat{\Psi}(q)dp\nonumber\\
	&=\varphi_{*}(r)\sum_{q\in 2\pi\Z^{d}}e^{ihq\cdot X}G(hq)\hat{\Psi}(q),
\end{align}
where
\begin{align}\label{Gq}
	G(hq)&:=\frac{1}{4(2\pi)^{d}}\int_{\R^{d}}\big(f_{\beta}(p,p)-f_{\beta}(p+hq/2,p-hq/2)\big)t_{*}(p)^{2}dp \,.
\end{align}

To evaluate $G(hq)$, we observe that it suffices to consider $h|q|\leq \kappa^{1/2}\leq 1$ and $G(q)$ is even and twice-differentiable at $q=0$ (using the property of $t_{*}(p)$, thanks to \cite[Appendix A]{FHSS-12}).  Using Taylor's theorem, the momentum cut-off with sufficiently small $\kappa$ and the fact that $G(0)=0$, we obtain
\begin{align}\label{GL-kinetic-esti}
	\big|G(hq)-\frac{1}{2}h^{2}(q\cdot\nabla)^{2}G(0)\big|\lesssim h^{4}|q|^{4},
\end{align}
with
\begin{align}
	&\frac{h^{2}}{2}(q\cdot \nabla)^{2}G(0)=\frac{h^{2}}{2}\sum_{i,j=1}^{3}q_{i}q_{j}\partial_{i}\partial_{j}G(0)\nonumber\\
	&=h^{2}\sum_{i,j=1}^{3}q_{i}q_{j}\int_{\R^{d}}\frac{t_{*}(p)^{2}}{4(2\pi)^{d}}\Big[\frac{\beta^{2}\tanh\big(\beta(|p|^{2}-\mu)/2\big)}{4(|p|^{2}-\mu)\cosh^{2}\big(\beta(|p|^{2}-\mu)/2\big)}p_{i}p_{j}\nonumber\\
	&\quad\quad\quad\quad\quad\quad+\Big(\frac{\tanh\big(\beta(|p|^{2}-\mu)/2\big)}{4(|p|^{2}-\mu)}-\frac{\beta/8}{(|p|^{2}-\mu)\cosh^{2}\big(\beta(|p|^{2}-\mu)/2\big)}\Big)\delta_{ij}\Big]dp\nonumber\\
	&=:h^{2}\sum_{i,j=1}^{3}(\Lambda_{0})_{ij}(\beta)q_{i}q_{j}.
\end{align}
Note that, by Assumption (A5), $(\Lambda_{0})_{ij}(\beta)=(\Lambda_{0})_{ij}+O(h^{2})$ with $(\Lambda_{0})_{ij}\equiv(\Lambda_{0})_{ij}(\beta_{c})$ defined in \eqref{eq:lambda0}.  We thus obtain
\begin{align*}
	\frac{h^{2}}{2}(q\cdot\nabla)^{2}G(0)&=h^{2}\sum_{i,j=1}^{3}(\Lambda_{0})_{ij}q_{i}q_{j}+O(h^{4})|q|^{2}.
\end{align*}
Consequently, by \eqref{GL-kinetic-esti} and the momentum cutoff in center of mass coordinate, we obtain
\begin{align*}
	&\big\|P_{\kappa}V^{1/2}(k_{T_{c}}-K_{T_{c}})V^{1/2}P_{\kappa}\varphi-(-\nabla_{X}\cdot\Lambda_{0}\nabla_{X})\Psi\varphi_{*}\big\|_{\cL_{h}^{2}}^{2}\nonumber\\
	&\lesssim\sum_{q\in2\pi\Z^{d}}\big|G(hq)-\frac{h^{2}}{2}(q\cdot\nabla)^{2}G(0)\big|^{2}|\hat{\Psi}(q)|^{2}+h^{4}\sum_{q\in2\pi\Z^{d}}\one(h^{2}|q|^{2}\leq\kappa)h^{4}|q|^{4}|\hat{\Psi}(q)|^{2}\nonumber\\
	&\lesssim\sum_{q\in 2\pi\Z^{d}}\one(h^{2}|q|^{2}\leq\kappa)\big(h^{8}|q|^{8}+h^{8}|q|^{4}\big)|\hat{\Psi}(q)|^{2}\lesssim\kappa(\kappa^{2}+h^{4})\sum_{q\in2\pi\Z^{d}}h^{2}|q|^{2}|\hat{\Psi}(q)|^{2}\nonumber\\
	&\lesssim\kappa(\kappa^{2}+h^{4}) \|\nabla_{X}\Psi\|_{L^{2}_h(\T^d_{h})}^{2} \lesssim\kappa h^{2}(\kappa^{2}+h^{4})\|\Psi\|_{H_h^{1}(\T^d_{h})}^{2}\lesssim h^{9/2}\|\Psi\|_{H_h^{1}(\T^d_{h})}^{2} \,.
\end{align*}

Next, for the term $P_{\kappa}V^{1/2}(k_{T_{c}}-k_{T})V^{1/2}P_{\kappa}\varphi$ from \eqref{PLTP-decomp}, we recall $\beta-\beta_{c}=Dh^{2}\beta_{c}$ from Assumption (A5) and compute as in \eqref{GL-kinetic} to obtain
\begin{align*}
	[P_{\kappa}V^{1/2}(k_{T_{c}}-k_{T})V^{1/2}P_{\kappa}\varphi](\zeta_{X}^{r},\zeta_{X}^{-r})&=\widetilde{\Lambda}_{2}\varphi_{*}(r)\sum_{q\in 2\pi\Z^{d}}\hat{\Psi}(q),
\end{align*}
where
\begin{align}\label{tildeC1}
	\widetilde{\Lambda}_{2}&=\frac{1}{4(2\pi)^{d}}\int_{\R^{d}}t_{*}(p)^{2}\big(f_{\beta_{c}}(p,p)-f_{\beta}(p,p)\big)dp.
\end{align}
Let $\widetilde{f}_{p}(\beta):=f_{\beta}(p,p)$.  By the fundamental theorem of calculus and the fact that
\begin{align*}
	\beta\mapsto \widetilde{f}_{p}'(\beta)=\frac{1}{2\cosh^{2}\big(\beta(|p|^{2}-\mu)/2\big)}
\end{align*}
is Lipschitz for each fixed $p\in\R^{d}$ and satisfies
\begin{align}
	\big|\widetilde{f}_{p}'(\beta_{1})-\widetilde{f}_{p}'(\beta_{2})\big|&\leq C\big||p|^{2}-\mu\big||\beta_{1}-\beta_{2}|,
\end{align}
for some constant $C>0$, independent of $p$, we obtain for each $p\in\R^{d}$ that
\begin{align}\label{PKTKTcP-diff}
	f_{\beta_{c}}(p,p)-f_{\beta}(p,p)&=\widetilde{f}_{p}(\beta_{c})-\widetilde{f}_{p}(\beta)=-Dh^{2}\beta_{c}\int_{0}^{1}\widetilde{f}_{p}'(\beta_{c}+t\delta\beta)dt\nonumber\\
	&=-Dh^{2}\beta_{c}\widetilde{f}_{p}'(\beta_{c})-\delta\beta\int_{0}^{1}\big(\widetilde{f}_{p}'(\beta_{c}+t\delta\beta)-\widetilde{f}_{p}'(\beta_{c})\big)dt,
\end{align}
which gives
\begin{align*}
	\big|f_{\beta_{c}}(p,p)-f_{\beta}(p,p)+Dh^{2}\beta_{c}\widetilde{f}_{p}'(\beta_{c})\big|&\leq C\big||p|^{2}-\mu\big|(Dh^{2}\beta_{c})^{2}\lesssim h^{4}(1+|p|^{2}).
\end{align*}
Since $\int_{\R^{d}}t_{*}(p)^{2}(1+|p|^{2})dp$ is bounded (see \cite{FHSS-12} for details), this implies that 
\begin{align*}
	\widetilde{\Lambda}_{2}&=\frac{Dh^{2}}{8(2\pi)^{d}}\int_{\R^{d}}\frac{\beta_{c}}{\cosh^{2}\big(\beta_{c}(|p|^{2}-\mu)/2\big)}t_{*}(p)^{2}dp+O(h^{4})=:\Lambda_{2}Dh^{2}+O(h^{4}).
\end{align*}
This implies that 
\begin{align*}
	\big\|P_{\kappa}V^{1/2}(k_{T_{c}}-k_{T})V^{1/2}P_{\kappa}\varphi-\Lambda_{2}Dh^{2}\Psi\varphi_{*}\big\|_{\cL_{h}^{2}}&\lesssim h^{4}\|\Psi\|_{H^{1}_h(\T^d_{h})}
	,
\end{align*}
which completes the proof.
\end{proof}

The following proposition extracts the leading term from the second term on the l.h.s.~of \eqref{BdG-eqn8}.

\begin{prop}\label{prop:GL-nonlinear}
Let $\varphi$ be a solution to \eqref{BdG-BS-eqn} and define $\Psi$ by \eqref{GL-fun1}.  Under the same assumptions as in Lemma \ref{lem:small-offPal}, it holds that 
\begin{align}\label{GL-nonlinear}
	\big\|\tfrac{1}{2}PV^{1/2}N_{T}'(-2V^{1/2}P_{\kappa}\varphi)-(\Lambda_{3}|\Psi|^{2}\Psi)\varphi_{*}\big\|_{\cL_{h}^{2}}\lesssim h^{29/12}\|\Psi\|_{H^{1}_h(\T^d_{h})},
\end{align}
where $\Lambda_{3}$ is the constant given in \eqref{eq:lambda3}.
\end{prop}
\begin{proof}
Recall the integral kernel $G_{z}(x,y)=G_{z}(x-y)$ of $(z-\fh)^{-1}$ and the definition of $N_{T}'$ from \eqref{NL-leading}.  In terms of integral kernels and coordinate $\zeta_{X}^{r}$, we can write for each $\sigma\in\cL_{h}^{2}\cap\cL_{h}^{6}$ that
\begin{align*}
	[N_{T}'(\sigma)](\zeta_{X}^{r},\zeta_{X}^{-r})\DETAILS{&=\frac{1}{\beta}\sum_{n\in\Z}\int_{\R^{6}} G_{i\omega_{n}}(x-w_{1})\Delta_{\alpha}(w_{1},w_{2})G_{-i\omega_{n}}(w_{2}-w_{3})\overline{\Delta_{\alpha}(w_{3},w_{4})}\\
	&\quad\quad\quad\quad\quad\quad\times G_{i\omega_{n}}(w_{4}-w_{5})\Delta_{\alpha}(w_{5}-w_{6})G_{-i\omega_{n}}(w_{6}-y)d\textbf{w}_{1...6}\\}
	&=-\frac{2}{\beta}\sum_{n\in\Z}\int_{\R^{3d}\times\R^{3d}} d\textbf{Y}_{123}d\textbf{s}_{123}\sigma(\zeta_{Y_{1}}^{s_{1}},\zeta_{Y_{1}}^{-s_{1}})\overline{\sigma(\zeta_{Y_{2}}^{s_{2}},\zeta_{Y_{2}}^{-s_{2}})}\sigma(\zeta_{Y_{3}}^{s_{3}},\zeta_{Y_{3}}^{-s_{3}})\\
	&\quad\times G_{i\omega_{n}}(\zeta_{X-Y_{1}}^{r-s_{1}})G_{-i\omega_{n}}(\zeta_{Y_{1}-Y_{2}}^{-(s_{1}+s_{2})})G_{i\omega_{n}}(\zeta_{Y_{2}-Y_{3}}^{-(s_{2}+s_{3})})G_{-i\omega_{n}}(\zeta_{Y_{3}-X}^{r-s_{3}}),
\end{align*}
where $d\textbf{Y}_{123}=dY_{1}dY_{2}dY_{3}$ and $d\textbf{s}_{123}=ds_{1}ds_{2}ds_{3}$ are product measures on $\R^{3d}$.  By substituting $\sigma=P_{\kappa}\varphi$, shifting $Y_{i}\mapsto Y_{i}+X$ for each $i=1,2,3$ and using the fact that $\Psi(X+Y)=(e^{iY\cdot(-i\nabla_{X})}\Psi)(X)$, we obtain yield
\begin{align*}
	&[PV^{1/2}N_{T}'(-2V^{1/2}P_{\kappa}\varphi)](\zeta_{X}^{r},\zeta_{X}^{-r})\nonumber\\
    &=-\varphi_{*}(r)\int_{\R^{3d}} d\textbf{Y}_{123}\widetilde{\cG}_{V}(Y_{1},Y_{2},Y_{3}) \\
    & \qquad\qquad\qquad \times (e^{iY_{1}\cdot (-i\nabla_{X})}\Psi)(X)\overline{(e^{iY_{2}\cdot (-i\nabla_{X})}\Psi)(X)}(e^{iY_{3}\cdot (-i\nabla_{X})}\Psi)(X).
\end{align*}
where we denote the kernel
\begin{align}\label{G-kernel}
	\widetilde{\cG}_{V}(Y_{1},Y_{2},Y_{3})&=\frac{8}{\beta}\sum_{n\in\Z}\int_{\R^{4d}}d\textbf{s}_{0123}(V^{1/2}\varphi_{*})(s_{0})(V^{1/2}\varphi_{*})(s_{1})(V^{1/2}\varphi_{*})(s_{2})(V^{1/2}\varphi_{*})(s_{3})\nonumber\\
	&\quad\quad\times G_{i\omega_{n}}(\zeta_{Y_{1}}^{s_{1}-s_{0}})G_{-i\omega_{n}}(\zeta_{Y_{1}-Y_{2}}^{-(s_{1}+s_{2})})G_{i\omega_{n}}(\zeta_{Y_{2}-Y_{3}}^{-(s_{2}+s_{3})})G_{-i\omega_{n}}(\zeta_{Y_{3}}^{s_{0}-s_{3}}).
\end{align}
By denoting $\Phi_{Y}(X)=(1-e^{iY\cdot(-i\nabla_{X})})\Psi$, it follows that 
\begin{align}
	&[PV^{1/2}N_{T}'(-2V^{1/2}P_{\kappa}\varphi)](\zeta_{X}^{r},\zeta_{X}^{-r})=\widetilde{C}_{2}|\Psi(X)|^{2}\Psi(X)\varphi_{*}(r)\nonumber\\
	&\quad+\varphi_{*}(r)\int_{\R^{3d}} d\textbf{Y}_{123} \, \widetilde{\cG}_{V}(Y_{1},Y_{2},Y_{3})\Big(\Phi_{Y_{1}}(X)\overline{(e^{iY_{2}\cdot(-i\nabla_{X})}\Psi)(X)}(e^{iY_{3}\cdot(-i\nabla_{X})}\Psi)(X)\nonumber\\
	&\quad\quad\quad\quad\quad+\Psi(X)\overline{\Phi_{Y_{2}}(X)}(e^{iY_{3}\cdot(-i\nabla_{X})}\Psi)(X)+|\Psi(X)|^{2}\Phi_{Y_{3}}(X)\Big)\nonumber\\
	&=:\widetilde{\Lambda}_{3}|\Psi(X)|^{2}\Psi(X)\varphi_{*}(r)+[\cN_{T}(\varphi)](\zeta_{X}^{r},\zeta_{X}^{-r}),
\end{align}
where
\begin{align}
	\widetilde{\Lambda}_{3}&:=\int_{\R^{3}} d\textbf{Y}_{123} \, \widetilde{\cG}_{V}(Y_{1},Y_{2},Y_{3}).
\end{align}
It remains to compute the constant $\widetilde{\Lambda}_{3}$ explicitly and estimate $\|\cN_{T}(\varphi)\|_{\cL_{h}^{2}}$.  For the latter purpose, we note that, on one hand, one can show for each $Y\in\R^{d}$ that, using Sobolev inequality,
\begin{align}
	\|e^{iY\cdot(-i\nabla_{X})}\Psi\|_{L^{6}(\T^d_{h})}&\lesssim \|e^{iY\cdot(-i\nabla_{X})}\Psi\|_{H^{1}_h(\T^d_{h})}=\|\Psi\|_{H^{1}_h(\T^d_{h})}.
\end{align} 
On the other hand, using the fundamental theorem of calculus, we obtain
\begin{align*}
	\Phi_{Y}(X)&=(1-e^{iY\cdot(-i\nabla_{X})})\Psi(X)=-\int_{0}^{1}(iY\cdot(-i\nabla_{X}))e^{itY\cdot(-i\nabla_{X})}\Psi(X)dt\nonumber\\
	&=-Y\cdot\int_{0}^{1}\nabla_{X}\Psi(X+tY)dt.
\end{align*}
Hence, we obtain the pointwise estimate
\begin{align*}
	\big|\Phi_{Y}(X)\big|&\leq |Y|\int_{0}^{1}|\nabla_{X}\Psi(X+tY)|dt.
\end{align*}
By Minkowski and Sobolev inequalities, we obtain for each $Y\in\R^{d}$ that
\begin{align*}
	\|\Phi_{Y}\|_{L^{6}(\T^d_{h})}&\leq |Y|\Big(h^{d}\int_{\T^d_{h}}\Big|\int_{0}^{1}|\nabla_{X}\Psi(X+tY)|dt\Big|^{6}\,dX\Big)^{1/6}\nonumber\\
	&\leq |Y|\int_{0}^{1}\Big(h^{d}\int_{\T^d_{h}}|\nabla_{X}\Psi(X+tY)|^{6}\,dX\Big)^{1/6}dt\nonumber\\
	&\leq |Y|\int_{0}^{1}\big\|\nabla_{X}\Psi(\cdot+tY)\big\|_{L^{6}(\T^d_{h})}dt\lesssim |Y|\int_{0}^{1}\|\nabla_{X}\Psi(\cdot+tY)\|_{H^{1}_h(\T^d_{h})}dt\nonumber\\
	&\lesssim|Y|\|\nabla_{X}\Psi\|_{H^{1}_h(\T^d_{h})}\lesssim\kappa^{1/2}|Y|\|\Psi\|_{H^{1}_h(\T^d_{h})}\lesssim h^{5/12}|Y|\|\Psi\|_{H^{1}_h(\T^d_{h})}.
\end{align*}
Together with triangle, Minkowski and Sobolev inequalities, it follows from these estimates that 
\begin{align}\label{cNT-esti}
	&\|\cN_{T}(\varphi)\|_{\cL_{h}^{2}} \nonumber\\
    & =\Big[h^{d}\int_{\T^d_{h}}\Big|\int_{\R^{3d}} \widetilde{\cG}_{V}(Y_{1},Y_{2},Y_{3})\Big(\Phi_{Y_{1}}(X)\overline{(e^{iY_{2}\cdot(-i\nabla_{X})}\Psi)(X)}(e^{iY_{3}\cdot(-i\nabla_{X})}\Psi)(X)\nonumber\\
	&\quad\quad\quad+\Psi(X)\overline{\Phi_{Y_{2}}(X)}(e^{iY_{3}\cdot(-i\nabla_{X})}\Psi)(X)+|\Psi(X)|^{2}\Phi_{Y_{3}}(X))d\textbf{Y}_{123}\Big|^{2}\,dX\Big]^{1/2}\nonumber\\
	&\leq \int_{\R^{3d}}\Big[h^{d}\int_{\T^d_{h}}|\widetilde{\cG}_{V}(Y_{1},Y_{2},Y_{3})|^{2}\Big(|\Phi_{Y_{1}}(X)|^{2}|(e^{iY_{2}\cdot(-i\nabla_{X})}\Psi)(X)|^{2}|(e^{iY_{3}\cdot(-i\nabla_{X})}\Psi)(X)|^{2}\nonumber\\
	&\quad\quad+|\Psi(X)|^{2}|\Phi_{Y_{2}}(X)|^{2}|(e^{iY_{3}\cdot(-i\nabla_{X})}\Psi)(X)|^{2}+|\Psi(X)|^{4}|\Phi_{Y_{3}}(X)|^{2}\Big)\,dX\Big]^{1/2}d\textbf{Y}_{123}\nonumber\\
	&\leq \int_{\R^{3d}}|\widetilde{\cG}_{V}(Y_{1},Y_{2},Y_{3})|\Big(\|\Phi_{Y_{1}}\|_{L^{6}(\T^d_{h})}\|e^{iY_{2}\cdot(-i\nabla_{X})}\Psi\|_{L^{6}(\T^d_{h})}\|e^{iY_{3}\cdot(-i\nabla_{X})}\Psi\|_{L^{6}(\T^d_{h})}\nonumber\\
	&\quad\quad\quad+\|\Psi\|_{L^{6}(\T^d_{h})}\|\Phi_{Y_{2}}\|_{L^{6}(\T^d_{h})}\|e^{iY_{3}\cdot(-i\nabla_{X})}\Psi\|_{L^{6}(\T^d_{h})}+\|\Psi\|_{L^{6}(\T^d_{h})}^{2}\|\Phi_{Y_{3}}\|_{L^{6}(\T^d_{h})}\Big)d\textbf{Y}_{123}\nonumber\\
	&\leq h^{5/12} \|\Psi\|_{H^{1}_h(\T^d_{h})}^{3}\int_{\R^{3d}}|\tilde{\cG}_{V}(Y_{1},Y_{2},Y_{3})|\big(|Y_{1}|+|Y_{2}|+|Y_{3}|\big)d\textbf{Y}_{123}\nonumber\\
	&\lesssim h^{29/12}\|\Psi\|_{H^{1}_h(\T^d_{h})} \sum_{n\in\Z}\int_{\R^{7d}}|(V^{1/2}\varphi_{*})(s_{0})||(V^{1/2}\varphi_{*})(s_{1})||(V^{1/2}\varphi_{*})(s_{2})||(V^{1/2}\varphi_{*})|(s_{3})|\nonumber\\
	&\quad\quad\quad\quad\quad\quad\times G_{i\omega_{n}}^{s_{1}-s_{0}}(Y_{1})G_{-i\omega_{n}}^{-(s_{1}+s_{2})}(Y_{1}-Y_{2})G_{i\omega_{n}}^{-(s_{2}+s_{3})}(Y_{2}-Y_{3})G_{-i\omega_{n}}^{s_{0}-s_{3}}(Y_{3})\nonumber\\
	&\quad\quad\quad\quad\quad\quad\times\big(|Y_{1}|+|Y_{1}-Y_{2}|+|Y_{2}-Y_{3}|+|Y_{3}|\big)d\textbf{Y}_{123}d\textbf{s}_{0123},
\end{align}
where
\begin{align}
	G_{z}^{s}(X)&:=|G_{z}(\zeta_{X}^{s})|=|G_{z}(X+s/2)|.
\end{align}
To estimate the last integral, we note that the kernel $G_{\pm i\omega_{n}}(x)$ is a radial function and satisfies the following decay property for each $n$:
\begin{align*}
	\||x|^{s}G_{\pm i\omega_{n}}\|_{L^{p}(\R^{d})}&\leq C_{s,d}|\omega_{n}|^{-1}
\end{align*}
for $p=1,2$ and any $s\in\mathbb{N}_{0}$, where $C_{s,d}$ is some constant depends on $s$ and $d$ (since $\beta$ is fixed, we do not display it explicitly).  By the Young and Cauchy--Schwarz inequalities, it follows for each $n\in\Z$ and fixed $(s_{0},s_{1},s_{2},s_{3})$ that 
\begin{align*}
	&\int_{\R^{3d}}G_{i\omega_{n}}^{s_{1}-s_{0}}(Y_{1})G_{-i\omega_{n}}^{-(s_{1}+s_{2})}(Y_{1}-Y_{2})G_{i\omega_{n}}^{-(s_{2}+s_{3})}(Y_{2}-Y_{3})G_{-i\omega_{n}}^{s_{0}-s_{3}}(Y_{3})\nonumber\\
	&\quad\quad\quad\quad\quad\quad\times\big(|Y_{1}|+|Y_{1}-Y_{2}|+|Y_{2}-Y_{3}|+|Y_{3}|\big)d\textbf{Y}_{123}\nonumber\\
	&=\int_{\R^{d}}|Y_{1}|G_{i\omega_{n}}^{s_{1}-s_{0}}(Y_{1})\big(G_{-i\omega_{n}}^{-(s_{1}+s_{2})}*G_{i\omega_{n}}^{-(s_{2}+s_{3})}*G_{-i\omega_{n}}^{s_{0}-s_{3}}\big)(Y_{1})dY_{1}\nonumber\\
	&\quad\quad+\int_{\R^{d}}G_{i\omega_{n}}^{s_{1}-s_{0}}(Y_{1})\big((|\cdot|G_{-i\omega_{n}}^{-(s_{1}+s_{2})})*G_{i\omega_{n}}^{-(s_{2}+s_{3})}*G_{-i\omega_{n}}^{s_{0}-s_{3}}\big)(Y_{1})dY_{1}\nonumber\\
	&\quad\quad+\int_{\R^{d}}G_{i\omega_{n}}^{s_{1}-s_{0}}(Y_{1})\big(G_{-i\omega_{n}}^{-(s_{1}+s_{2})}*(|\cdot|G_{i\omega_{n}}^{-(s_{2}+s_{3})})*G_{-i\omega_{n}}^{s_{0}-s_{3}}\big)(Y_{1})dY_{1}\nonumber\\
	&\quad\quad+\int_{\R^{d}}G_{i\omega_{n}}^{s_{1}-s_{0}}(Y_{1})\big(G_{-i\omega_{n}}^{-(s_{1}+s_{2})}*G_{i\omega_{n}}^{-(s_{2}+s_{3})}*(|\cdot|G_{-i\omega_{n}}^{s_{0}-s_{3}})\big)(Y_{1})dY_{1}\nonumber\\
	&\leq \||x|G_{i\omega_{n}}^{s_{1}-s_{0}}\|_{L^{2}}\|G_{-i\omega_{n}}^{-(s_{1}+s_{2})}\|_{L^{2}}\|G_{i\omega_{n}}^{-(s_{2}+s_{3})}\|_{L^{1}}\|G_{-i\omega_{n}}^{s_{0}-s_{3}}\|_{L^{1}}\nonumber\\
	&\quad+\|G_{i\omega_{n}}^{s_{1}-s_{0}}\|_{L^{2}}\Big(\|(|x|G_{-i\omega_{n}}^{-(s_{1}+s_{2})})\|_{L^{2}}\|G_{i\omega_{n}}^{-(s_{2}+s_{3})}\|_{L^{1}}\|G_{-i\omega_{n}}^{s_{0}-s_{3}}\|_{L^{1}}\nonumber\\
	&\quad\quad\quad\quad\quad\quad\quad\quad\quad\quad+\|G_{-i\omega_{n}}^{-(s_{1}+s_{2})}\|_{L^{1}}\|(|x|G_{i\omega_{n}}^{-(s_{2}+s_{3})})\|_{L^{2}}\|G_{-i\omega_{n}}^{s_{0}-s_{3}}\|_{L^{1}}\nonumber\\
	&\quad\quad\quad\quad\quad\quad\quad\quad\quad\quad+\|G_{-i\omega_{n}}^{-(s_{1}+s_{2})}\|_{L^{1}}\|\|G_{i\omega_{n}}^{-(s_{2}+s_{3})}\|_{L^{1}}\||x|G_{-i\omega_{n}}^{s_{0}-s_{3}}\|_{L^{2}}\Big)\nonumber\\
	&\lesssim \big(1+|s_{0}|+|s_{1}|+|s_{2}|+|s_{3}|\big)|\omega_{n}|^{-4},
\end{align*}
where, in the last line, we use the estimate
\begin{align*}
	\big\||x|G_{\pm i\omega_{n}}^{s}\|_{L^{2}}&\leq \||x|G_{\pm i\omega_{n}}\|_{L^{2}}+|s|\|G_{\pm i\omega_{n}}\|_{L^{2}}.
\end{align*}
Substituting this estimate into \eqref{cNT-esti}, summing over $n$ and then applying Assumption (A1) yields the bound
\begin{align}
	\|\cN_{T}(\varphi)\|_{\cL_{h}^{2}}&\lesssim h^{29/12}\|\Psi\|_{H^{1}_h(\T^d_{h})}\|V^{1/2}\varphi_{*}\|_{L^{1}}^{3}\big(\|V^{1/2}\varphi_{*}\|_{L^{1}}+\||x|V^{1/2}\varphi_{*}\|_{L^{1}}\big)\nonumber\\
	&\lesssim h^{29/12}\|\Psi\|_{H^{1}_h(\T^d_{h})}\|V\|_{L^{1}}^{3/2}\big(\|V\|_{L^{1}}^{1/2}+\||x|^{2}V\|_{L^{1}}^{1/2}\big)\nonumber\\
	&\lesssim h^{29/12}\|\Psi\|_{H^{1}_h(\T^d_{h})}.
\end{align}

Finally, we compute $\widetilde{\Lambda}_{3}$.  By Fourier transform and omitting integral domain for simplicity in notations, we obtain
\begin{align*}
	\widetilde{\Lambda}_{3}&=-\int_{\cC}\frac{4dz}{2\pi i}\rho(\beta z)\int_{\R^{7d}} (V^{1/2}\varphi_{*})(s_{0})(V^{1/2}\varphi_{*})(s_{1})(V^{1/2}\varphi_{*})(s_{2})(V^{1/2}\varphi_{*})(s_{3})\nonumber\\
	&\quad\times G_{z}(\zeta_{Y_{1}}^{s_{1}-s_{0}})G_{-z}(\zeta_{Y_{1}-Y_{2}}^{-(s_{1}+s_{2})})G_{z}(\zeta_{Y_{2}-Y_{3}}^{-(s_{2}+s_{3})})G_{-z}(\zeta_{Y_{3}}^{s_{0}-s_{3}})d\textbf{Y}_{123}d\textbf{s}_{0123}\nonumber\\
	&=-\int_{\cC}\frac{4dz}{2\pi i}\rho(\beta z)\int_{\R^{11d}}(V^{1/2}\varphi_{*})(s_{0})(V^{1/2}\varphi_{*})(s_{1})(V^{1/2}\varphi_{*})(s_{2})(V^{1/2}\varphi_{*})(s_{3})\nonumber\\
	\displaybreak
	&\times\frac{e^{ip_{0}\cdot \zeta_{Y_{1}}^{s_{1}-s_{0}}}e^{ip_{1}\cdot\zeta_{Y_{1}-Y_{2}}^{-(s_{1}+s_{2})}}e^{ip_{2}\cdot\zeta_{Y_{2}-Y_{3}}^{-(s_{2}+s_{3})}}e^{ip_{3}\cdot \zeta_{Y_{3}}^{s_{0}-s_{3}}}d\textbf{Y}_{123}d\textbf{s}_{0123}d\textbf{p}_{0123}}{(2P)^{4d}\big(|p_{0}|^{2}-(\mu+\omega_{n})\big)\big(|p_{1}|^{2}-(\mu-\omega_{n})\big)\big(|p_{2}|^{2}-(\mu+\omega_{n})\big)\big(|p_{3}|^{2}-(\mu-\omega_{n})\big)}\nonumber\\
	&=-\int_{\cC}\frac{4dz}{2\pi i}\rho(\beta z)\int_{\R^{5d}}(V^{1/2}\varphi_{*})(s_{0})(V^{1/2}\varphi_{*})(s_{1})(V^{1/2}\varphi_{*})(s_{2})(V^{1/2}\varphi_{*})(s_{3})\nonumber\\
	&\quad\times\frac{e^{-ip\cdot(s_{0}-s_{1})/2}}{|p|^{2}-(\mu+z)}\frac{e^{ip\cdot(s_{1}+s_{2})/2}}{|p|^{2}-(\mu-z)}\frac{e^{ip\cdot(s_{2}+s_{3})/2}}{|p|^{2}-(\mu+z)}\frac{e^{-ip\cdot(s_{0}-s_{3})/2}}{|p|^{2}-(\mu-z)}d\textbf{s}_{0123}dp\nonumber\\
	&=\frac{1}{16(2\pi)^{d}}\int_{\R^{d}}t_{*}(p)^{4}\Big(\frac{\tanh\big(\beta(|p|^{2}-\mu)/2\big)}{2(|p|^{2}-\mu)^{3}}-\frac{\beta}{4(|p|^{2}-\mu)^{2}\cosh^{2}\big(\beta(|p|^{2}-\mu)/2\big)}\Big)dp.
\end{align*}
Since the integrand above is Lipschitz, it implies that $\widetilde{\Lambda}_{3}=\Lambda_{3}+O(h^{2})$, which completes the proof.
\end{proof}



\subsection{Putting everything together}

We are finally in position to prove our first main result.

\begin{proof}[Proof of Theorem \ref{thm:deriv-GL}]
    Let $\alpha$ be a solution of the BdG equation \eqref{BdG-eqn3a} satisfying \eqref{small-data-deriv}. Then, as discussed in Section \ref{sec:bs}, $\varphi:=V^{1/2}\alpha$ is a solution of the BdG equation \eqref{BdG-BS-eqn} in Birman--Schwinger form and, by Lemma \ref{lem:Hsnorm-Delta}, it satisfies $\|\varphi\|_{\cH^1_h}\lesssim h$.

    We define $\Psi$ by \eqref{GL-fun1} and note that, by \eqref{eq:normpsiphi},
    $$
    \|\Psi\|_{H^{1}_h(\T^d_{h})} = \|P_\kappa\varphi\|_{\cH_{h}^{1}}\lesssim h \,.
    $$
    In particular, $\Psi\in H^1_h(\T^d_h)$ and therefore
    \begin{align}\label{GL-f}
	   F_{T}^{\rm GL,sc}(\Psi) := \big(-\nabla_{X}\cdot\Lambda_{0}\nabla_{X}-\Lambda_{2}Dh^{2}\big)\Psi+\Lambda_{3}|\Psi|^{2}\Psi\in H_{h}^{-1}(\T^d_{h}),
    \end{align}
    where the matrix $\Lambda_{0}$ and constants $\Lambda_{2},\Lambda_{3}$ are given in \eqref{eq:lambda0}--\eqref{eq:lambda3}. Because of the momentum cut-off, we also know that $\Psi\in H^s_h(\T^d)$ for any $s>0$, but at this stage of the proof our norm bounds for $s>1$ are not optimal.

    Let us define
    \begin{align}\label{approx-BdG}
        F_{T}^{\rm BCS,app}(\varphi) := P_{\kappa}L_{T}P_{\kappa}\varphi+\tfrac{1}{2}PV^{1/2}N_{T}'(-2V^{1/2}P_{\kappa}\varphi).
    \end{align}
    Then, by \eqref{BdG-eqn8} and Proposition \ref{prop:small-NP-diff},
    \begin{align}\label{appBdG-eqn}
    \|F_{T}^{\rm BCS,app}(\varphi)\|_{\cH_{h}^{-1}}
    & \leq \big\|\tfrac{1}{2}\lambda_{\kappa}^{\perp}PV^{1/2}N_{T}'(-2V^{1/2}P_{\kappa}\varphi)\big\|_{\cH_{h}^{-1}}+\|P_{\kappa}L_{T}P_{\kappa}^{\perp}\varphi\|_{\cL_{h}^{2}}\nonumber\\
    &\quad\quad+\|\tfrac12 P_{\kappa}V^{1/2}N_{T}'(-2V^{1/2}\varphi)-\tfrac12 P_{\kappa}V^{1/2}N_{T}'(-2V^{1/2}P_{\kappa}\varphi)\|_{\cL_{h}^{2}}\nonumber\\
    &\quad\quad\quad\|\tfrac12 P_{\kappa}V^{1/2}\widetilde{N}_{T}(-2V^{1/2}\varphi)\|_{\cL_{h}^{2}}\nonumber\\
	&\lesssim \big(h^{9/4}+h^{13/6}\big)\|\Psi\|_{H^{1}_h(\T^d_{h})}\lesssim h^{13/6}\|\Psi\|_{H^{1}_h(\T^d_{h})},
    \end{align}
    and, by Propositions \ref{prop:GL-matrix} and \ref{prop:GL-nonlinear},
    \begin{align}\label{appBdG-scaleGL-diff}
	\|F_{T}^{\rm BCS,app}(\varphi)&-F_{T}^{\rm GL,sc}(\Psi)\varphi_{*}\|_{\cL_{h}^{2}}\nonumber\\
	&\lesssim \big(h^{9/4}+h^{29/12}\big)\|\Psi\|_{H^{1}_h(\T^d_{h})}\lesssim h^{9/4}\|\Psi\|_{H^{1}_h(\T^d_{h})}.
    \end{align}
    By the triangle inequality \eqref{appBdG-eqn} and \eqref{appBdG-scaleGL-diff} imply
    \begin{align}\label{GL-sc-norm}
	h^{-2}&\|F_{T}^{\rm GL,sc}(\Psi)\|_{H^{-1}_h(\T^d_{h})}=h^{-2}\|F_{T}^{\rm GL,sc}(\Psi)\varphi_{*}\|_{\cL_{h}^2}\nonumber\\
	&\leq h^{-2}\|F_{T}^{\rm BCS,app}(\varphi)-F_{T}^{\rm GL,sc}(\Psi)\varphi_{*}\|_{\cL_{h}^{2}}+h^{-2}\|F_{T}^{\rm BCS,app}(\varphi)\|_{\cH_{h}^{-1}}\nonumber\\
	&\lesssim h^{1/6}\|\Psi\|_{H^{1}_h(\T^d_{h})}.
    \end{align}

    Now we rescale and define the $\Z^d$-periodic function $\psi$ by
    $$
    \psi(X):= h^{-1}\Psi(X/h) \,.
    $$
    It follows that
    $$
    \|\psi\|_{L^2(\T^d)} = h^{-1} \|\Psi\|_{L^2_h(\T^d_h)} \lesssim 1 \,,
    \quad
    \|\psi\|_{H^1(\T^d)} = h^{-1} \|\Psi\|_{H^1_h(\T^d_h)}
    $$
    and
    $$
    \big\|\big(-\nabla\cdot \Lambda_0\nabla-\Lambda_{2}D\big)\psi+\Lambda_{3}|\psi|^{2}\psi \big\|_{H^{-1}(\T^d)}
    = h^{-3} \big\|F_{T}^{\rm GL,sc}(\Psi) \big\|_{H_{h}^{-1}(\T^d_{h})}
    \lesssim h^{1/6} \|\psi\|_{H^1(\T^d)} \,.
    $$
    This proves one of the bounds in \eqref{GL-esti}, as well as \eqref{GL-esti1}, except that the latter is stated with $\|\psi\|_{L^{2}(\T^d)}$ in place of $\|\psi\|_{H^{1}(\T^d)}$.

    To prove the first bound in \eqref{GL-esti}, it suffices to bound $\|\Psi\|_{H^{1}_h(\T^d_{h})}$ by $\|\Psi\|_{L^{2}_h(\T^d_{h})}$. Let $f:=F_{T}^{\rm GL,sc}(\Psi)$.  By multiplying both sides by $\overline{\Psi}$ and then integrating over $\T^d_{h}$, we obtain
\begin{align}
	\int_{\T^d_{h}}\overline{\nabla_{X}\Psi(X)}\cdot&(\Lambda_{0}\nabla_{X}\Psi)(X)\,dX+\Lambda_{3}\int_{\T^d_{h}}|\Psi(X)|^{4}\,dX\nonumber\\
	&=\Lambda_{2}Dh^{2}\int_{\T^d_{h}}|\Psi(X)|^{2}\,dX+\int_{\T^d_{h}}\overline{\Psi(X)}f(X)\,dX.
\end{align}
Let $\nu_{\min}>0$ denotes the smallest eigenvalue of $\Lambda_0$ (recall that $\Lambda_0$ is positive definite, as shown in \cite[Sec.~1.4]{FHSS-12}). Dropping the non-negative quartic term on the l.h.s.~and applying the Cauchy--Schwarz inequality yields
\begin{align}\label{GL-H1-L2}
	\nu_{\rm min}\|\nabla_{X}\Psi\|_{L_{h}^{2}(\T_{h}^{d})}^{2}&\leq h^{d} \int_{\T_{h}^{d}}\overline{\nabla_{X}\Psi(X)}\cdot\Lambda_{0}\nabla_{X}\Psi(X)\,dX\nonumber\\
	&\leq \Lambda_{2}Dh^{2}\|\Psi\|_{L_{h}^{2}(\T_{h}^{d})}^{2}+\|f\|_{H_{h}^{-1}(\T_{h}^{d})}\|\Psi\|_{H_{h}^{1}(\T_{h}^{d})}\nonumber\\
	&\lesssim h^{2}\|\Psi\|_{L_{h}^{2}(\T_{h}^{d})}^{2}+h^{13/6}\|\Psi\|_{H_{h}^{1}(\T_{h}^{d})}^{2}\nonumber\\
	&\lesssim h^{2}\|\Psi\|_{L_{h}^{2}(\T_{h}^{d})}^{2}+h^{13/6}\|\Psi\|_{L_{h}^{2}(\T_{h}^{d})}^{2}+h^{1/6}\|\nabla_{X}\Psi\|_{L_{h}^{2}(\T_{h}^{d})}^{2}\nonumber\\
	&\lesssim h^{2}\|\Psi\|_{L_{h}^{2}(\T_{h}^{d})}^{2}+h^{1/6}\|\nabla_{X}\Psi\|_{L_{h}^{2}(\T_{h}^{d})}^{2}.
\end{align}
Here we used the bound \eqref{GL-sc-norm} for $f$. For all sufficiently small $h>0$, the last term on the r.h.s.~side can be absorbed into the l.h.s.~ and we arrive at the bound
\begin{align}\label{eq:gradbound}
	h^{-2}\|\nabla_{X}\Psi\|_{L^{2}_h(\T^d_{h})}^{2}&\lesssim \|\Psi\|_{L^{2}_h(\T^d_{h})}^{2}.
\end{align}
This implies the first bound in \eqref{GL-esti}. In particular, this proves \eqref{GL-esti1}.

    To complete the proof of Theorem \ref{thm:deriv-GL}, it remains to prove the bound on $\xi$, defined in \eqref{alpha-decomposition}. Since $\varphi_{*}=V^{1/2}\alpha_{*}$ and $\alpha_{*}$ obeys the eigenvalue equation
    \begin{align*}
        \alpha_{*}&=k_{T_{c}}V\alpha_{*}=k_{T_{c}}V^{1/2}\varphi_{*},
    \end{align*}
    we note that
    \begin{align}\label{xi}
        h\psi(hX)\alpha_{*}(r)&=\Psi(X)\alpha_{*}(r)=(k_{T_{c}}V^{1/2}\varphi_{*})(r)\lambda_{\kappa}\int_{\R^{d}}\varphi_{*}(s)\varphi(\zeta_{X}^{s},\zeta_{X}^{-s})ds\nonumber\\
        &=(k_{T_{c}}V^{1/2}P_{\kappa}\varphi)(\zeta_{X}^{r},\zeta_{X}^{-r}) \,,
    \end{align}
    that is,
    $$
    \xi = \alpha - k_{T_{c}}V^{1/2}P_{\kappa}\varphi \,.
    $$
    Recall from Section \ref{sec:bs} that $\alpha$ and $\varphi$ are related by 
    \begin{align*}
	   \varphi&=V^{1/2}\alpha,\quad\alpha=K_{T}V^{1/2}\varphi-\tfrac{1}{2}N_{T}(-2V^{1/2}\varphi) \,.
    \end{align*}
    It follows that 
    \begin{align}\label{xi2}
	   \xi &=K_{T}V^{1/2}\varphi-\tfrac{1}{2}N_{T}(-2V^{1/2}\varphi)-k_{T_{c}}V^{1/2} P_{\kappa}\varphi\nonumber\\
	   &=(K_{T}-k_{T_{c}})V^{1/2}P_{\kappa}\varphi+K_{T}V^{1/2}P_{\kappa}^{\perp}\varphi-\tfrac{1}{2}N_{T}(-2V^{1/2}\varphi).
    \end{align}
    Let us bound the three terms on the r.h.s.
    
    For the last term we use Proposition \ref{prop:BCS-NL-norm} and Lemma \ref{lem:Lpnorm-Delta} and obtain
    $$
    \| \tfrac{1}{2}N_{T}(-2V^{1/2}\varphi) \|_{\cH^1_h} \lesssim h^{-1} \|V^{1/2} \varphi\|_{\cL^6_h}^3 \lesssim h^{-1} \|\varphi\|_{\cH^1_h}^3 \lesssim h^2\|\psi\|_{L^{2}(\T^{d})} \,.
    $$
    For the second term we use the fact that $K_{T}V^{1/2}$ is a bounded operator on $\cL_{h}^{2}$ that commutes with $-\Delta_{X}$. Using Lemma \ref{lem:Pper-Pbdd}, we obtain
    \begin{align}
	   \|K_{T}V^{1/2}P_{\kappa}^{\perp}\varphi\|_{\cH_{h}^{1}}&=\|(\one-h^{-2}\Delta_{X})^{1/2}K_{T}V^{1/2}P_{\kappa}^{\perp}\varphi\|_{\cL_{h}^{2}}\nonumber\\
	   &\leq \|K_{T}V^{1/2}\|_{\cL_{h}^{2}\rightarrow\cL_{h}^{2}}\|(\one-h^{-2}\Delta_{X})P_{\kappa}^{\perp}\varphi\|_{\cL_{h}^{2}}\nonumber\\
	   &\lesssim\|P_{\kappa}^{\perp}\varphi\|_{\cH_{h}^{1}}\lesssim h^{7/6}\|\psi\|_{L^{2}(\T^{d})}.
    \end{align}
    For the first term in \eqref{xi2}, we use Lemmas \ref{lem:KTkT-diff} and \ref{lem:kT-diff} and the fact that both $K_{T}$ and $k_{T}$ commutes with $-\Delta_{X}$ for any $T$ to obtain
    \begin{align}
	\big\|(K_{T}-k_{T_{c}})V^{1/2}P_{\kappa}\varphi\big\|_{\cH_{h}^{1}}&\leq \big\|(K_{T}-k_{T})V^{1/2}P_{\kappa}\varphi\big\|_{\cH_{h}^{1}}+\big\|(k_{T}-k_{T_{c}})V^{1/2}P_{\kappa}\varphi\big\|_{\cH_{h}^{1}}\nonumber\\
	&=\big\|(K_{T}-k_{T})V^{1/2}(\one-h^{-2}\Delta_{X})^{1/2}P_{\kappa}\varphi\big\|_{\cL_{h}^{2}}\nonumber\\
	&\quad\quad\quad\quad+\big\|(k_{T}-k_{T_{c}})V^{1/2}(\one-h^{-2}\Delta_{X})^{1/2}P_{\kappa}\varphi\big\|_{\cL_{h}^{2}}\nonumber\\
	&\lesssim\Big\|\frac{\Delta_{X}}{\one-\Delta_{X}}P_{\kappa}\varphi\Big\|_{\cH_{h}^{1}}+(\beta-\beta_{c})\big\|P_{\kappa}\varphi\|_{\cH_{h}^{1}}\nonumber\\
	&\lesssim(\kappa+h^{2})\|P_{\kappa}\varphi\|_{\cH_{h}^{1}}\lesssim h^{11/6}\|\psi\|_{L^{2}(\T^{d})}.
    \end{align}

    Combining the bounds on the three terms in \eqref{xi2}, we get $\|\xi \|_{\cH_{h}^{1}}\lesssim h^{7/6}\|\psi\|_{L^{2}(\T^{d})}$. This completes the proof of Theorem \ref{thm:deriv-GL}.
\end{proof}

\section{Proof on Theorem \ref{thm:deriv-BdG}}\label{sec:pf-deriv-BdG}
In this section, we prove another version of Theorem \ref{thm:deriv-BdG}.  The proof follows from a similar logic as for the proof of Theorem \ref{thm:deriv-GL}, except in reverse order, together with an application of second order perturbation theory.

\smallskip


\subsection{A priori estimates on solution to Ginzburg-Landau equation}\label{sec:a-priori-GL}

Before we construct our trial state, we need to establish some elementary estimates on solution of the GL equation.  Let $\psi\in H^{1}(\T^{d})$ be a solution to \eqref{GL-eqn}.  Then, by standard regularity argument, we see that $\psi$ is smooth.  

In order to the similarity to the proof of Theorem \ref{thm:deriv-GL}, we will work with the function
\begin{align}\label{Psi-trial}
	\Psi(X) :=h \, \psi(hX),\quad\text{ for }X\in\T_{h}^{d} \,.
\end{align}
Thus, $\Psi$ is smooth, periodic on $\T_{h}^{d}$ and satisfies the estimate
\begin{align}\label{GL-scaling}
	\|\Psi\|_{H_{h}^{s}(\T_{h}^{d})}&=h\|\psi\|_{H^{s}(\T^{d})}\lesssim_s h
    \qquad\text{for all}\ s
\end{align}
and the rescaled GL equation 
\begin{align}\label{GL-eqn-sc}
	F_{T}^{\rm GL,sc}(\Psi)= 0 \,,
\end{align}
where
\begin{align*}
	F_{T}^{\rm GL,sc}(\Psi)=\big(-\nabla_{X}\cdot\Lambda_{0}\nabla_{X}-\Lambda_{2}Dh^{2}\big)\Psi+\Lambda_{3}|\Psi|^{2}\Psi
\end{align*}
was defined in \eqref{GL-f}.  

Then we have the following elementary lemma:

\begin{lemma}\label{lem:GL-H2-norm}
Let $\psi$ be a solution to \eqref{GL-eqn}.  Then it holds that
\begin{align}\label{GL-ortho-norm}
	\|\psi\|_{H^{2}(\T^{d})} \lesssim \|\psi\|_{H^{1}(\T^{d})}\lesssim \|\psi\|_{L^{2}(\T^{d})} \lesssim 1\,.
\end{align}
\end{lemma}

\begin{proof}
Testing the GL equation with $\psi$ gives
\begin{align*}
    \|\nabla\psi \|_{L^2(\T^d)}^2 + \|\psi\|_{L^4(\T^d)}^4 & \lesssim \|\Lambda_0^{1/2}\nabla\psi\|_{L^2(\T^d)}^2 + \Lambda_3 \|\psi\|_{L^4(\T^d)}^4 = \Lambda_2 D \|\psi\|_{L^2(\T^d)}^2 \\
    & \lesssim \|\psi\|_{L^2(\T^d)}^2 \,.    
\end{align*}
In particular, this gives $\|\nabla\psi \|_{L^2(\T^d)}\lesssim \|\psi\|_{L^2(\T^d)}$, which gives the second inequality in \eqref{GL-ortho-norm}. It also gives, combined with H\"older's inequality,
$$
\|\psi\|_{L^2(\T^d)} \leq \|\psi\|_{L^4(\T^d)} \lesssim \|\psi\|_{L^2(\T^d)}^{1/2} \,,
$$
so $\|\psi\|_{L^2(\T^d)}\lesssim 1$, which is the third inequality in \eqref{GL-ortho-norm}.

To proceed, we test the GL equation with $\nabla\cdot\Lambda_0\nabla\psi$ and obtain, after a short computation,
\[
\|\Delta\psi\|_{L^2(\T^d)}^2 \lesssim \|\nabla\cdot\Lambda_0\nabla\psi\|_{L^2(\T^d)}^2 \leq \Lambda_2 D \|\Lambda_0^{1/2}\nabla\psi\|_{L^2(\T^d)}^2\lesssim \|\nabla\psi \|_{L^2(\T^d)}^2 \,.
\]
This gives the first inequality in \eqref{GL-ortho-norm} and concludes the proof.
\end{proof}

Furthermore, with $\kappa$ is chosen as in \eqref{kappa-fix}, we decompose
\begin{align}
	\Psi=\Psi_{<}+\Psi_{>},\quad\text{where } \Psi_{<} :=\lambda_{\kappa}\Psi,\quad \Psi_{>} :=\lambda_{\kappa}^{\perp}\Psi \,.
\end{align} 

\begin{lemma}\label{lem:small-H1-largeGL}
If $\psi\in H^2(\T^d)$, then
\begin{align}\label{small-H1-largeGL}
	\|\Psi_{>}\|_{H_{h}^{1}(\T_{h}^{d})}&\lesssim h^{19/12}\|\psi\|_{H^{2}(\T^{d})}.
\end{align}
\end{lemma}

We emphasize that in this lemma, we do not need the GL equation for $\psi$. In fact, if $\psi$ solve this equation, then $\psi\in H^s(\T^d)$ for arbitrary large $s$ and then a small variation of the argument below shows that $\|\Psi_{>}\|_{H_{h}^{1}(\T_{h}^{d})}\lesssim_N h^N$ for any $N$. For us, however, \eqref{small-H1-largeGL} suffices.

\begin{proof}
On the range of $\lambda_{\kappa}^{\perp}$, we know that $-h^{-2}\Delta_{X}\geq \kappa h^{-2}=h^{-7/6}$ so that, by Fourier transform,
\begin{align*}
	\big\|\Psi_{>}\big\|_{H_{h}^{1}(\T_{h}^{d})}^{2}&=\big\|(\one-h^{-2}\Delta_{X})^{1/2}\lambda_{\kappa}^{\perp}\Psi\big\|_{L_{h}^{2}(T_{h}^{d})}^{2}\lesssim h^{-2}\big\|\nabla_{X}\Psi\|_{L_{h}^{2}(T_{h}^{d})}^{2}\nonumber\\
	&\lesssim \sum_{q\in 2\pi\Z^{d}}\one(h^{2}|q|^{2}>\kappa)|q|^{2}|\hat{\Psi}(q)|^{2}\nonumber\\
	&\lesssim h^{2}\kappa^{-1}\sum_{q\in 2\pi\Z^{d}}\one(h^{2}|q|^{2}>\kappa)|q|^{4}|\hat{\Psi}(q)|^{2}\lesssim h^{7/6}\|\Psi\|_{H_{h}^{2}(\T_{h}^{d})}^{2}.
\end{align*}
It then follows from \eqref{GL-scaling} with $s=2$ that $\|\Psi_{>}\|_{H_{h}^{1}(\T_{h}^{d})}^2 \lesssim h^{19/6}\|\psi\|_{H^{2}(\T^{d})}^{2}$.
\end{proof}

\begin{lemma}\label{lem:GL-eqn-Psi<}
Let $\psi$ be a solution of the GL equation \eqref{GL-eqn}. Then it holds that
\begin{align}
	\|\lambda_{\kappa}F_{T}^{\rm GL,sc}(\Psi_{<})\|_{L_{h}^{2}(T_{h}^{d})}&\lesssim h^{43/12}\|\psi\|_{L^{2}(\T^{d})}.
\end{align}
\end{lemma}

\begin{proof}
By projecting \eqref{GL-eqn-sc} onto $\Ran(\lambda_{\kappa})$ 
, we obtain
\begin{align}
	\label{GL-eqn-sc1}0&=(-\nabla_{X}\cdot\Lambda_{0}\nabla_{X}-\Lambda_{2}Dh^{2})\Psi_{<}+\Lambda_{3}\lambda_{\kappa}(|\Psi|^{2}\Psi) \,.
\end{align} 
This implies
\begin{align}
	\lambda_{\kappa}F_{T}^{\rm GL,sc}(\Psi_{<})&=(-\nabla_{X}\cdot\Lambda_{0}\nabla_{X}-\Lambda_{2}Dh^{2})\Psi_{<}+\Lambda_{3}\lambda_{\kappa}\big(|\Psi_{<}|^{2}\Psi_{<}\big)\nonumber\\
	&=\Lambda_{3}\lambda_{\kappa}\big(|\Psi_{<}|^{2}\Psi_{<}-|\Psi|^{2}\Psi\big)\nonumber\\
	&=-\Lambda_{3}\lambda_{\kappa}\Big(|\Psi_{<}|^{2}\Psi_{>}+\big(|\Psi_{>}|^{2}+2\Re(\overline{\Psi_{<}}\Psi_{>})\big)\Psi\Big).
\end{align}
It follows from the Sobolev inequality, \eqref{GL-scaling} (with $s=1$), Lemmas \ref{lem:GL-H2-norm} and \ref{lem:small-H1-largeGL} that
\begin{align*}
	\big\|\lambda_{\kappa}F_{T}^{\rm GL,sc}(\Psi_{<})\big\|_{L_{h}^{2}(\T_{h}^{d})}&\lesssim\||\Psi_{<}|^{2}\Psi_{>}\|_{L_{h}^{2}(\T_{h}^{d})}+\|(|\Psi_{>}|^{2}+2\Re(\overline{\Psi_{<}}\Psi_{>}))\Psi\|_{L_{h}^{2}(\T_{h}^{d})}\nonumber\\
	&\lesssim\|\Psi_{<}\|_{L_{h}^{6}(\T_{h}^{d})}^{2}\|\Psi_{>}\|_{L_{h}^{2}(\T_{h}^{d})}+\|\Psi_{>}\|_{L_{h}^{6}(\T_{h}^{d})}^{2}\|\Psi\|_{L_{h}^{6}(\T_{h}^{d})}\nonumber\\
	&\quad\quad\quad+\|\Psi_{<}\|_{L_{h}^{6}(\T_{h}^{d})}\|\Psi_{>}\|_{L^{6}(\T_{h}^{d})}\|\Psi\|_{L^{6}(\T_{h}^{d})}\nonumber\\
	&\lesssim\|\Psi\|_{H_{h}^{1}(\T_{h}^{d})}^{2}\|\Psi_{>}\|_{H_{h}^{1}(\T_{h}^{d})}\lesssim h^{19/12+2}\|\psi\|_{H^{1}(\T^{d})}^2\|\psi\|_{H^{2}(\T^{d})}\nonumber\\
    &\lesssim h^{43/12} \|\psi\|_{L^{2}(\T^{d})}.
\end{align*}
This completes the proof.
\end{proof}


\subsection{Construction of trial state}\label{sec:trial}

In this subsection, given a solution $\psi\in H^{1}(\T^{d})$ of the GL equation \eqref{GL-eqn}, we construct an operator $\alpha_\psi$ which, as we shall see in the next subsection, is an approximate solution of the BdG equation \eqref{BdG-eqn3a}. Here we will prove some properties of $\alpha_\psi$ that will be crucial in the next subsection.

First, to simplify notations, we abbreviate, for $\varphi\in\cL_{h}^{2}$,
\begin{align}
    F_{T}^{\rm BCS}(\varphi)&:=L_{T}\varphi+\tfrac{1}{2}V^{1/2}N_{T}(-2V^{1/2}\varphi),
\end{align}
where we recall $L_{T}=\one-V^{1/2}K_{T}V^{1/2}$.  Note that this is the l.h.s.~of \eqref{BdG-BS-eqn}. Moreover, we shall use the operators $P_\kappa$ and $P_{\kappa}^{\perp}=\one-P_{\kappa}$ from \eqref{P}. The parameter $\kappa$ will be chosen as in \eqref{kappa-fix}. Under Assumptions (A1)--(A5), which we assume throughout, we know from Proposition \ref{prop:inv-bdd} that the operator $P_{\kappa}^{\perp}L_{T}P_{\kappa}^{\perp}$ is boundedly invertible on $\Ran(P_{\kappa}^{\perp})$.

For $\psi\in H^1(\T^d)$, we define $\varphi_{\psi}\in\cL_{h}^{2}$ by
\begin{align}\label{eq:defphipsi}
    \varphi_{\psi}&:=\varphi_{\psi,0}+\varphi_{\psi,1} \,,
\end{align}
where
\begin{align}
    \varphi_{\psi,0}(\zeta_{X}^{r},\zeta_{X}^{-r})&:= h\,\psi(hX)\,\varphi_{*}(r) \,, \quad
    \varphi_{\psi,1}:=-(P_{\kappa}^{\perp}L_{T}P_{\kappa}^{\perp})^{-1} P_\kappa^\perp F_{T}^{\rm BCS}(\varphi_{\psi,0}) \,.
\end{align}
Then we define $\alpha_\psi\in\cL_{h}^{2}$ by
\begin{align}\label{trial}
    \alpha_\psi &:=-\frac{1}{2}\Big[\tanh\Big(\frac{\beta H(-2V^{1/2}\varphi_{\psi})}{2}\Big)\Big]_{12}=K_{T}V^{1/2}\varphi_{\psi}-\tfrac{1}{2}N_{T}(-2V^{1/2}\varphi_{\psi}) \,.
\end{align}
Note that, setting
$$
\gamma_\psi := \frac12 - \frac{1}{2}\Big[\tanh\Big(\frac{\beta H(-2V^{1/2}\varphi_{\psi})}{2}\Big)\Big]_{11} \,,
$$
we obtain a pair of operators $(\gamma_\psi,\alpha_\psi)$ on $L^2(\R^d)$ that satisfies \eqref{BdG-1pdm}; see Appendix~\ref{app:derbdg}.

The main result of this subsection is the following proposition:

\begin{prop}\label{prop:trial-property}
Let $\psi$ be a solution of the GL equation \eqref{GL-eqn}.  Then
\begin{align}\label{trial-property}
	\big\|\varphi_{\psi}-V^{1/2}\alpha_{\psi}\big\|_{\cL_{h}^2}& \lesssim h^{19/6}\|\psi\|_{L^{2}(\T^{d})}.
\end{align}
\end{prop}

Before proceeding to the proof of Proposition \ref{prop:trial-property}, we study the components $\varphi_{\psi,0}$ and $\varphi_{\psi,1}$ in the next lemma:

\begin{lemma}\label{lem:ortho-GL-regularity}
We have 
\begin{align}\label{eq:phipsi0}
	\|\varphi_{\psi,0}\|_{\cH_{h}^{1}}&\lesssim h\|\psi\|_{H^{1}(\T^{d})} \,.
\end{align}
Moreover, if $\psi\in H^2(\T^d)$, then it holds for $s=0,1$ that
\begin{align}\label{ortho-GL-regularity}
	\big\|F_{T}^{\rm BCS}(\varphi_{\psi,0})\big\|_{\cH_{h}^{s}}&\lesssim h^{3-s}\left( \|\psi\|_{H^{2}(\T^{d})} + \|\psi\|_{H^{1}(\T^{d})}^3 \right)
\end{align}
and
\begin{align}\label{ortho-GL-regularity2}
	\|\varphi_{\psi,1}\|_{\cH_{h}^{s}}&\lesssim h^{13/6-s} \left( \|\psi\|_{H^{2}(\T^{d})} + \|\psi\|_{H^{1}(\T^{d})}^3 \right).
\end{align}
In particular,
\begin{align}
    \label{eq:phipsisize}
    \|\varphi_\psi\|_{\mathcal H^1_h} \lesssim h\left( \|\psi\|_{H^{2}(\T^{d})} + \|\psi\|_{H^{1}(\T^{d})}^3 \right).
\end{align}
\end{lemma}

\begin{proof}
The bound \eqref{eq:phipsi0} is clear since, due to Lemma \ref{lem:GL-H2-norm},
\begin{align}
    \label{eq:phipsinorm}
    \|\varphi_{\psi,0}\|_{\cH_{h}^{s}}=h\|\psi\|_{H^{s}(\T^{d})}\lesssim h\|\psi\|_{H^{1}(\T^{d})}
    \qquad\text{for}\ s=0,1 \,.
\end{align}

Next, we prove \eqref{ortho-GL-regularity}. Since $\ell_{T_{c}}\varphi_{*}=0$, we can write
\begin{align*}
	F_{T}^{\rm BCS}(\varphi_{\psi,0})&=(L_{T}-\ell_{T_{c}})\varphi_{\psi,0}+\tfrac{1}{2}V^{1/2}N_{T}(-2V^{1/2}\varphi_{\psi,0}).
\end{align*}
The r.h.s.~consists of a linear and a nonlinear part, which we bound separately. For the linear part, by Lemmas \ref{lem:KTkT-diff}, \ref{lem:kT-diff} and \ref{lem:Hsnorm-Delta}, we obtain
\begin{align*}
	&\big\|(L_{T}-\ell_{T_{c}})\varphi_{\psi,0}\big\|_{\cH_{h}^{s}}=\big\|V^{1/2}(K_{T}-k_{T_{c}})V^{1/2}\varphi_{\psi,0}\big\|_{\cH_{h}^{s}}\nonumber\\
    &\leq \big\|V^{1/2}(K_{T}-K_{T_{c}})V^{1/2}\varphi_{\psi,0}\big\|_{\cH_{h}^{s}}+\big\|V^{1/2}(K_{T_{c}}-k_{T_{c}})V^{1/2}\varphi_{\psi,0}\big\|_{\cH_{h}^{s}}\nonumber\\
    &\lesssim\Big\|\frac{\Delta_{X}}{\one-\Delta_{X}}\varphi_{\psi,0}\Big\|_{\cH_{h}^{s}}+h^{2}\|\varphi_{\psi,0}\big\|_{\cH_{h}^{s}}\lesssim h^{-s}\big\|\Delta_{X}\varphi_{\psi,0}\big\|_{\cL_{h}^{2}}+h^{2}\|\varphi_{\psi,0}\|_{\cH_{h}^{s}}\,.
\end{align*}
In the last inequality, we used
$$
\sup_{q\in 2\pi\Z^d} \frac{(1+|q|^2)^s}{(1+ h^2|q|^2)^2} \lesssim h^{-2s} \,.
$$
As in \eqref{eq:phipsinorm}, using Lemma \ref{lem:GL-H2-norm}, we find
\begin{align*}
	& h^{-s}\big\|\Delta_{X}\varphi_{\psi,0}\big\|_{\cL_{h}^{2}}+h^{2}\|\varphi_{\psi,0}\|_{\cH_{h}^{s}} = h^{3-s} \|\Delta \psi \|_{L^2(\T^d)} + h^3 \|\psi\|_{H^2(\T^d)} \lesssim h^{3-s}\|\psi\|_{H^{2}(\T^{d})} \,.
\end{align*}
Combining these inequality gives
\begin{align}
    \label{eq:trialroughlinear}
    \big\|(L_{T}-\ell_{T_{c}})\varphi_{\psi,0}\big\|_{\cH_{h}^{s}} \lesssim h^{3-s}\|\psi\|_{H^{2}(\T^{d})} \,,
\end{align}
which is the claimed bound for the linear part.

For the nonlinear part, by Lemmas \ref{lem:Lpnorm-Delta}, \ref{lem:Hsnorm-Delta}, \ref{lem:GL-H2-norm}, Proposition \ref{prop:BCS-NL-norm} and \eqref{eq:phipsinorm}, we have
\begin{align*}
	\big\|V^{1/2}N_{T}(-2V^{1/2}\varphi_{\psi,0})\big\|_{\cH_{h}^{s}}&\lesssim h^{-s}\|V^{1/2}\varphi_{\psi,0}\|_{\cL_{h}^{6}}^{3}\lesssim\|\varphi_{\psi,0}\|_{\cH_{h}^{1}}^{3} = h^{3-s} \|\psi\|_{H^1(\T^d)}^3.
\end{align*}
This proves \eqref{ortho-GL-regularity}.

Next, by Proposition \ref{prop:inv-bdd}, the fact that $P_{\kappa}^{\perp}L_{T}P_{\perp}$ commutes with $-\Delta_{X}$ and estimate \eqref{ortho-GL-regularity}, we obtain immediately for $s=0,1$ that
\begin{align*}
	\big\|\varphi_{\psi,1}\big\|_{\cH_{h}^{s}}&=\big\|(P_{\kappa}^{\perp}L_{T}P_{\kappa}^{\perp})^{-1} P_\kappa^\perp F_{T}^{\rm BCS}(\varphi_{\psi,0})\big\|_{\cH_{h}^{s}}\nonumber\\
	&\leq \big\|(P_{\kappa}^{\perp}L_{T}P_{\kappa}^{\perp})^{-1}\big\|_{\cL_{h}^{2}\rightarrow\cL_{h}^{2}}\big\|F_{T}^{\rm BCS}(\varphi_{\psi,0})\big\|_{\cH_{h}^{s}}\lesssim 
    \kappa^{-1} h^{3-s}\|\psi\|_{H^{2}(\T^{d})} \nonumber\\
    &= h^{13/6-s}\|\psi\|_{H^{2}(\T^{d})}.
\end{align*}
This proves \eqref{ortho-GL-regularity2}. 

Finally, the bound \eqref{eq:phipsisize} follows immediately from \eqref{eq:phipsi0} and \eqref{ortho-GL-regularity2}, thus completing the proof.
\end{proof}

\begin{proof}[Proof of Proposition \ref{prop:trial-property}]
    Using the definitions of $\alpha_\psi$ and $F_{T}^{\rm BCS}$, we obtain
    \begin{align}\label{trial2}
	   \varphi_{\psi}-V^{1/2}\alpha_\psi & = \varphi_\psi - V^{1/2} K_{T}V^{1/2}\varphi_{\psi}+ \tfrac{1}{2}N_{T}(-2V^{1/2}\varphi_{\psi})
       =F_{T}^{\rm BCS}(\varphi_{\psi}) \,.
    \end{align}
    Thus, to prove \eqref{trial-property}, we need to bound $F_{T}^{\rm BCS}(\varphi_{\psi})$. Let us write this quantity in a convenient way. Recalling the map $F_{T}^{\rm BCS,app}$ from \eqref{approx-BdG}, we claim that we have
\begin{align}\label{trial-property2}
	F_{T}^{\rm BCS}(\varphi_{\psi})
	&=P_{\kappa}F_{T}^{\rm BCS,app}(\varphi_{\psi,0})+P_{\kappa}L_{T}P^{\perp}\varphi_{\psi,1}+\tfrac{1}{2}P_{\kappa}V^{1/2}\widetilde{N}_{T}(-2V^{1/2}\varphi_{\psi})\nonumber\\
	&\quad+\tfrac{1}{2}P_{\kappa}V^{1/2}\big(N_{T}'(-2V^{1/2}\varphi_{\psi})-N_{T}'(-2V^{1/2}P_{\kappa}\varphi_{\psi,0})\big)\nonumber\\
	&\quad+\tfrac{1}{2}P_{\kappa}^{\perp}V^{1/2}\big(N_{T}(-2V^{1/2}\varphi_{\psi})-N_{T}(-2V^{1/2}\varphi_{\psi,0})\big) \,.
\end{align}
    Indeed, we have
    \begin{align*}
	F_{T}^{\rm BCS}(\varphi_{\psi})&=P_{\kappa}F_{T}^{\rm BCS}(\varphi_{\psi})+P_{\kappa}^{\perp}F_{T}^{\rm BCS}(\varphi_{\psi})\nonumber     \\
    &=P_{\kappa}L_{T}P_{\kappa}\varphi_{\psi}+\tfrac{1}{2}P_{\kappa}V^{1/2}N_{T}'(-2V^{1/2}P_{\kappa}\varphi_{\psi})+P_{\kappa}L_{T}P_{\kappa}^{\perp}\varphi_{\psi}\nonumber\\
	&\quad+\tfrac{1}{2}P_{\kappa}V^{1/2}\big(N_{T}'(-2V^{1/2}\varphi_{\psi})-N_{T}'(-2V^{1/2}P_{\kappa}\varphi_{\psi})\big)\nonumber\\
	&\quad+\tfrac{1}{2}P_{\kappa}V^{1/2}\widetilde{N}_{T}(-2V^{1/2}\varphi_{\psi})+P_{\kappa}^{\perp}\big(L_{T}\varphi_{\psi}+\tfrac{1}{2}V^{1/2}N_{T}(-2V^{1/2}\varphi_{\psi})\big) \,.
    \end{align*}   
    To see that the r.h.s.~coincides with the r.h.s.~of \eqref{trial-property2} we use the facts that $P_{\kappa}\varphi_{\psi,1}=0$ because $\varphi_{\psi,1}\in\Ran(P_{\kappa}^{\perp})$ and $P^{\perp}\varphi_{\psi,0}=0$ so that
\begin{align*}
	P_{\kappa}L_{T}P_{\kappa}^{\perp}\varphi_{\psi}=P_{\kappa}L_{T}P^{\perp}(\varphi_{\psi,0}+\varphi_{\psi,1})=P_{\kappa}L_{T}P^{\perp}\varphi_{\psi,1},
\end{align*}
 since $\lambda_{\kappa}$ commutes with $L_{T}$ and $P$, and we used the following computation for the term $P_{\kappa}^{\perp}L_{T}\varphi_{\psi}$:
\begin{align*}
	P_{\kappa}^{\perp}L_{T}\varphi_{\psi} &=P_{\kappa}^{\perp}L_{T}\varphi_{\psi,0}-P_{\kappa}^{\perp}F_{T}^{\rm BCS}(\varphi_{\psi,0})\nonumber\\
	&=P_{\kappa}^{\perp}L_{T}\varphi_{\psi,0}-P_{\kappa}^{\perp}\big(L_{T}\varphi_{\psi,0}+\tfrac{1}{2}V^{1/2}N_{T}(-2V^{1/2}\varphi_{\psi,0})\big)\nonumber\\
	&=-\tfrac{1}{2}P_{\kappa}^{\perp}V^{1/2}N_{T}(-2V^{1/2}\varphi_{\psi,0}).
\end{align*}
    This proves \eqref{trial-property2}. Thus, to prove \eqref{trial-property}, we need to show that the $\mathcal H^1_h$-norm of the r.h.s.~of \eqref{trial-property2} is bounded by a constant times $h^{19/6}$. We prove this for each term individually.

We begin with the difference of nonlinear terms in the second-to-last line of \eqref{trial-property2}. By Corollary \ref{cor:NL-H1-bdd}, Lemmas \ref{lem:Hsnorm-Delta}, \ref{lem:GL-H2-norm}, \ref{lem:small-H1-largeGL} and \ref{lem:ortho-GL-regularity}, we have
\begin{align}
	\big\|P_{\kappa}V^{1/2}&\big(N_{T}'(-2V^{1/2}\varphi_{\psi})-N_{T}'(-2V^{1/2}P_{\kappa}\varphi_{\psi,0})\big)\big\|_{\cH_{h}^{s}}\nonumber\\
    &\lesssim\big(\|\varphi_{\psi}\|_{\cH_{h}^{1}}^{2}+\|P_{\kappa}\varphi_{\psi,0}\|_{\cH_{h}^{1}}^{2}\big)\big\|\varphi_{\psi}-P_{\kappa}\varphi_{\psi,0}\big\|_{\cH_{h}^{1}}\nonumber\\
	&\lesssim\big(\|\varphi_{\psi,0}\|_{\cH_{h}^{1}}^{2}+\|\varphi_{\psi,1}\|_{\cH_{h}^{1}}^{2}\big)\big(\|P_{\kappa}^{\perp}\varphi_{\psi,0}\|_{\cH_{h}^{1}}+\|\varphi_{\psi,1}\|_{\cH_{h}^{1}}\big)\nonumber\\
	&\lesssim h^{2}\big(\|\Psi_{>}\|_{H_{h}^{1}(\T_{h}^{d})}+\|\varphi_{\psi,1}\|_{\cH_{h}^{1}}\big)\lesssim h^{2}(h^{19/12}+h^{7/6})\|\psi\|_{H^{2}(\T^{d})}\nonumber\\
    &\lesssim h^{19/6}\|\psi\|_{L^{2}(\T^{d})}.
\end{align}
Here we used $\|P_{\kappa}^{\perp}\varphi_{\psi,0}\|_{\cH_{h}^{1}} = \|\Psi_{>}\|_{H_{h}^{1}(\T_{h}^{d})}$, which is shown similarly as \eqref{eq:phipsinorm}.

Next, we turn to the difference of nonlinear terms in the last line of \eqref{trial-property2}. We divide it into leading ($N_T'$) and subleading ($\widetilde N_T$) terms as in \eqref{NL-leading} and \eqref{NL-subleading} and then estimate them separately.  For the leading terms, we argue in the same way as above to obtain
\begin{align}
	\big\|P_{\kappa}^{\perp}&V^{1/2}\big(N_{T}'(-2V^{1/2}\varphi_{\psi})-N_{T}'(-2V^{1/2}\varphi_{\psi,0})\big)\big\|_{\cL_{h}^{2}}\nonumber\\
	&\lesssim\big(\|\varphi_{\psi}\|_{\cH_{h}^{1}}^{2}+\|\varphi_{\psi,0}\|_{\cH_{h}^{1}}^{2}\big)\|\varphi_{\psi,1}\|_{\cH_{h}^{1}}\lesssim h^{19/6}\|\psi\|_{L^{2}(\T^{d})}.
\end{align}
For the subleading term, by Corollary \ref{cor:NL-H1-bdd} with $\varepsilon=1/4$ and Lemmas \ref{lem:Hsnorm-Delta}, \ref{lem:GL-H2-norm} and \ref{lem:ortho-GL-regularity}, we obtain
\begin{align}\label{trial-estimate-NL-sub}
	\big\|P_{\kappa}^{\perp}&V^{1/2}\big(\widetilde{N}_{T}(-2V^{1/2}\varphi_{\psi})-\widetilde{N}_{T}(-2V^{1/2}\varphi_{\psi,0})\big)\big\|_{\cL_{h}^{2}}\nonumber\\
	&\lesssim\big\|\widetilde{N}_{T}(-2V^{1/2}\varphi_{\psi})\big\|_{\cL_{h}^{2}}+\big\|\widetilde{N}_{T}(-2V^{1/2}\varphi_{\psi,0})\big\|_{\cL_{h}^{2}}\nonumber\\
	&\lesssim h^{-3/4}\big(\|\varphi_{\psi,0}\|_{\cH_{h}^{1}}^{4}+\|\varphi_{\psi,1}\|_{\cH_{h}^{1}}^{4}\big)\lesssim h^{13/4}\|\psi\|_{L^{2}(\T^{d})}.
\end{align}
Consequently, we obtain
\begin{align}
	&\big\|P_{\kappa}^{\perp}V^{1/2}\big(N_{T}(-2V^{1/2}\varphi_{\psi})-N_{T}(-2V^{1/2}\varphi_{\psi,0})\big)\big\|_{\cL_{h}^{2}}\nonumber\\
	&\leq \big\|P_{\kappa}^{\perp}V^{1/2}\big(N_{T}'(-2V^{1/2}\varphi_{\psi})-N_{T}'(-2V^{1/2}\varphi_{\psi,0})\big)\big\|_{\cL_{h}^{2}}\nonumber\\
	&\quad\quad+	\big\|P_{\kappa}^{\perp}V^{1/2}\big(\widetilde{N}_{T}(-2V^{1/2}\varphi_{\psi})-\widetilde{N}_{T}(-2V^{1/2}\varphi_{\psi,0})\big)\big\|_{\cL_{h}^{2}}\lesssim h^{19/6}\|\psi\|_{L^{2}(\T^{d})}.
\end{align}

Next, we consider the last two terms in the first line of the r.h.s.~of \eqref{trial-property2}.  The nonlinear term $P_{\kappa}V^{1/2}\widetilde{N}_{T}(-2V^{1/2}\varphi_{\psi})$ can be estimated using \eqref{trial-estimate-NL-sub} and thus yields
\begin{align}
	\big\|P_{\kappa}V^{1/2}\widetilde{N}_{T}(-2V^{1/2}\varphi_{\psi})\big\|_{\cL_{h}^{2}}&\lesssim h^{13/4}\|\psi\|_{L^{2}(\T^{d})}.
\end{align}
For the term $P_{\kappa}L_{T}P^{\perp}\varphi_{\psi,1}$, we recall
\begin{align*}
	\varphi_{\psi,1}&=-(P_{\kappa}^{\perp}L_{T}P_{\kappa}^{\perp})^{-1} P_\kappa^\perp F_{T}^{\rm BCS}(\varphi_{\psi,0}),
\end{align*}
so
\begin{align}\label{trial-estimate-L3a}
	\big\|P_{\kappa}L_{T}P^{\perp}\varphi_{\psi,1}\big\|_{\cL_{h}^{2}}&\leq \big\|P_{\kappa}L_{T}P^{\perp} (P_\kappa^{\perp}L_{T}P_{\kappa}^{\perp})^{-1} \big\|_{\cL_{h}^{2}\rightarrow\cL_{h}^{2}} \|F_{T}^{\rm BCS}(\varphi_{\psi,0})\big\|_{\cL_{h}^{2}} \,.
\end{align}
The second factor on the right side is $\lesssim h^3\|\psi\|_{L^{2}(\T^{d})}$ by Lemmas \ref{lem:GL-H2-norm} and \ref{lem:ortho-GL-regularity}. To bound the remaining operator norm, we recall that $P_\kappa = P \lambda_\kappa$ and therefore, since $\lambda_\kappa$ commutes with $L_T$, 
$$
P_{\kappa}L_{T}P^{\perp} (P_\kappa^{\perp}L_{T}P_{\kappa}^{\perp})^{-1} 
= P_{\kappa}L_{T} \lambda_\kappa P^{\perp} (P_\kappa^{\perp}L_{T}P_{\kappa}^{\perp})^{-1} \,.
$$
It follows that
\begin{align}
    \label{trial-estimate-L3b}
    \big\|P_{\kappa}L_{T}P^{\perp} (P_\kappa^{\perp}L_{T}P_{\kappa}^{\perp})^{-1} \big\|_{\cL_{h}^{2}\rightarrow\cL_{h}^{2}}
\leq \| P_\kappa L_T \|_{\cL_{h}^{2}\rightarrow\cL_{h}^{2}} 
\big\| \lambda_{\kappa}P^{\perp}(P_{\kappa}^{\perp}L_{T}P_{\kappa}^{\perp})^{-1} \big\|_{\cL_{h}^{2}\rightarrow\cL_{h}^{2}} \,.
\end{align}
Proposition \ref{prop:PLP-inv-bdd} yields
\begin{align*}
    \lambda_{\kappa}P^{\perp}(P_{\kappa}^{\perp}L_{T}P_{\kappa}^{\perp})^{-2}P^{\perp}\lambda_{\kappa}&\lesssim \kappa^{-1} P^\bot \lambda_{\kappa}\Big(\one+\kappa^{-1}\frac{-\Delta_{X}}{\one-\Delta_{X}}\Big)\lambda_\kappa P^\bot \lesssim\kappa^{-1} \,,
\end{align*}
so
\[
\big\| \lambda_{\kappa}P^{\perp}(P_{\kappa}^{\perp}L_{T}P_{\kappa}^{\perp})^{-1} \big\|_{\cL_{h}^{2}\rightarrow\cL_{h}^{2}} \lesssim \kappa^{-1/2} \,.
\]
Meanwhile, Lemma \ref{lem:Q0LP-bdd} gives
\[
\| P_\kappa L_T \|_{\cL_{h}^{2}\rightarrow\cL_{h}^{2}} 
= \| L_T P_\kappa \|_{\cL_{h}^{2}\rightarrow\cL_{h}^{2}}
\lesssim \kappa \,.
\]
Inserting the previous two bounds into \eqref{trial-estimate-L3b}, we find from \eqref{trial-estimate-L3a} that
\begin{align}\label{trial-estimate-L3}
	\big\|P_{\kappa}L_{T}P^{\perp}\varphi_{\psi,1}\big\|_{\cL_{h}^{2}}
    &\lesssim\kappa^{1/2}h^{3}\|\psi\|_{L^{2}(\T^{d})}\lesssim h^{41/12}\|\psi\|_{L^{2}(\T^{d})}.
\end{align}

Finally, we estimate the first term on the r.h.s.~of \eqref{trial-property2}.  Again, we decompose
\begin{align}\label{trial-final}
	P_{\kappa}F_{T}^{\rm BCS,app}(\varphi_{\psi,0})&=\lambda_{\kappa}F_{T}^{\rm GL,sc}(\Psi_{<})\varphi_{*}\nonumber\\
	&\quad+\lambda_{\kappa}\Big(PL_{T}P_{\kappa}\varphi_{\psi}-\big(-\nabla_{X}\cdot\Lambda_{0}\nabla_{X}-\Lambda_{2}Dh^{2}\big)\Psi_{<}\varphi_{*}\Big)\nonumber\\
	&\quad+\lambda_{\kappa}\Big(\tfrac{1}{2}PV^{1/2}N_{T}'(-2V^{1/2}P_{\kappa}\varphi_{\psi,0})-(\Lambda_{3}|\Psi_{<}|^{2}\Psi_{<})\varphi_{*}\Big).
\end{align}
Note that $P_{\kappa}\varphi_{\psi,0}=\Psi_{<}\varphi_{*}$.  Then the last terms can be bounded in the same way as in Propositions \ref{prop:GL-matrix} and \ref{prop:GL-nonlinear}, thus yielding
\begin{align}
	\label{GL-trial-linear}&\big\|\lambda_{\kappa}\big(PL_{T}P_{\kappa}\varphi_{\psi}-\big(-\nabla_{X}\cdot\Lambda_{0}\nabla_{X}-\Lambda_{2}Dh^{2}\big)\Psi_{<}\varphi_{*}\big)\big\|_{\cL_{h}^{2}}\lesssim h^{13/4}\|\psi\|_{L^{2}(\T^{d})},\\
	\label{GL-trial-nonlinear}&\Big\|\lambda_{\kappa}\Big(\frac{1}{2}PV^{1/2}N_{T}'(-2V^{1/2}P_{\kappa}\varphi_{\psi,0})-(\Lambda_{3}|\Psi_{<}|^{2}\Psi_{<})\varphi_{*}\Big)\Big\|_{\cL_{h}^{2}}\lesssim h^{41/12}\|\psi\|_{L^{2}(\T^{d})}.
\end{align}
Finally, the first term on the r.h.s.~of \eqref{trial-final} is estimated directly by Lemma \ref{lem:GL-eqn-Psi<} and gives
\begin{align*}
	\big\|\lambda_{\kappa}F_{T}^{\rm GL,sc}(\Psi_{<})\varphi_{*}\big\|_{\cL_{h}^{2}}=\big\|\lambda_{\kappa}F_{T}^{\rm GL,sc}(\Psi_{<})\big\|_{L_{h}^{2}(\T_{h}^{d})}\lesssim h^{43/12}\|\psi\|_{L^{2}(\T^{d})}.
\end{align*}
Putting everything together yields \eqref{trial-property}, which completes the proof.
\end{proof}

In addition to the bound on the $\cL_{h}^2$-norm of $\varphi_\psi-V^{1/2}\alpha_\psi$ in Proposition \ref{prop:trial-property}, we also need a rough bound on its $\cH_{h}^{1}$-norm. This bound is significantly less precise but is nevertheless sufficient for our purpose. Note that we do not use the GL equation.

\begin{lemma}\label{lem:regularity-trial}
Let $\psi\in H^2(\T^d)$. Then it holds that
\begin{align}\label{regularity-trial}
	\|\varphi_{\psi} - V^{1/2}\alpha_{\psi}\|_{\cH_{h}^{1}}&\lesssim h^{7/6}\|\psi\|_{H^{2}(\T^{d})} + h^{2} \left( \|\psi\|_{H^{2}(\T^{d})}^3 + \|\psi\|_{H^{1}(\T^{d})}^9\right).
\end{align}
\end{lemma} 

\begin{proof}
By \eqref{trial2} and the fact that $\ell_{T_{c}}\varphi_{\psi,0}=0$, we write
\begin{align}\label{varp-alpha-decomp}
	\varphi_{\psi}-V^{1/2}\alpha_{\psi}&=L_{T}\varphi_{\psi}+ \tfrac{1}{2}V^{1/2}N_{T}(-2V^{1/2}\varphi_{\psi})\nonumber\\
	&=(L_{T}-\ell_{T_{c}})\varphi_{\psi,0}+L_{T}\varphi_{\psi,1}+ \tfrac{1}{2}V^{1/2}N_{T}(-2V^{1/2}\varphi_{\psi}).
\end{align}
We will bound each term on the r.h.s of \eqref{varp-alpha-decomp}.

For the first term in \eqref{varp-alpha-decomp}, we use \eqref{eq:trialroughlinear} with $s=1$ and see that its $\cH_{h}^1$-norm is $\lesssim h^2\|\psi\|_{H^{1}(\T^{d})}$.

Next, since $K_{T}$ is bounded and commutes with $-\Delta_{X}$, by Lemma \ref{lem:ortho-GL-regularity}, the second term in \eqref{varp-alpha-decomp} is bounded as follows:
\begin{align}
	\big\|L_{T}\varphi_{\psi,1}\big\|_{\cH_{h}^{1}}&\leq \|L_{T}\|_{\cL_{h}^{2}\rightarrow\cL_{h}^{2}}\|\varphi_{\psi,1}\|_{\cH_{h}^{1}}\lesssim h^{7/6}\|\psi\|_{H^{2}(\T^{d})}.
\end{align}

Now, by Proposition \ref{prop:BCS-NL-norm}, Lemmas \ref{lem:Hsnorm-Delta}, \ref{lem:Lpnorm-Delta} and \eqref{eq:phipsisize}, we estimate the last term in \eqref{varp-alpha-decomp} as follows:
\begin{align}
	\big\|V^{1/2}N_{T}(-2V^{1/2}\varphi_{\psi})\big\|_{\cH_{h}^{1}} 
    & \lesssim h^{-1}\|V^{1/2}\varphi_{\psi}\|_{\cL_{h}^{6}}^{3}\lesssim h^{-1}\|\varphi_{\psi}\|_{\cH_{h}^{1}}^{3} \nonumber \\
    & \lesssim h^{2} \left( \|\psi\|_{H^{2}(\T^{d})}^3 + \|\psi\|_{H^{1}(\T^{d})}^9\right).
\end{align}
Putting everything together yields \eqref{regularity-trial}.
\end{proof}


\subsection{Completing the proof of Theorem \ref{thm:deriv-BdG}}\label{sec:pf-deriv-BdG2}

Finally, we finish the proof of Theorem \ref{thm:deriv-BdG}. Let $\psi\in H^1(\T^d)$ be a weak solution of the GL equation \eqref{GL-eqn} and construct $\alpha_{\psi}$ according to \eqref{trial}.

By the definition of $\alpha_{\psi}$, we obtain 
\begin{align}\label{final-deriv-BdG}
	\alpha_{\psi}&+\frac{1}{2}\Big[\tanh\Big(\frac{\beta H(-2V\alpha_{\psi})}{2}\Big)\Big]_{12}=\alpha_{\psi}-K_{T}V\alpha_{\psi}+\tfrac{1}{2}N_{T}(-2V\alpha_{\psi})\nonumber\\
	&=K_{T}V^{1/2}\big(\varphi_{\psi}-V^{1/2}\alpha_{\psi}\big) - \tfrac{1}{2} \left( N_{T}(-2V^{1/2}\varphi_{\psi}) - N_{T}(-2V\alpha_{\psi})\right).
\end{align}
The r.h.s.~has a linear and a nonlinear contribution, which we bound separately.

By the boundedness of operator $K_{T}V^{1/2}$ and Proposition \ref{prop:trial-property}, the linear term gives
\begin{align}\label{final-deriv-BdG-L}
	\big\|K_{T}V^{1/2}\big(\varphi_{\psi}-V^{1/2}\alpha_{\psi}\big)\big\|_{\cL_{h}^{2}}&\lesssim\big\|\varphi_{\psi}-V^{1/2}\alpha_{\psi}\big\|_{\cL_{h}^{2}}\lesssim h^{19/6}\|\psi\|_{L^{2}(\T^{d})}.
\end{align}

To estimate the nonlinear terms in \eqref{final-deriv-BdG}, we decompose them into leading and subleading terms as usual.  For the leading part, by Corollary \ref{cor:NL-H1-bdd}, \eqref{eq:phipsisize}, Lemmas \ref{lem:GL-H2-norm} and \ref{lem:regularity-trial}, we have
\begin{align}
	\big\|N_{T}'(-2V^{1/2}\varphi_{\psi})& -N_{T}'(-2V\alpha_{\psi}) \big\|_{\cL_{h}^{2}} \nonumber\\
    &\lesssim \big(\|\varphi_{\psi}\|_{\cH_{h}^{1}}^{2} + \|V^{1/2}\alpha_{\psi}\|_{\cH_{h}^{1}}^{2} \big) \big\|\varphi_{\psi} - V^{1/2}\alpha_{\psi}\big\|_{\cH_{h}^{1}} \nonumber\\
	&\lesssim \big(\|\varphi_{\psi}\|_{\cH_{h}^{1}}^{2} + \|V^{1/2}\alpha_{\psi}-\varphi_{\psi}\|_{\cH_{h}^{1}}^{2} \big) \big\|\varphi_{\psi} - V^{1/2}\alpha_{\psi}\big\|_{\cH_{h}^{1}} \nonumber \\
    &\lesssim h^{19/6}\|\psi\|_{L^{2}(\T^{d})}.
\end{align}
Similarly, for the subleading part, by Corollary \ref{cor:NL-H1-bdd} with $\varepsilon=1/4$, Lemmas \ref{lem:GL-H2-norm} and \ref{lem:regularity-trial}, we have
\begin{align}
	\big\|\widetilde{N}_{T}(-2V^{1/2}\varphi_{\psi})&-\widetilde{N}_{T}(-2V\alpha_{\psi})\big\|_{\cL_{h}^{2}}
    \leq \big\|\widetilde{N}_{T}(-2V^{1/2}\varphi_{\psi})\big\|_{\cL_{h}^{2}}
    + \big\|\widetilde{N}_{T}(-2V\alpha_{\psi})\big\|_{\cL_{h}^{2}} \nonumber \\
    & \lesssim h^{-3/4}\big(\|\varphi_{\psi}\|_{\cH_{h}^{1}}^{4}+\|V^{1/2}\alpha_{\psi}\|_{\cH_{h}^{1}}^{4}\big)\nonumber\\
	&\lesssim h^{-3/4}\big(\|\varphi_{\psi}\|_{\cH_{h}^{1}}^{4}+\|V^{1/2}\alpha_{\psi}-\varphi_{\psi}\|_{\cH_{h}^{1}}^{4}\big)\lesssim h^{13/4}\|\psi\|_{L^{2}(\T^{d})}.
\end{align}
By combining these estimates with \eqref{final-deriv-BdG-L}, we obtain \eqref{BdG-approx}.

\medskip

It remains to prove \eqref{Valpha-trial-rem}. We define $\xi_\psi$ by \eqref{Valpha-trial-decomp}. To simplify notations, we introduce $\varrho_{\psi},\eta_{\psi}\in\cL_{h}^{2}$ as
\begin{align}\label{rho-psi}
	\varrho_{\psi}&:=V^{1/2}\alpha_{\psi}-\varphi_{\psi},
\end{align}
and
\begin{align}\label{eta-psi}
	\eta_{\psi}&:=\alpha_{\psi}+\frac{1}{2}\Big[\tanh\Big(\frac{\beta H(-2V\alpha_{\psi})}{2}\Big)\Big]_{12}=(\one-K_{T}V)\alpha_{\psi}+\frac{1}{2}N_{T}(-2V\alpha_{\psi}),
\end{align}
which, thanks to Proposition \ref{prop:trial-property} and \eqref{BdG-approx}, satisfy the estimates
\begin{align*}
	\|\varrho_{\psi}\|_{\cH_{h}^{1}}&\lesssim h^{7/6}\|\psi\|_{L^{2}(\T^{d})},\quad\|\eta_{\psi}\|_{\cL_{h}^{2}}\lesssim h^{19/6}\|\psi\|_{L^{2}(\T^{d})}.
\end{align*}
In the next lemma, we prove a rough bound on the $\cH_{h}^{1}$-norm of $\eta_{\psi}$, which, as in Lemma \ref{lem:regularity-trial}, is much less precise compared to the corresponding $\cL_{h}^{2}$-norm in \eqref{BdG-approx}, but is sufficient for our purpose.

\begin{lemma}\label{lem:H1-bdd-eta}
If $\psi\in H^{2}(\T^{d})$, then it holds that
\begin{align}\label{H1-bdd-eta}
	\|\eta_{\psi}\|_{\cH_{h}^{1}} \lesssim h^{7/6}\|\psi\|_{H^{2}(\T^{d})} & + h^2 \left( \|\psi\|_{H^2(\T^d)}^3 + \|\psi\|_{H^1(\T^d)}^9 \right) \nonumber \\
    & + h^{5} \left( \|\psi\|_{H^2(\T^d)}^9 + \|\psi\|_{H^1(\T^d)}^{27} \right).
\end{align}
\end{lemma}

We emphasize that in our application, $\psi$ will be a solution of the GL equation and therefore $\|\psi\|_{H^2(\T^d)}\lesssim 1$ by Lemma \ref{lem:GL-H2-norm}. Therefore, only the first term on the right side of \eqref{H1-bdd-eta} will be relevant.

\begin{proof}
By \eqref{trial} and \eqref{eta-psi}, we can write
\begin{align}\label{eta-psi2}
	\eta_{\psi}&=(\one-K_{T}V)\alpha_{\psi}+\frac{1}{2}N_{T}(-2V\alpha_{\psi})\nonumber\\
	&=K_{T}V^{1/2}\big(\varphi_{\psi}-V^{1/2}\alpha_{\psi}\big)+\frac{1}{2}N_{T}(-2V\alpha_{\psi})-\frac{1}{2}N_{T}(-2V^{1/2}\varphi_{\psi}).
\end{align}
Since $K_{T}V^{1/2}$ commutes with $-\Delta_{X}$ and is bounded as an operator on $\cL_{h}^{2}$, by Lemma \ref{lem:regularity-trial}, the linear part of \eqref{eta-psi2} yields 
\begin{align*}
	\big\|K_{T}V^{1/2}\big(\varphi_{\psi}-V^{1/2}\alpha_{\psi}\big)\big\|_{\cH_{h}^{1}}&\lesssim\big\|K_{T}V^{1/2}\big\|_{\cL_{h}^{2}\rightarrow\cL_{h}^{2}}\big\|\varphi_{\psi}-V^{1/2}\alpha_{\psi}\big\|_{\cH_{h}^{1}}\lesssim h^{7/6}\|\psi\|_{H^{2}(\T^{d})}.
\end{align*}

Next, for the nonlinear part in \eqref{eta-psi2}, we use the triangle inequality, Proposition \ref{prop:BCS-NL-norm}, Lemmas \ref{lem:Lpnorm-Delta} and \ref{lem:regularity-trial} to obtain 
\begin{align*}
	&\big\|N_{T}(-2V\alpha_{\psi})-N_{T}(-2V^{1/2}\varphi_{\psi})\big\|_{\cH_{h}^{1}}\nonumber\\
	&\lesssim h^{-1}\big(\|V\alpha_{\psi}\|_{\cL_{h}^{6}}^{3}+\|V^{1/2}\varphi_{\psi}\|_{\cL_{h}^{6}}^{3}\big)\lesssim h^{-1}\big(\|V^{1/2}\alpha_{\psi}\|_{\cH_{h}^{1}}^{3}+\|\varphi_{\psi}\|_{\cH_{h}^{1}}^{3}\big)\nonumber\\
	&\lesssim h^{-1}\big(\|\varphi_{\psi}\|_{\cH_{h}^{1}}^{3}+\|\varphi_{\psi}-V^{1/2}\alpha_{\psi}\|_{\cH_{h}^{1}}^{3}\big) \\
    & \lesssim 
    h^2 \left( \|\psi\|_{H^2(\T^d)}^3 + \|\psi\|_{H^1(\T^d)}^9 \right) + h^{5} \left( \|\psi\|_{H^2(\T^d)}^9 + \|\psi\|_{H^1(\T^d)}^{27} \right).
\end{align*}
Putting everything together yields \eqref{H1-bdd-eta}.
\end{proof}

Next, we note that one can write
\begin{align*}
	\alpha_{\psi}&=K_{T}V\alpha_{\psi}-\frac{1}{2}N_{T}(-2V\alpha_{\psi})+\eta_{\psi}\nonumber\\
	&=K_{T}V^{1/2}\big(\varphi_{\psi}+\varrho_{\psi}\big)-\frac{1}{2}N_{T}(-2V^{1/2}(\varphi_{\psi}+\varrho_{\psi}))+\eta_{\psi}\nonumber\\
	&=K_{T}V^{1/2}\varphi_{\psi,0}+K_{T}V^{1/2}\big(\varphi_{\psi,1}+\varrho_{\psi}\big)-\frac{1}{2}N_{T}(-2V^{1/2}(\varphi_{\psi}+\varrho_{\psi}))+\eta_{\psi}.
\end{align*}
We observe that, since $\alpha_{*}=k_{T_{c}}V\alpha_{*}=k_{T_{c}}V^{1/2}\varphi_{*}$ and 
\begin{align*}
	h \, \psi(h(x+y)/2) \, \alpha_{*}(x-y)&=\Psi(X) \, \alpha_{*}(r)=\alpha_{*}(r)\int_{\R^{d}}\varphi_{*}(s) \,\varphi_{\psi,0}(\zeta_{X}^{s},\zeta_{X}^{-s}) \, ds \nonumber\\
	&=(k_{T_{c}}V^{1/2}\varphi_{*})(r)\int_{\R^{d}}\varphi_{*}(s)\varphi_{\psi,0}(\zeta_{X}^{s},\zeta_{X}^{-s}) \, ds \nonumber\\
	&=\big(k_{T_{c}}V^{1/2}P\varphi_{\psi,0}\big)(\zeta_{X}^{r},\zeta_{X}^{-r}).
\end{align*}
Thus, using the same arguments as in \eqref{xi}--\eqref{xi2} and the fact that $P\varphi_{\psi,0}=\varphi_{\psi,0}$, we can write
\begin{align}\label{xi-psi}
	\xi_{\psi}&=(K_{T}-k_{T_{c}})V^{1/2}\varphi_{\psi,0}+K_{T}V^{1/2}\big(\varphi_{\psi,1}+\varrho_{\psi}\big)\nonumber\\
	&\quad\quad\quad\quad\quad\quad\quad-\frac{1}{2}N_{T}(-2V^{1/2}(\varphi_{\psi}+\varrho_{\psi}))+\eta_{\psi}.
\end{align}

Now, we estimate the $\cL_{h}^{2}$-norm of each term on the r.h.s. of \eqref{xi-psi}.  By Lemma \ref{lem:H1-bdd-eta}, we already know that $\|\eta_{\psi}\|_{\cH_{h}^{1}}\lesssim h^{7/6}\|\psi\|_{L^{2}(\T^{d})}$.  For the nonlinear term in the second line of \eqref{xi-psi}, by Proposition \ref{prop:BCS-NL-norm}, Lemmas \ref{lem:ortho-GL-regularity} and \ref{lem:regularity-trial}, we obtain
\begin{align*}
	\big\|N_{T}(-2V^{1/2}(\varphi_{\psi}+\varrho_{\psi}))\big\|_{\cH_{h}^{1}}&\lesssim h^{-1}\big(\|V^{1/2}\varphi_{\psi}\|_{\cL_{h}^{6}}^{3}+\|V^{1/2}\varrho_{\psi}\|_{\cL_{h}^{6}}^{3}\big)\nonumber\\
	&\lesssim h^{-1}\big(\|\varphi_{\psi}\|_{\cH_{h}^{1}}^{3}+\|\varrho_{\psi}\|_{\cH_{h}^{1}}^{3}\big)\lesssim h^{-1}\big(h^{3}+h^{7/2}\big)\|\psi\|_{L^{2}(\T^{d})}\nonumber\\
    &\lesssim h^{2}\|\psi\|_{L^{2}(\T^{d})}.
\end{align*}
Similarly, for the second term in the first line of \eqref{xi-psi}, using the boundedness of $K_{T}V^{1/2}$ as an operator on $\cL_{h}^{2}$ and the fact that it commutes with $-\Delta_{X}$, we obtain from Lemmas \ref{lem:ortho-GL-regularity} and \ref{lem:regularity-trial} that
\begin{align*}
	\big\|K_{T}V^{1/2}(\varphi_{\psi,1}+\varrho_{\psi})\big\|_{\cH_{h}^{1}}&\lesssim \|K_{T}V^{1/2}\|_{\cL_{h}^{2}\rightarrow\cL_{h}^{2}}\big(\|\varphi_{\psi,1}\|_{\cH_{h}^{1}}+\|\varrho_{\psi}\|_{\cH_{h}^{1}}\big)\lesssim h^{7/6}\|\psi\|_{L^{2}(\T^{d})}.
\end{align*}

Finally, for the first term in the first line of \eqref{xi-psi}, we use Lemmas \ref{lem:KTkT-diff}, \ref{lem:kT-diff}, \ref{lem:GL-H2-norm} and Assumption (A5) to obtain
\begin{align*}
	&\big\|\big(K_{T}-k_{T_{c}}\big)V^{1/2}\varphi_{\psi,0}\big\|_{\cH_{h}^{1}}\lesssim\big\|(K_{T}-k_{T})V^{1/2}\varphi_{\psi,0}\big\|_{\cH_{h}^{1}}+\big\|(k_{T}-k_{T_{c}})V^{1/2}\varphi_{\psi,0}\big\|_{\cH_{h}^{1}}\nonumber\\
	&=\big\|(K_{T}-k_{T})V^{1/2}(\one-h^{-2}\Delta_{X})^{1/2}\varphi_{\psi,0}\big\|_{\cL_{h}^{2}} \nonumber \\
    & \quad +\big\|(k_{T}-k_{T_{c}})V^{1/2}(\one-h^{-2}\Delta_{X})^{1/2}\varphi_{\psi,0}\big\|_{\cL_{h}^{2}}\nonumber\\
	&\lesssim\Big\|\frac{\Delta_{X}}{\one-\Delta_{X}}(\one-h^{-2}\Delta_{X})^{1/2}\varphi_{\psi,0}\Big\|_{\cL_{h}^{2}}+(\beta-\beta_{c})\big\|(\one-h^{-2}\Delta_{X})^{1/2}\varphi_{\psi,0}\big\|_{\cL_{h}^{2}}\nonumber\\
	&\lesssim\Big\|\frac{(\one-h^{-2}\Delta_{X})^{1/2}}{\one-\Delta_{X}}\Big\| \big\|\Delta_{X}\Psi\big\|_{L_{h}^{2}(\T_{h}^{d})}+h^{2}\|\Psi\|_{H_{h}^{1}(\T_{h}^{d})}\nonumber\\
	&\lesssim h^2 \|\Delta\psi\|_{L^2(\T^d)} +h^{3}\|\psi\|_{L^{2}(\T^{d})}.\nonumber
\end{align*}
Here we also used the fact that all of $V,K_{T},k_{T}$ and $k_{T_{c}}$ commute with $-\Delta_{X}$, and, together with an application of spectral theory,
\begin{align*}
	\Big\|\frac{(\one-h^{-2}\Delta_{X})^{1/2}}{\one-\Delta_{X}}\Big\|&\leq\sup_{q\in 2\pi\Z^{d}}\frac{(\one+|q|^{2})^{1/2}}{\one+h^2 |q|^{2}}\lesssim h^{-1}.
\end{align*}
By Lemma \ref{lem:GL-H2-norm}, we obtain the bound $h^2 \|\psi\|_{L^2(\T^d)}$ for the first term in the first line of \eqref{xi-psi}.

Consequently, by putting everything together, we obtain \eqref{Valpha-trial-rem}, which completes the proof of Theorem \ref{thm:deriv-BdG}.$\hfill\qed$


\appendix

\section{Derivation of the BdG equation}\label{app:derbdg}

In this appendix we discuss the relation between the BCS functional and the BdG equation, namely we show that the BdG equation is satisfied by critical points of the BCS functional; see also \cite[Theorem 2.2]{ChenSigal-20}.

The BCS energy functional at inverse temperature $\beta=T^{-1}>0$ is
\begin{align}\label{BCS-energy}
	\cF_{T}^{\rm BCS}(\gamma,\alpha)&:=\Tr_{\T^d_{h}}[\fh\gamma]-\frac{1}{\beta}\Tr_{\T^d_{h}}[S(\Gamma_{(\gamma,\alpha)})]-\int_{\R}dx\int_{\T^d_{h}}dy \, V(x-y)|\alpha(x,y)|^{2},
\end{align}
where, as before, $\Tr_{\T^d_{h}}$ is the trace per unit volume, where $\fh=-\Delta-\mu$ and where 
$$
S(\Gamma):=-\frac{1}{2}\big(\Gamma\ln\Gamma+(1-\Gamma)\ln(1-\Gamma)\big)
$$ 
is the entropy operator. We call a pair $(\gamma,\alpha)$ \emph{admissible} if it satisfies \eqref{BdG-1pdm}, if $\gamma$ and $\alpha$ are $h^{-1}\Z^d$-periodic and if $\Tr_{\T^d_h}(-\Delta +1)^{1/2}\gamma(-\Delta+1)^{1/2}<\infty$.

For background on this functional and admissible states, we refer to \cite{BacLieSol-94,HainHaSeriSolov-08,HainzlSeiringer-16}.

Let $(\gamma,\alpha)$ be admissible, assume that $\ker\Gamma$ and $\ker(1-\Gamma)$ are trivial and assume that $(\gamma,\alpha)$ is a critical point of $\cF_T^{\rm BCS}$ in the sense that for any admissible pair $(\tilde\gamma,\tilde\alpha)$ we have
$$
\frac{d}{dt}\Big|_{t=0} \cF_T^{\rm BCS}((1-t)\gamma+t\tilde\gamma,(1-t)\alpha+t\tilde\alpha) = 0 \,.
$$
(Note that $((1-t)\gamma+t\tilde\gamma,(1-t)\alpha+t\tilde\alpha)$ is admissible.) Computing the derivative, we see that we have
\begin{align*}
    & \Tr_{\T^d_{h}}[\fh(\tilde\gamma-\gamma)]+\frac{1}{2\beta}\Tr_{\T^d_{h}}(\Gamma_{(\tilde\gamma,\tilde\alpha)} - \Gamma_{(\gamma,\alpha)}) \ln\frac{\Gamma_{(\gamma,\alpha)}}{1-\Gamma_{(\gamma,\alpha)}}\nonumber\\
	&\quad\quad\quad\quad\quad-2\Re\int_{\R}dx\int_{\T^d_{h}}dy \, V(x-y)\overline{\alpha(x,y)}(\tilde\alpha(x,y)-\alpha(x,y)) = 0 \,.
\end{align*}
Since
\begin{align*}
    & \Tr_{\T^d_{h}}[\fh(\tilde\gamma-\gamma)]-2\Re\int_{\R}dx\int_{\T^d_{h}}dy \, V(x-y)\overline{\alpha(x,y)}(\tilde\alpha(x,y)-\alpha(x,y)) \\
    & = \frac12\Tr_{\T^d_{h}}H(-2V\alpha)(\Gamma_{(\tilde\gamma,\tilde\alpha)} - \Gamma_{(\gamma,\alpha)})      
\end{align*}
and since $(\tilde\gamma,\tilde\alpha)$ are arbitrary, we deduce that
\begin{align}\label{BdG-eqn}
H(-2V\alpha) + \frac1\beta \ln\frac{\Gamma_{(\gamma,\alpha)}}{1-\Gamma_{(\gamma,\alpha)}} = 0 \,.
\end{align}
We can then solve \eqref{BdG-eqn} for $\Gamma_{\gamma,\alpha}$ and obtain
\begin{align}\label{BdG-eqn2}
	\Gamma_{(\gamma,\alpha)}&=\frac{1}{1+e^{\beta H(-2V\alpha)}}=\frac{1}{2}-\frac{1}{2}\tanh\Big(\frac{\beta H(-2V\alpha)}{2}\Big),
\end{align}
where we have used the identities
\begin{align*}
	\frac{1}{1+e^{x}}+\frac{1}{1+e^{-x}}=1,\quad\frac{1}{1+e^{x}}-\frac{1}{1+e^{-x}}=-\tanh(x/2).
\end{align*}
We observe that the r.h.s.~of \eqref{BdG-eqn2} depends only on $\alpha$, so that $\gamma$ can be seen as a function of $\alpha$, i.e.,
\begin{align*}
	\gamma&\equiv \gamma(-2V\alpha)=\Big[\frac{1}{1+e^{\beta H(-2V\alpha)}}\Big]_{11},
\end{align*}
where $[\cdot]_{ij}$ denotes the $ij$-entry of the operator in its block decomposition.  Hence, it suffices to consider the following nonlinear equation for $\alpha$:
\begin{align}\label{BdG-eqn3}
	\alpha&=\Big[\frac{1}{1+e^{\beta H(-2V\alpha)}}\Big]_{12}=-\frac{1}{2}\Big[\tanh\Big(\frac{\beta H(-2V\alpha)}{2}\Big)\Big]_{12}.
\end{align}
This is precisely equation \eqref{BdG-eqn3a}.



\begin{thebibliography}{17}
	
	\bibitem{BacLieSol-94} 
	V. Bach, E. H. Lieb and J. P. Solovej,  
	``Generalized Hartree--Fock theory and the Hubbard model",  
	J. Stat. Phys. {\bf 76} (1994), no. 1-2, 3--89.
	
	\bibitem{BCS-57}
	J. Bardeen, L. Cooper, J. Schrieﬀer, 
	``Theory of superconductivity," 
	\textit{Phys. Rev.} \textbf{108} (1957), 1175--1204.
	
	
	
	\bibitem{ChenSigal-20}
	L. Chen, I. M. Sigal.
	``Vortex lattices and the Bogoliubov--de Gennes equations,"
	\textit{Adv. Math.} \textbf{380} (2021), 107546.
	
	\bibitem{DeuHainzlMaier-22} 
    A. Deuchert, C. Hainzl, M. Maier,
	``Microscopic derivation of Ginzburg--Landau theory and the BCS critical temperature shift in a weak homogeneous magnetic field,"
    \textit{Probab. Math. Phys.} \textbf{4} (2023), no. 1, 1--89.

	\bibitem{DeuHainzlMaier-23}
	A. Deuchert, C. Hainzl, M. Maier,
	``Microscopic derivation of Ginzburg--Landau theory and the BCS critical temperature shift in general external fields,"
	\textit{Calc. Var.} \textbf{62}(203), (2023), 1--81.
	
	\bibitem{FrankHainzl-18}
	R. L. Frank, C. Hainzl,
	``The BCS Critical Temperature in a Weak External Electric Field via a Linear Two-Body Operator,"
    In: Cadamuro, D., Duell, M., Dybalski, W., Simonella, S. (eds) Macroscopic Limits of Quantum Systems. MaLiQS 2017. Springer Proceedings in Mathematics \& Statistics, vol 270. Springer, Cham. 

    \bibitem{FHL-19}
    R. L. Frank, C. Hainzl, E. Langmann, ``The BCS critical temperature in a weak homogeneous magnetic field," \textit{J. Spectr.} Theory \textbf{9} (2019), no. 3, 1005--1062.
    
	\bibitem{FHSS-12}
	R. L. Frank, C. Hainzl, R. Seiringer, J. P. Solovej,
	``Microscopic derivation of Ginzburg--Landau theory,"
	\textit{J. Amer. Math. Soc.} \textbf{25} (2012), 667--713.
	
	
	\bibitem{FHSS-16}
	R. L. Frank, C. Hainzl, R. Seiringer, J. P. Solovej,
	``The external field dependence of the BCS critical temperature,"
	\textit{Commun. Math. Phys.} \textbf{342} (2016), 189--216.
	
	\bibitem{deGennes-66}
	P. G. de Gennes,
	\textit{Superconductivity of Metals and Alloys, }
	Westview Press (1966).
	
	\bibitem{GinzLandau-50}
	V. L. Ginzburg, L. D. Landau, 
	``On the theory of superconductivity,"
	\textit{Zh. Eksp. Teor. Fiz.} \textbf{20} (1950), 1064--1082.
	
	\bibitem{Gorkov-59}
	L. P. Gor'kov, 
	``Microscopic derivation of the Ginzburg--Landau equations in the theory of superconductivity,"
	\textit{Zh. Eksp. Teor. Fiz.} \textbf{36} (1959), 1918--1923; English translation \textit{Soviet Phys. JETP} \textbf{9} (1959), 1364--1367.
	
	\bibitem{HainHaSeriSolov-08}
	C. Hainzl, E. Hamza, R. Seiringer, J. P. Solovej, 
	``The BCS functional for general pair interactions,"
	\textit{Commun. Math. Phys.} \textbf{281} (2008), 349--367.
%
%

    \bibitem{HainzlSeiringer-16}
    C. Hainzl, R. Seiringer,
    ``The BCS functional of superconductivity and its mathematical properties," \textit{J. Math. Phys.} \textbf{57} (2016), 021101.

	\bibitem{Hebey}
	E. Hebey,
	``\textit{Nonlinear Analysis on Manifolds: Sobolev Spaces and Inequalities},"
	American Mathematical Society (1999).
		
	\bibitem{LanLifs-81}
	L. D. Landau, E. M. Lifshitz, L. P. Pitaevskii,
	\textit{Course of Theoretical Physics, Volume 9: Statistical Physics, Part 2 (1st ed)}. Pergamon Press (1980).
	
\end{thebibliography}
\end{document}